\newif\ifLTEX
\LTEXfalse

\documentclass[a4paper]{amsart}

\usepackage{mathtools,amssymb}
\usepackage[abbrev]{amsrefs}
\usepackage{xparse}

\usepackage{graphicx}
 
\usepackage{hyperref}
\usepackage{subcaption}
\usepackage{enumitem}
\usepackage{amsthm}
\usepackage{pdfpages}
\usepackage[all]{xy} 
\usepackage{float}
\usepackage{marginfix}
\usepackage{ifthen}

\newcounter{signcmdcounter}

\newtheorem{thm}{Theorem}[section]
\newtheorem{lem}[thm]{Lemma}

\newtheorem{cor}[thm]{Corollary}
\newtheorem{prop}[thm]{Proposition}

\theoremstyle{definition}

\newtheorem{conv}[thm]{Convention}

\newcommand{\C}{\mathbb{C}}

\newcommand{\Z}{\mathbb{Z}}

\newcommand{\D}{\mathbb{D}}
\newcommand{\Hy}{\mathbb{H}^2}

\newcommand{\M}{\mathcal{M}}
\newcommand{\F}{\mathcal{F}}

\newcommand{\DC}{DC~presentation}
\newcommand{\refline}{reflection curve}%pointwise fixed simple closed curve
\newcommand{\equiva}{topologically conjugate }%$\theta $に関する同値
\newcommand{\topconj}{topologically conjugate }%対合に関する共役

\newcommand{\pmsign}[1]{%
    \ifthenelse{\equal{#1}{p}}{%
        +%
    }{\ifthenelse{\equal{#1}{m}}{%
        -%
    }{%
    \pm%
}}}
\newcommand{\repeatsyms}[2]{%
    \ifthenelse{\equal{#1}{0}}{%
        -%
    }{\ifthenelse{\equal{#1}{1}}{%
        #2%
    }{\ifthenelse{#1 > 1}{%
        \setcounter{signcmdcounter}{#1}
        \whiledo{\value{signcmdcounter} > 1}{%
            #2, %
            \addtocounter{signcmdcounter}{-1}
        }%
        #2%
    }{%
}}}}
\newcommand{\NEC}[4]{(#1, \pmsign{#2}, [\repeatsyms{#3}{2}], \{ \repeatsyms{#4}{(-)} \})}

\allowdisplaybreaks
\title[DC presentations for involutions]{Dehn twist--crosscap slide presentations for involutions on non-orientable surfaces of genera up to 5}

\author[G.~Omori]{Genki Omori}
\address{
(Genki Omori)
Department of Mathematics, Faculty of Science and Technology, Tokyo University of Science, 2641 Yamazaki, Noda-shi, Chiba, 278-8510 Japan
}
\email{omori\_genki@rs.tus.ac.jp}

\author[N.~Sakata]{Naoki Sakata}
\address{
(Naoki Sakata)
Department of Physics, Ochanomizu University, 2-1-1, Otsuka, Bunkyo-ku, Tokyo, 112-8610 Japan
}
\email{sakata@casis.sakura.ne.jp}


\date{\today}

\begin{document}
\maketitle
\begin{abstract}
We give Dehn twist--crosscap slide presentations for involutions on non-orientable surfaces of genera up to 5.
\end{abstract}

\section{Introduction}

\subsection{Background}

Let $S$ be a connected closed surface.  
A self-homeomorphism $\varphi $ on $S$ is a \textit{periodic map} on $S$ if there exists a positive integer $n$ such that $\varphi ^n=\textrm{id}_S$.   
In particular, when $n=2$ in the definition above, we also call $\varphi $ an \textit{involution} on $S$.
For self-homeomorphisms $\varphi $ on $S$ and $\varphi ^\prime $ on a surface $S^\prime $, $\varphi $ is \textit{\topconj }to $\varphi ^\prime $ if there exists a homeomorphism $F\colon S\to S^\prime $ such that $\varphi ^\prime =F\circ \varphi \circ F^{-1}$, namely, the following diagram commutes: 
\[
\xymatrix{
S \ar[r]^{\varphi } \ar[d]_F &  S \ar[d]^F \\
S^\prime  \ar[r]_{\varphi ^\prime } &S^\prime .\ar@{}[lu]|{\circlearrowright}
}
\]

In the case where $S$ is an orientable surface, Dehn~\cite{Dehn} showed that any orientation-preserving self-homeomorphism $\varphi $ on $S$ is isotopic to a product of Dehn twists. 
We call such a product a \textit{Dehn twist presentation} for $\varphi $. 
Since the conjugation of a Dehn twist is also a Dehn twist, one Dehn twist presentation for a self-homeomorphism $\varphi $ on $S$ gives a Dehn twist presentation for any homeomorphism which is \topconj to $\varphi $. 
Nielsen~\cites{Nielsen, Nielsen2} classified topological conjugacy classes of orientation-preserving periodic maps on oriented surfaces and proved that any orientation-preserving periodic map on an oriented surface $S$ with negative Euler characteristic is realized as an isometry for some hyperbolic metric on $S$. 

Let $\Sigma _{g}$ be a connected closed oriented surface of genus $g\geq 0$. 
Birman-Hilden~\cite{Birman-Hilden} gave a specific Dehn twist presentation for the hyperelliptic involution on $\Sigma _g$.  
Matsumoto~\cite{Matsumoto} gave a Dehn twist presentation for a certain involution on $\Sigma _2$.
Korkmaz~\cite{Korkmaz1} and Gurtas~\cite{Gurtas} generalized Matsumoto's Dehn twist presentation for higher genus cases. 
Ishizaka~\cite{Ishizaka} gave a list of topological conjugacy classes of orientation-preserving periodic maps on $\Sigma _g$ which commute with the hyperelliptic involution and constructed their Dehn twist presentations. 
For the other cases that $g\leq 4$, Hirose~\cite{Hirose} gave an explicit list of all topological conjugacy classes of orientation-preserving periodic maps on $\Sigma _g$ and constructed Dehn twist presentations for their periodic maps.
The first author~\cite{Omori} gave a Dehn twist presentation for a periodic map which is a generalization of the hyperelliptic involution and is called the balanced superelliptic rotation. 

There is an application of constructing a Dehn twist presentation for a periodic map to a Lefschetz fibration, which is a fiber structure deeply related to a symplectic 4-manifold. 
The {\it mapping class group} $\M (\Sigma _{g})$ of $\Sigma _{g}$ is the group of isotopy classes of orientation-preserving self-homeomorphisms on $\Sigma _{g}$. 
For a given relation among right-handed Dehn twists in $\M (\Sigma _{g})$, we can construct a Lefschetz fibration over $\Sigma _0$ with a regular fiber $\Sigma _g$.
Since a some power of a Dehn twist presentation for a periodic map on $\Sigma _g$ is a relation in $\M (\Sigma _{g})$ among Dehn twists, a Dehn twist presentation for a periodic map by right-handed Dehn twists gives a Lefschetz fibration. 
Endo-Nagami~\cite{Endo-Nagami} gave a signature formula for Lefschetz fibrations over $\Sigma _0$ that is calculated from a corresponding relation among right-handed Dehn twists in $\M (\Sigma _{g})$. 

In the case where $S$ is a non-orientable surface, Lickorish~\cite{Lickorish1} showed that any self-homeomorphism $\varphi $ of $S$ is isotopic to a product of Dehn twists and ``crosscap slides.'' 
We call such a product a \textit{Dehn twist--crosscap slide presentation} for $\varphi $ (we often abbreviate it to \textit{\DC }).

Let $N_{g}$ be the connected closed non-orientable surface of genus $g\geq 1$ as in Figure~\ref{scc_closed_nonorisurf}, that is a connected sum of $g$ real projective planes.  
The $\times $-marks in Figure~\ref{scc_closed_nonorisurf} means the identification of the antipodal points of the boundary components, and in this paper, we use the model of $N_g$ as the surface in Figure~\ref{scc_closed_nonorisurf}. 
We remark that a neighborhood of the $\times $-mark in $\mathcal{N}$ as in Figure~\ref{scc_closed_nonorisurf} is homeomorphic to a M\"{o}bius band.
We call a M\"{o}bius band in a suface a \textit{crosscap} in it. 
A list of topological conjugacy classes of involutions on $N_g$ which are realized as isometries for some hyperbolic metric on $N_g$ is given by Bujalance-Etayo-Mart\'{i}nez-Szepietowski~\cite{BEMS} for $g=4,\ 5$. 
Remark that one \DC \ for a self-homeomorphism $\varphi $ on $N_g$ also gives a \DC \ for any homeomorphism which is \topconj to $\varphi $. 
Dugger~\cite{Dugger} gave a list of topological  conjugacy classes of all involutions on $N_g$ for $2\leq g\leq 7$ by using equivariant surgeries. 
For the case of non-orientable surfaces, the Nielsen realization theorem was given by Evans-Kolbe~\cite{Evans-Kolbe}. 
In particular, by their result, any periodic map on $N_g$ is realized as an isometry on $N_g$ for some hyperbolic metric on $N_g$. 
\begin{figure}[ht]
\includegraphics[scale=0.8]{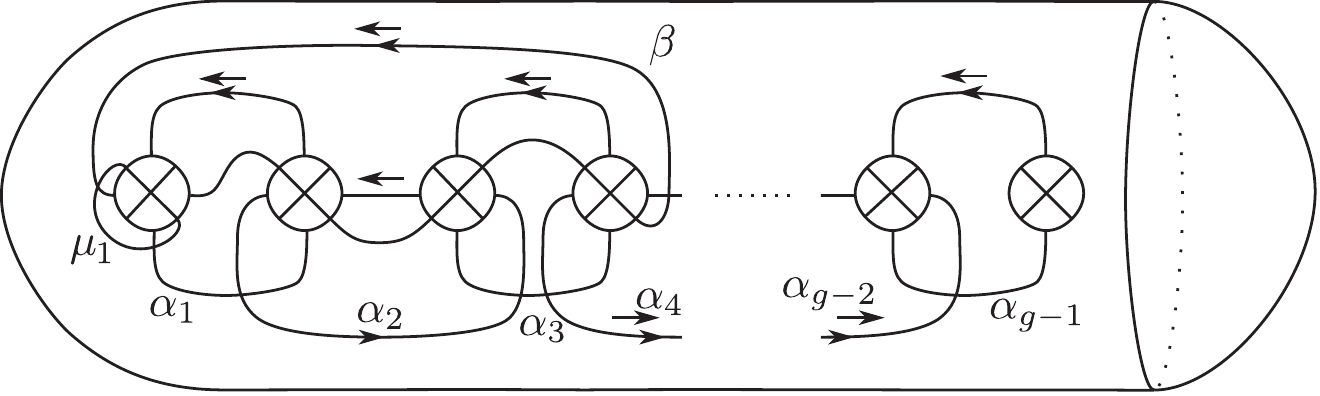}
\caption{The surface $N_g$ and simple closed curves $\mu _1,$ $\alpha _1, \dots , \alpha _{g-1},$ and $\beta $  on $N_g$.}\label{scc_closed_nonorisurf}
\end{figure}

Stukow~\cite{Stukow} gave a \DC \ for the hyperelliptic involution on $N_g$. 
The {\it mapping class group} $\M (N_{g})$ of $N_{g}$ is the group of isotopy classes of all self-homeomorphisms on $N_{g}$. 
In this paper, we give \DC s for all involutions on $N_g$ for $2\leq g\leq 5$ with respect to the generating set for $\M (N_g)$ given by Lickorish~\cite{Lickorish1} for $g=2$, by Birman-Chillingworth~\cite{Birman-Chillingworth} for $g=3$, and by Szepietowski~\cite{Szepietowski2} for $g\geq 4$ (Theorem~\ref{gen_mcg}).

\subsection{Main result}\label{section_main-result}

Topological conjugacy classes of involutions on $N_g$ for $2\leq g\leq 7$ are listed by Dugger~\cite{Dugger}. 
By Evans-Kolbe's result in~\cite{Evans-Kolbe}, these topological conjugacy classes are represented by isometries. 
As an argument in Section~\ref{section_NEC-group}, an involution on $N_g$ 
is corresponding to an ``NSK-map'' of genus $g$, where an NSK-map $\theta \colon \Gamma \to \Z _2=\Z /2\Z$ is a surjective homomorphism from an NEC group $\Gamma $ whose kernel is isomorphic to the fundamental group of $N_g$. 
In Table~\ref{table:NSKmap}, we list a representative for each equivalence class (which are also called a topological conjugacy class) of NSK-maps of genus $2\leq g\leq 5$. 
We remark that the genus 4 and 5 cases of Table~\ref{table:NSKmap} are given by Bujalance-Etayo-Mart\'{i}nez-Szepietowski~\cite{BEMS}. 
By Proposition~\ref{prop_topconj_inv-NSK}, the topological conjugacy relation on NSK-maps of genus $g$ is compatible with that on involutions on $N_g$. 
Therefore, to construct \DC s\ for involutions on $N_g$ with $2 \leq g \leq 5$, it is sufficient to give a \DC \ for the involution corresponding to each NSK-map listed in Table~\ref{table:NSKmap}.

Let $\iota _{g;s}$ and $\iota _{g;s,t}$ be involutions on $N_g$ for $2\leq g\leq 5$ which are  described as in Figures~\ref{figure_involution_genus2}, \ref{figure_involution_genus3}, \ref{figure_involution_genus4}, and \ref{figure_involution_genus5}. 
The words ``ref.'', ``$\pi $'', and ``anti.'' in the figures mean that the corresponding involution is induced by the reflection about the gray plane in the figure, the $\pi $-rotation about the black axis in the figure, and the antipodal action, respectively. 
The gray points and gray simple closed curves on surfaces in Figures~\ref{figure_involution_genus2}, \ref{figure_involution_genus3}, \ref{figure_involution_genus4}, and \ref{figure_involution_genus5} indicate isolated fixed points and \refline s of the corresponding involution, respectively (a \textit{\refline } of an involution $\iota $ on a surface means a simple closed curve on the surface which is fixed by $\iota $ pointwise). 
An NSK-map of genus $g$ which is listed in Table~\ref{table:NSKmap} is also labeled by a pair $(g;s)$ or a tuple $(g;s,t)$ of positive integers $g$, $s$, and $t$ as $\theta =\theta _{g;s}$ or $\theta =\theta _{g;s,t}$. 
By Proposition~\ref{prop_topconj}, an involution which is corresponding to $\theta _{g;s}$ (resp. $\theta _{g;s,t}$) is \topconj to $\iota _{g;s}$ (resp. $\iota _{g;s,t}$). 

\begin{figure}[ht]
\includegraphics[scale=0.7]{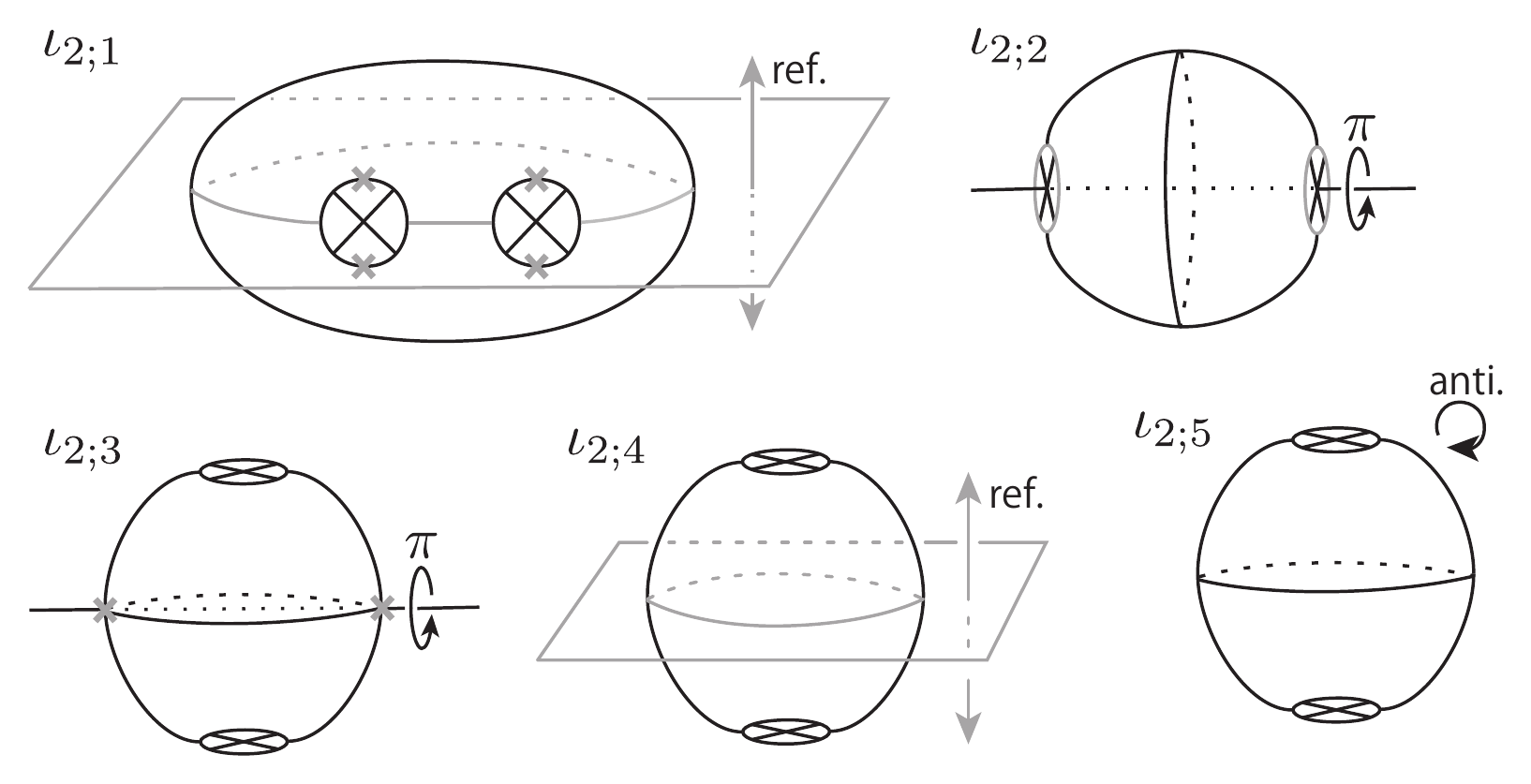}
\caption{A list of representatives of topological conjugacy classes of non-trivial involutions on genus 2 non-orientable surfaces.}\label{figure_involution_genus2}
\end{figure}

\begin{figure}[ht]
\includegraphics[scale=0.7]{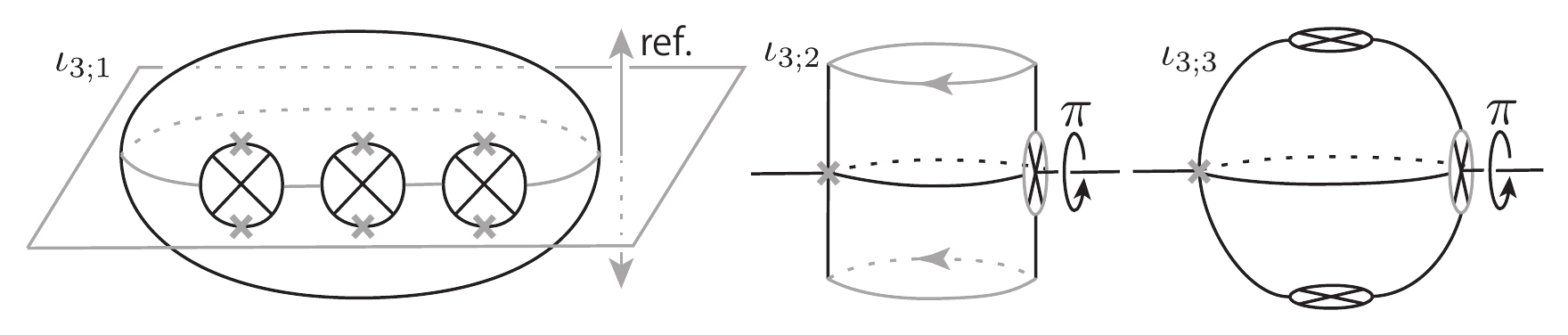}
\caption{A list of representatives of topological conjugacy classes of non-trivial involutions on genus 3 non-orientable surfaces.}\label{figure_involution_genus3}
\end{figure}

\begin{figure}[ht]
\includegraphics[scale=0.68]{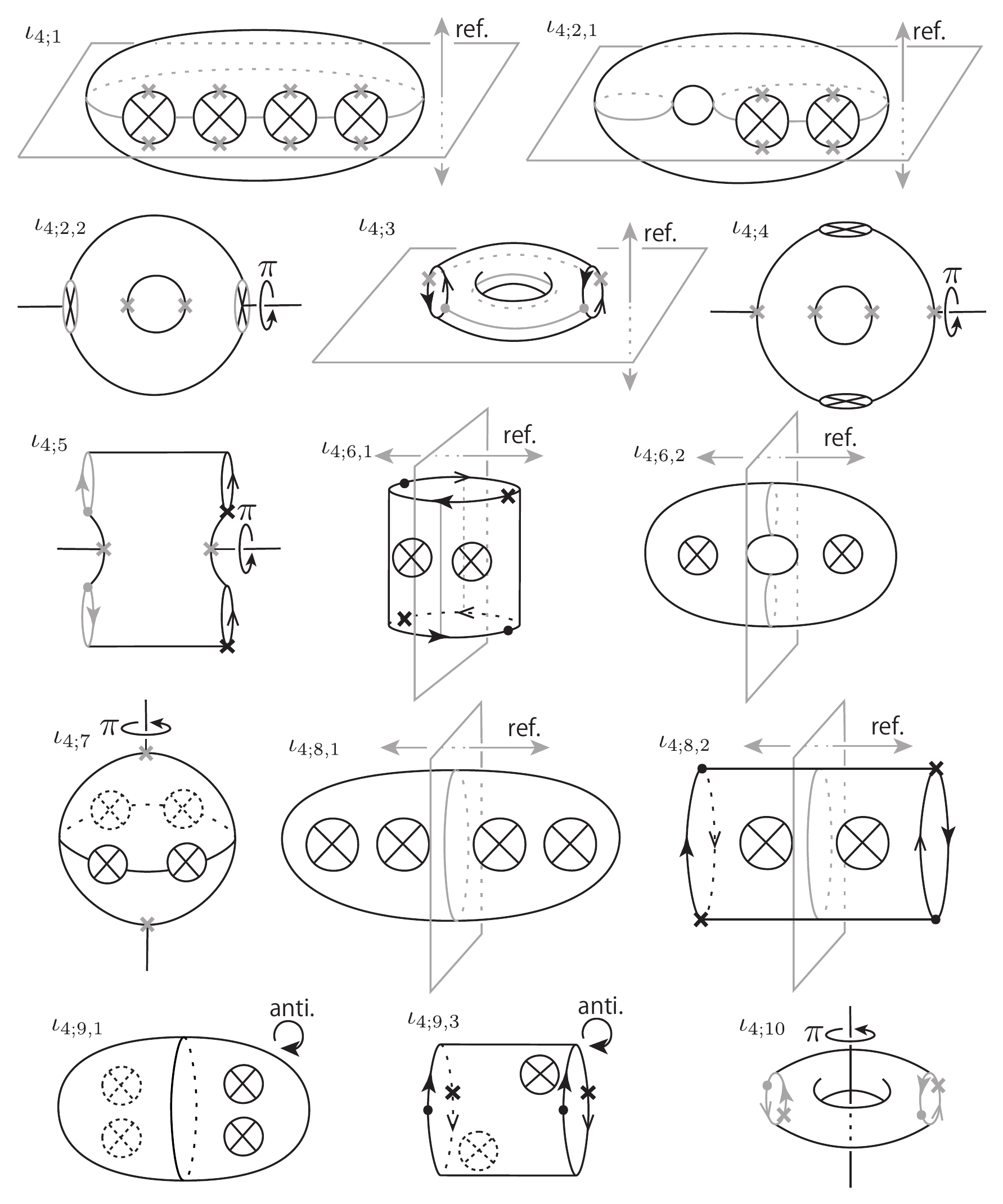}
\caption{A list of representatives of topological conjugacy classes of non-trivial involutions on genus 4 non-orientable surfaces.}\label{figure_involution_genus4}
\end{figure}

\begin{figure}[ht]
\includegraphics[scale=0.68]{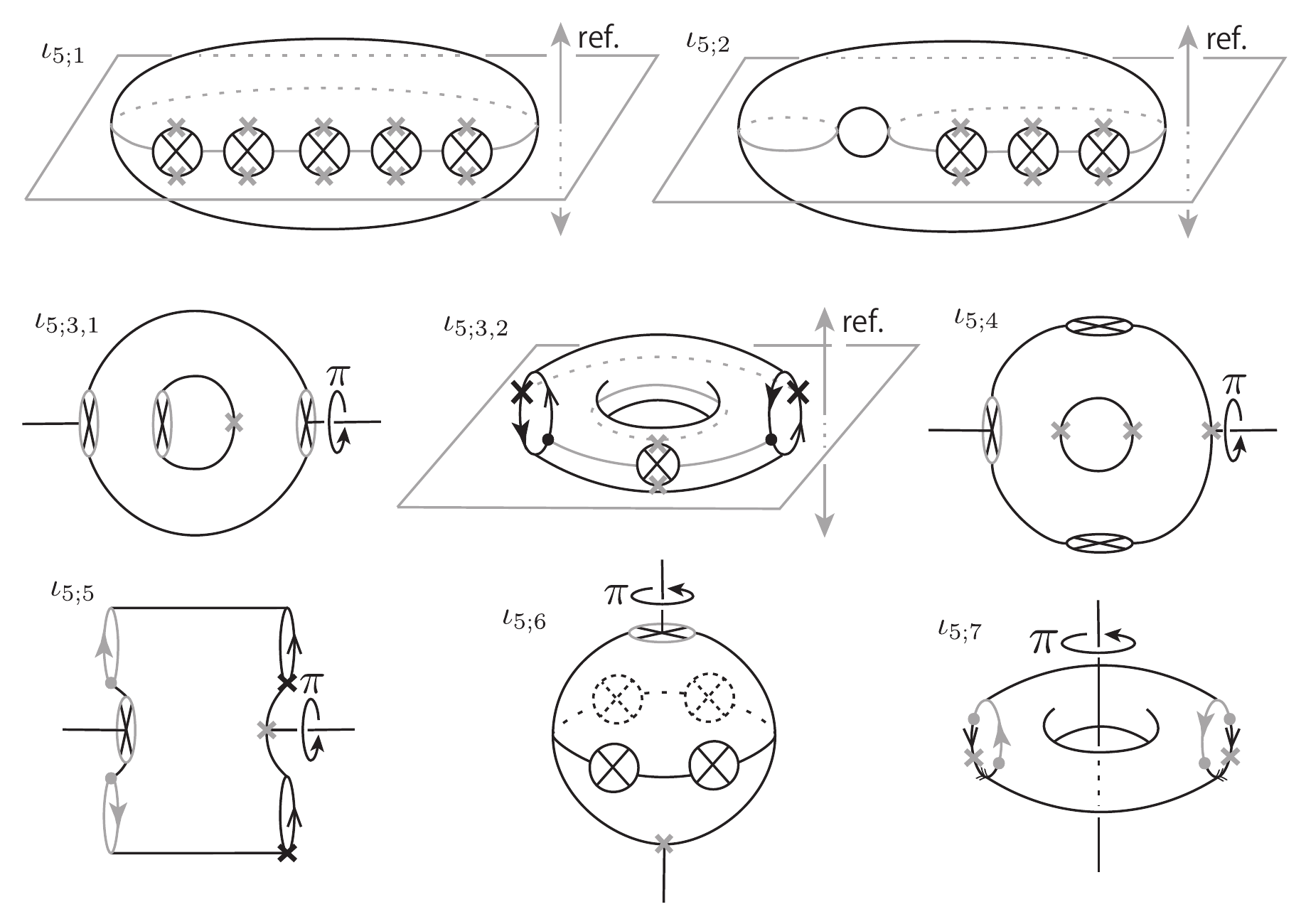}
\caption{A list of representatives of topological conjugacy classes of non-trivial involutions on genus 5 non-orientable surfaces.}\label{figure_involution_genus5}
\end{figure}

For a two-sided simple closed curve $\gamma $ on a closed surface $S$, we choose an orientation of a regular neighborhood of $\gamma $ in $S$. Then denote by $t_\gamma $ the right-handed Dehn twist along $\gamma $ with respect to the chosen orientation. In particular, for a given explicit two-sided simple closed curve, an arrow on a side of the simple closed curve indicates the positive direction of the Dehn twist (see Figure~\ref{scc_closed_nonorisurf} and \ref{dehntwist}).

\begin{figure}[ht]
\includegraphics[scale=0.7]{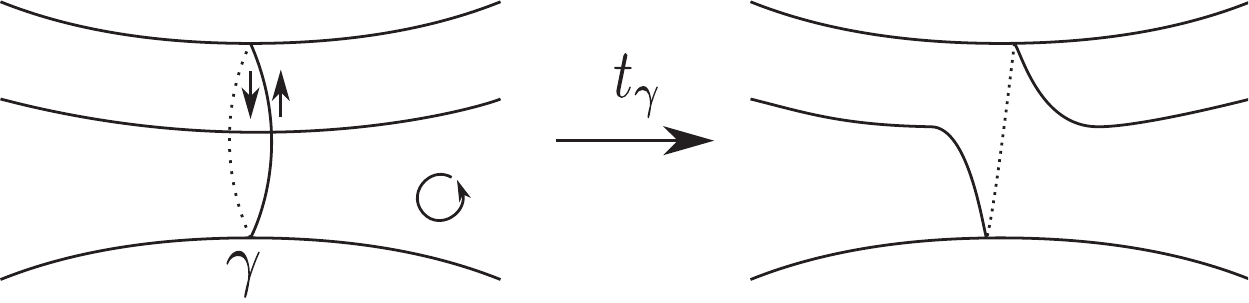}
\caption{The right-handed Dehn twist $t_\gamma $ along a two-sided simple closed curve $\gamma $ on $S$.}\label{dehntwist}
\end{figure}
 
Let $\mu $ be a one-sided simple closed curve on a closed non-orientable surface $N$ and $\alpha $ a two-sided simple closed curve on $N$ such that $\mu $ and $\alpha $ intersect transversely at one point. 
A regular neighborhood $\mathcal{N}$ of the union $\mu \cup \alpha $ in $N$ is a non-orientable surface of genus 2 with one boundary component (see Figure~\ref{crosscap_slide}). 
Recall that a neighborhood of $\mu $ is a crosscap in $\mathcal{N}$. 
Then, we denote by $Y_{\mu , \alpha }$ a self-homeomorphism on $N$ which is described as the result of pushing 
once the crosscap along $\alpha $ as in Figure~\ref{crosscap_slide}. 
We call $Y_{\mu , \alpha }$ a {\it crosscap slide} about $\mu $ and $\alpha $.

\begin{figure}[ht]
\includegraphics[scale=1.3]{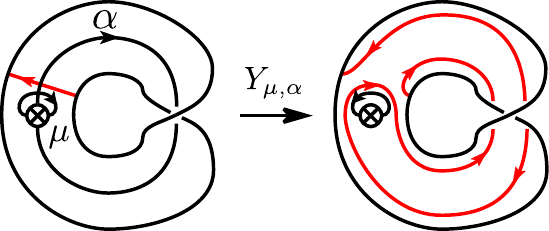}
\caption{The crosscap slide $Y_{\mu ,\alpha }$ about simple closed curves $\mu $ and $\alpha $ on $N$.}\label{crosscap_slide}
\end{figure}

For a non-orientable surface $N$ of genus $g$ and a homeomorphism $F \colon N_g\to N$, we remark that the composition $F\circ t_{\gamma }\circ F^{-1}$ (resp. $F\circ Y_{\mu ,\alpha }\circ F^{-1}$) coincides with the Dehn twist $t_{F(\gamma )}^\varepsilon $ (resp. the crosscap slide $Y_{F(\mu ), F(\alpha )}$) on $N$, where $\varepsilon \in \{ \pm1 \}$. 
Thus, if a \DC \ for one representative of the topological conjugacy class $[\iota ]$ of an involution $\iota $ on $N_g$ is given, then we obtain a \DC \ for any $\iota ^\prime \in [\iota ]$.  
Hence we also call a \DC \ for an involution $\iota $ on $N_g$ the \textit{\DC } for the topological conjugacy class $[\iota ]$, namely, for a product $w$ of Dehn twists and crosscap slides, $[\iota ]=w$ means that there exists a representative $\iota ^\prime \in [\iota ]$ such that $\iota ^\prime $ is isotopic to $w$. 
 
Let $\mu _1,\ \alpha _1,\ \dots ,\ \alpha _{g-1},\ \beta $ be simple closed curves on $N_g$ as in Figure~\ref{scc_closed_nonorisurf}. 
The following theorem is given by Lickorish~\cite{Lickorish1} for $g=2$, by Birman-Chillingworth~\cite{Birman-Chillingworth} for $g=3$, and by Szepietowski~\cite{Szepietowski2} for $g\geq 4$. 

\begin{thm}[\cites{Lickorish1, Birman-Chillingworth, Szepietowski2}]\label{gen_mcg}
The mapping class group $\M (N_g)$ is generated by $t_{\alpha _1}, \dots , t_{\alpha _{g-1}}$, and $Y_{\mu_1,\alpha _1}$ for $g\in \{ 2,3\}$, and by $t_{\alpha _1}, \dots , t_{\alpha _{g-1}}$, $Y_{\mu_1,\alpha _1}$, and $t_\beta $ for $g\geq 4$.
\end{thm}

Note that, for $g\geq 4$, Hirose~\cite{Hirose2} proved that the generating set for $\M (N_g)$ in Theorem~\ref{gen_mcg} is minimal among generating sets for $\M (N_g)$ by Dehn twists and crosscap slides. 

For convenience, we abuse notation of a self-homeomorphism on $N_g$ and its isotopy class, and express the homeomorphisms $t_{\alpha _i}$ $(i=1,\dots ,g-1)$, $t_\beta $, $Y_{\mu _1,\alpha _1}$, and $f^{-1}$ for any homeomorphism $f$ by $a_i=i$, $b$, $y$, and $\bar{f}$, respectively. 
The identity map on $N_{g}$ is expressed by $\mathrm{id}$. 
For self-homeomorphisms $\varphi $ and $\psi $ on $N_g$, we define $\psi \varphi :=\psi \circ \varphi $, i.e. $\varphi $ is applied first. 
The main theorem in this paper is as follows:

\begin{thm}\label{main_thm}
For $2\leq g\leq 5$, \DC s\ for any topological conjugacy classes of involutions are listed as follows: 

Genus 2 case: $[\iota _{2;1}]=y,\quad [\iota _{2;2}]=\mathrm{id},\quad [\iota _{2;3}]=1y,\quad [\iota _{2;4}]=1,\quad [\iota _{2;5}]=1$.\\

Genus 3 case: $[\iota _{3;1}]=(12)^3,\quad [\iota _{3;2}]=y,\quad [\iota _{3;3}]=1y$. \\

Genus 4 case: 
\begin{align*}
[\iota _{4;1}]&=\bar{y}23y23\bar{y}23,\quad [\iota _{4;2,1}]=21321y\bar{1}\bar{2}\bar{3}\bar{1}\bar{2}\bar{y},\quad [\iota _{4;2,2}]=12y\bar{1}\bar{2}\bar{y}12,\\
[\iota _{4;3}]&=23y(23)^212\bar{y}\bar{2}\bar{1}\bar{y},\quad [\iota _{4;4}]=by\bar{2}\bar{3}\bar{y}\bar{2}\bar{3}y\bar{2}\bar{3},\quad [\iota _{4;5}]=y(321)^2\bar{y}\bar{1}\bar{2}\bar{3}\bar{1}\bar{2},\\
[\iota _{4;6,1}]&=12y21\bar{y}\bar{2}\bar{1}y123,\quad [\iota _{4;6,2}]=by123\bar{y}2\bar{3}2312y\bar{2},\quad [\iota _{4;7}]=(123)^2,\\
[\iota _{4;8,1}]&=\bar{y}\bar{2}y12\bar{y}321,\quad [\iota _{4;8,2}]=\bar{b}(y2^2)^2121y\bar{2}\bar{1},\\
[\iota _{4;9,1}]&=\bar{y}\bar{2}\bar{3}\bar{y}23y21,\quad [\iota _{4;9,3}]=\bar{y}\bar{2}\bar{1}y21\bar{y}\bar{1}\bar{2}y12\bar{y}\bar{b},\quad [\iota _{4;10}]=\bar{b}\bar{y}23y2\bar{y}\bar{1}\bar{2}\bar{3}.
\end{align*}

Genus 5 case: 
\begin{align*}
[\iota _{5;1}]&=(1234)^5,\quad [\iota _{5;2}]=2312y4321\bar{y}\bar{1}\bar{2}\bar{3}\bar{4}\bar{2}\bar{1}\bar{3}\bar{2}\bar{y},\\
[\iota _{5;3,1}]&=12y\bar{1}\bar{2}\bar{3}\bar{y}\bar{2}y123y123,\quad [\iota _{5;3,2}]=4321y4321\bar{y}\bar{2}\bar{3}\bar{4}y\bar{1}\bar{2}\bar{3}\bar{1}\bar{2}\bar{y}21,\\
[\iota _{5;4}]&=by\bar{2}\bar{3}\bar{y}\bar{2}\bar{3}y\bar{2}\bar{3},\quad [\iota _{5;5}]=4321\bar{y}432y\bar{1}\bar{2}\bar{3}\bar{4}\bar{y}\bar{2}\bar{3}\bar{1}\bar{2}y21,\\
[\iota _{5;6}]&=y234\bar{y}234,\quad [\iota _{5;7}]=\bar{b}\bar{4}\bar{3}\bar{2}\bar{1}y1234\bar{y}23y2\bar{y}\bar{1}\bar{2}\bar{3}.\\
\end{align*}
\end{thm} 

We can check that by Proposition~\ref{prop_topconj}, the topological conjugacy classes of the involutions appear in Theorem~\ref{main_thm}.
Contents of this paper are as follows. 
In Section~\ref{section_NSK_conjclass}, we give a list of representatives for topological conjugacy classes of involutions of $N_g$ for $2 \leq g \leq 5$ by using a list of equivalence classes of NSK-maps by Bujalance-Etayo-Mart\'{i}nez-Szepietowski~\cite{BEMS}. 
In Section~\ref{section_blowup}, we explain a blowup operation for an involution on $N_g$ and the topological conjugacy class of an involution on $N_g$ which is obtained by the blowup operation. 
In Section~\ref{section_mainthm}, we give the proof of Theorem~\ref{main_thm}.

\section{NSK-maps and topological conjugacy classes of involutions}\label{section_NSK_conjclass}
%Preliminary
In this section, we give a list of topological conjugacy classes of involutions on $N_g$ for $2 \leq g \leq 5$ via equivalence classes of NSK-maps. 
We remark that Dugger~\cite{Dugger} gave a list of topological  conjugacy classes of involutions on $N_g$ for $2\leq g\leq 7$ by using equivariant surgeries. 
A list of equivalence classes of NSK-maps is given by Bujalance-Etayo-Mart\'{i}nez-Szepietowski~\cite{BEMS} for $g=4,\ 5$. 

\subsection{NEC groups}\label{section_NEC-group}

A \emph{non-Euclidean crystallographic group} (\emph{NEC group}, for short) is a discrete subgroup of the isometric transformations of the hyperbolic plane $\Hy$.
Note that an NEC group contains an orientation-reversing element in general.
An NEC group $\Gamma $ has a signature consisting of the following sequence:
\begin{equation*}
\sigma (\Gamma )=\left( g, \pm, [m_1, \ldots, m_r], \left\{C_{1}, \ldots, C_{k} \right\} \right).
\end{equation*}
Here, $g$ and $m_i$ are integers with $g \geq 0$ and $m_i \geq 2$, and each $C_i$ is a sequence of integers $(n_{i1}, \ldots, n_{is_i})$ with $n_{ij} \geq 2$.
Wilkie~\cite{Wilkie} has studied the NEC groups  and introduced their signatures.
Wilkie showed the algebraic classification of the NEC groups by using their signatures, and Macbeath~\cite{Macbeath} gave the complete classification of them.

\begin{conv}
\label{conv:signature}
% generators from a signature
The signature $(g, \pm, [m_1, \dots, m_r], \left\{C_{1}, \dots, C_{k} \right\} )$ of an NEC group $\Gamma$ determines a presentation with generators and relations as in Table~\ref{table:NECgrp}.
The symbols $x_i$ $(1\leq i\leq r)$ denote elliptic transformations, and $c_{i,j}$ $(1\leq i\leq k,\ 0\leq j\leq s_i)$ denote reflections.
In addition, the symbols $e_i$ $(1\leq i\leq k)$ and $a_i$, $b_i$ $(1\leq i\leq g)$ are hyperbolic translations, and $d_i$ $(1\leq i\leq g)$ are glide reflections.  
\begin{table}[ht]
\begin{tabular}{lll}
\hline
Signature & Generator(s) & Relation(s) \\ \hline
$m_i$ & $x_i$ & $x_i^{m_i} = 1$ \\
$C_i$ & $e_i$ & $c_{i,s_i} = e_i^{-1} c_{i,0} e_i$\\
      & $c_{i,0}, c_{i,1}, \dots, c_{i,s_i}$ & $c_{i,j}^2 = 1$, $(c_{i, j - 1} c_{i, j})^{n_{ij}} = 1$ \\
$+$ & $a_1, b_1, \dots, a_g, b_g$ & $x_1 \dotsm x_r e_1 \dotsm e_k a_1 b_1 a_1^{-1} b_1^{-1} \dotsm a_g b_g a_g^{-1} b_g^{-1} = 1$ \\
$-$ & $d_1, \dots, d_g$ & $x_1 \dotsm x_r e_1 \dotsm e_k d_1^2 \dotsm d_g^2 = 1$ \\
\hline
\end{tabular}
\caption{The table shows the generators and defining relations of the presentation for the NEC group corresponding to a signature $\left( g, \pm, [m_1, \ldots, m_r], \left\{C_{1}, \ldots, C_{k} \right\} \right)$. One of the last two generators is adopted depending on whether the sign is $+$ or $-$. If $g = 0$, then $a_i$, $b_i$, and $d_i$ are excluded from the generators and relations. Similarly, the corresponding symbols are excluded if there is no $m_i$ or $C_i$.}
\label{table:NECgrp}
\end{table}
\end{conv}

Each NEC group $\Gamma$ induces an orbifold as an orbit space $\Hy / \Gamma$, similar to Fuchsian groups.
The signature of an NEC group describes the information of the singular points of the corresponding orbifold.
For an NEC group $\Gamma$ of signature $(g, \varepsilon, [m_1, \ldots, m_r], \left\{C_1, \ldots, C_k \right\})$, the orientability of the underlying space of the corresponding orbifold is determined by whether $\varepsilon$ is $+$ or $-$. When $\varepsilon = +$, it is a sphere with $g$ handles and ordered $k$ holes, and when $\varepsilon = -$, it is a sphere with $g$ crosscaps and $k$ ordered holes.
The singular locus of the orbifold consists of $r$ ordered cone points in the interior and $s_i$ cyclically ordered reflector (or dihedral) corners on the $i$-th boundary component.
Each $i$-th cone point is with angle $2\pi/m_i$ and each $j$-th dihedral corner in $i$-th boundary is a corner reflector with angle $\pi/ n_{ij}$. 
We remark that the quotient image of the isolated fixed point of $x_i$ coincide with the $i$-th cone point of $\Hy /\Gamma $, and that of the union of the geodesics fixed by $c_{i,0}, \dots , c_{i, s_i}$ in $\Hy /\Gamma $ coincide with the $i$-th boundary component of $\Hy /\Gamma $.  

We obtain an NEC group with a specific signature from the pair of a closed surface with non-low genus and its involution.
Let $\iota $ be an involution on a closed surface $S$ and $\Gamma^\prime$ be a torsion free NEC group such that $\Gamma^\prime$ is isomorphic to the fundamental group $\pi_1(S)$.
By Nielsen realization theorem, there exists an isometry representative of the isotopy class of $\iota$, and we have an isometry $\widehat\iota$ on $\Hy$ that is a lift of the representative.
We denote by $\Gamma$ the subgroup of the isometry group of $\Hy$ generated by generators of $\Gamma^\prime$ and the isometry $\widehat\iota$.
Since $\widehat\iota$ is a lift of $\iota$, the quotient spaces $\Hy/\Gamma$ and $S / \iota$ have the same topological type.
Hence, $\Gamma$ is a discrete subgroup, that is, $\Gamma$ is an NEC group.
Since the involution $\iota $ on $S$ acts on a neighborhood of its fixed point as a rotation or a reflection, the orbifold $\Hy /\Gamma \approx S/\iota $ has only cone points with angle $\pi$ and boundary components with no reflector corners. 
Then we have $m_i=2$ for each $1\leq i\leq r$ and $s_i = 0$ for each $1\leq i\leq k$, that means that each $C_i$ is the empty sequence $C_i=(-)$. 
Hence the signature of $\Gamma $ is 
\begin{equation*}
    \sigma (\Gamma )=(h, \varepsilon, [\overbrace{2, \ldots, 2}^{r}], \{\overbrace{(-), \ldots, (-)}^{k}\}). 
\end{equation*}
For convenience, we write
\begin{equation*}
    (h, \varepsilon, [\overbrace{2, \ldots, 2}^{r}], \{\overbrace{(-), \ldots, (-)}^{k}\}) =(h,\varepsilon , [(2)^{r}],\{ (-)^{k}\}),
\end{equation*}
and express $[(2)^{0}]=[-]$ and $\{ (-)^{0}\}=\{ -\}$. 

There exists a natural epimorphism from an NEC group to $\Z_2$ induced by an involution $\iota$ on a closed surface $S$ whose kernel is isomorphic to $\pi_1(S)$.
Let $\Gamma^\prime$ be a subgroup of the isometry group that is isomorphic to $\pi_1(S)$, $\Gamma$ the subgroup of the isometry group induced by $\Gamma^\prime$ and $\iota$ introduced in the above discussion, and $\widehat{\iota}$ an isometry that is a lift of $\iota$.
Then, we have $p \circ \widehat\iota = \iota \circ p$, where $p$ is the projection $\Hy \to \Hy/\Gamma^\prime \approx S$.
Thus, $p \circ \widehat\iota^{\;2} = p$ holds.
So, $\widehat\iota^{\;2}$ is in $\Gamma^\prime$.
Hence, there is a homomorphism $\theta_\iota\colon \Gamma \to \Gamma/\Gamma^\prime \cong \Z_2$.
If the surface $S$ is homeomorphic to $N_g$, then we call such a map $\theta_\iota$ a \emph{non-orientable-surface-kernel map} (\emph{NSK-map}, for short) \emph{of genus $g$}, namely, an NSK-map of genus $g$ is an epimorphism from an NEC group to a finite group such that $\ker(\theta_\iota)$ is isomorphic to $\pi _1(N_g)$. 
For an involution $\iota $ on $N_g$ and a corresponding NSK-map $\theta \colon \Gamma \to \Z _2$, we call the signature of $\Gamma $ the \emph{signature} of $\iota $, and write $\sigma (\iota )=\sigma (\Gamma )$. 

Conversely, an NSK-map $\theta \colon \Gamma \to \Z _2$ of genus $g$ induces an involution of a non-orientable surface of genus $g$ as follows. 
Since $\ker \theta $ is isomorphic to $\pi _1(N_g)$ and $\Hy $ is simply connected, the quotient space $\Hy /\ker \theta $ is homeomorphic to $N_g$. 
By the surjectivity of $\theta$, there exists an isometry $\widehat\iota \in \Gamma$ such that $\theta(\widehat\iota)$ is the nontrivial element of $\Z_2$.
Hence, $\widehat\iota$ induces an involution on $\Hy/\ker \theta \approx N_g$.
Note that the involution is uniquely determined, not depending on the choice of isometries in $\Gamma$ mapped to the nontrivial elements.

NSK-maps help classify involutions on non-orientable surfaces.
In order to see the correspondence between NSK-maps and involutions, we review an equivalence relation in NSK-maps as follows.
Two NSK-maps of genus $g$, $\theta_1\colon \Gamma_1 \to G$ and $\theta_2\colon \Gamma_2 \to G$, are said to be \emph{topologically conjugate} if there exists an isomorphism $\psi\colon \Gamma_1 \to \Gamma_2$ and a self-isomorphism $\alpha$ of $G$ such that $\theta_2 \circ \psi = \alpha \circ \theta_1$. 
Compatibility of topological conjugations between NSK-maps of genus $g$ with ones between corresponding involutions on $N_g$ is ensured by the following proposition. 
\begin{prop}[\cite{Macbeath}*{Theorem~3} (cf. \cite{BCCS}*{Section~2})]\label{prop_topconj_inv-NSK}
Let $\theta _i\colon \Gamma _i\to \Z _2$ $(i=1, 2)$ be NSK-maps of genus $g$ and $\iota _i$ the involution on $N_g$ corresponding to $\theta _i$ as above. 
Then, $\theta _1$ is \equiva to $\theta _2$ if and only if $\iota _1$ is also \topconj to $\iota _2$. 
\end{prop}

By an argument in this section, we have the following lemma. 

\begin{lem}\label{lem_number_fixedpt}
Let $\theta \colon \Gamma \to \Z _2$ be an NSK-map of genus $g$ 
from an NEC group $\Gamma $ of signature $(h,\varepsilon , [(2)^{r}],\{ (-)^{k}\})$ and 
$\iota $ the corresponding involution on $N_g$. 
Then, the number of the isolated fixed points of $\iota$ and that of the \refline s of it are $r$ and $k$, respectively.
\end{lem}

\subsection{Lists of equivalence classes of NSK-maps of genera up to \texorpdfstring{$5$}{5}}

In this section, we review the list of topological conjugacy classes of NSK-maps of genus $2\leq g\leq 5$ follows from~\cite{BEMS}. 
Let $\Gamma $ be an NEC group of signature $(h,\varepsilon , [(2)^{r}],\{ (-)^{k}\})$. 
We regard $\Gamma $ as a subgroup of the isometry group of $\mathbb{H}^2$ which has the presentation as in Table~\ref{table:NECgrp}. 
Recall that the presentation has the generators $x_1, \dots, x_r$, $e_1, \dots, e_k$, $c_{1,0}, c_{2,0}, \dots, c_{k,0}$, and  $a_1, b_1, \dots, a_h, b_h$ if $\varepsilon =+$ (resp. and $d_1, \dots, d_h$ if $\varepsilon =-$), and as isometries, $x_j$ for $1\leq j\leq r$ is a rotation of $\mathbb{H}^2$, $c_{j,0}$ for $1\leq j\leq r$ is a reflection of $\mathbb{H}^2$, $d_i$ is a glide reflection of $\mathbb{H}^2$, and $a_i, b_i$, and $e_i$ are translations on $\mathbb{H}^2$. 
In particular, we remark that $x_j$ and $c_{j,0}$ are order 2 elements, and the other elements are infinite order elements. 
First, since the fundamental group $\pi _1(N_g)$ of $N_g$ is torsion free for $g\geq 2$ by \cite{Karrass-Magnus-Solitar}, we have the following lemma.
\begin{lem}\label{representatives_z2_action}
Suppose that $g\geq 2$. 
Let $\theta \colon \Gamma \to \Z _2=\{ 1, X\}$ be an NSK-map of genus $g$ 
from an NEC group $\Gamma $ of signature $(h,\varepsilon , [(2)^{r}],\{ (-)^{k}\})$. 
Then the images of $x_j$ and $c_{j,0}$ by $\theta $ are nontrivial elements in $\Z _2$, namely, we have $\theta (x_j)=X$ $(1\leq j\leq r)$ and $\theta (c_{j,0})=X$ $(1\leq j\leq k)$. 
\end{lem}

The next proposition gives a necessary and sufficient condition for being that two NSK-maps are topologically conjugate. 

\begin{prop}[\cite{BEMS}*{Theorem 2}]\label{thm_bems}
Let $\Gamma$ be an NEC group of signature $(h; \varepsilon; [(2)^r]; \{ (-)^k \})$ and $\theta_i\colon \Gamma \to \Z _2$ an NSK-map for each $i = 1, 2$. 
Set $n_i$ and $m_i$ as follows:
\begin{align*}
n_i &= \# \{ j \in \{ 1, \ldots, k \} \mid \theta_i(e_j) = X \},\\
m_i &= \# \{ j \in \{ 1, \ldots, h \} \mid \theta_i(d_j) = X \}.
\end{align*}
Then, $\theta_1$ and $\theta_2$ are topologically conjugate if and only if the following two conditions hold:
\begin{enumerate}[label=\textup{(\arabic*)}]
\item $n_1 = n_2$
\item If $r = n_1 = n_2 = 0$ and $\varepsilon = -$, then $m_1 \equiv m_2 \pmod{2}$. In addition, if one of $m_i$ is zero, then the other is also zero.
\end{enumerate}
\end{prop}

We give a topological interpretations for the condition~(1) of the above proposition in Lemma~\ref{theta_e_i}. 

By using Proposition~\ref{thm_bems}, Bujalance-Etayo-Mart\'{i}nez-Szepietowski gave a list of topological conjugacy classes of NSK-maps of genera 4 and 5 in Sections~4 and~5 in~\cite{BEMS}. 
By their list and \cite{Dugger}*{Theorem~8.6}, we list representatives of topological conjugacy classes of NSK-maps of genus $2\leq g\leq 5$ as in Table~\ref{table:NSKmap}\footnote{For the lists of genera 2 and 3 in Table~\ref{table:NSKmap}, we can check the correctness of the lists by the fact that the number of topological conjugacy classes of involutions on $N_2$ is 5 and that on $N_3$ is 3 by Theorem~8.6 in \cite{Dugger}, and distinct two NSK-maps are not topologically conjugate.}. 
These representatives are expressed as $\theta _{g;s}$ or $\theta _{g;s,t}$, as in the fourth column of Table~\ref{table:NSKmap}.
The subscripts are positive integers, where $g$ is the genus of an NSK-map (column 1), $s$ is a number to distinguish NEC groups that are the domains (column 2), and $t$ is one to differentiate NSK-maps with the same genus and domain.
The fifth column presents the image of the generator $(a_1, b_1, \dots, a_g, b_g, d_1, \dots, d_g,\ x_1, \dots, x_r,\ e_1, \dots, e_k,\ c_{1,0}, \dots, c_{k,0})$ of the domain by each NSK-map. 
We remark that, for a genus and an NEC group, if there exists a unique NSK-map of the genus whose domain is the group, we abuse notation and write $\theta_{g;s}$ instead of $\theta_{g;s,1}$ for simplicity. 
For instance, the NSK-map $\theta _{4;2,1}\colon \Gamma \to \Z _2$ shown in the second row of genus 4 in Table 2 is the 1st NSK-map whose domain is the 2nd NEC group in genus 4.
Furthermore, the map satisfies
\[
\theta _{4;2,1}\colon (x_1, x_2, e_1, e_2, c_{1,0}, c_{2,0})\to (X, X, 1, 1, X, X),
\]
where $X$ is the nontrivial element of $\Z_2$.

We can describe the involutions of the genus 2 non-orientable surface $N_2$ by using NSK-maps with the domains as discrete subgroups of the isometry group of the Euclidean plane instead of NEC groups.
As a matter of fact, we cannot describe any involution of $N_2$ by an NSK-map with an NEC group domain because $N_2$ admits a Euclidean metric.
However, a similar argument above can be applied to discrete subgroups of the isometry group of the Euclidean plane, and the same fact holds.
Besides, as Macbeath mentioned in \cite{Macbeath}*{Section~10}, it is possible to represent the discrete subgroups of the isometry group of the Euclidean plane using the signatures of the NEC groups introduced in Convention~\ref{conv:signature}. 
Hence, we can present the involutions of $N_2$ in the same way as $N_g$ with $g = 3, 4, 5$.

\begin{table}[ht]
\begin{tabular}{lllll}
\hline
Genus & Num. & $\Gamma$ & Name & Image \\ \hline
2 & 1      & $\NEC{0}{p}{2}{1}$ & $\theta_{2;1}$ & $(X, X, 1, X)$ \\
 & 2      & $\NEC{0}{p}{0}{2}$ & $\theta_{2;2}$ & $(X, X, X, X)$ \\
 & 3      & $\NEC{1}{m}{2}{0}$ & $\theta_{2;3}$ & $(\text{$1$ or $X$}, X, X)$ \\
 & 4      & $\NEC{1}{m}{0}{1}$ & $\theta_{2;4}$   & $(1, 1, X)$ \\
 & 5      & $\NEC{2}{m}{0}{0}$ & $\theta_{2;5}$ & $(1, X)$ \\ \hline
3 & 1      & $\NEC{0}{p}{3}{1}$ & $\theta_{3;1}$ & $(X, X, X, X, X)$ \\
 & 2      & $\NEC{0}{p}{1}{2}$ & $\theta_{3;2}$ & $(X, X, 1, X, X)$ \\
 & 3      & $\NEC{1}{m}{1}{1}$ & $\theta_{3;3}$ & $(\text{$1$ or $X$}, X, X, X)$ \\ \hline
4 & 1      & $\NEC{0}{p}{4}{1}$ & $\theta_{4;1}$   & $(X, X, X, X, 1, X)$ \\
 & 2      & $\NEC{0}{p}{2}{2}$ & $\theta_{4;2,1}$ & $(X, X, 1, 1, X, X)$ \\
 &        &                                    & $\theta_{4;2,2}$ & $(X, X, X, X, X, X)$ \\
 & 3      & $\NEC{0}{p}{0}{3}$ & $\theta_{4;3}$   & $(X, X, 1, X, X, X)$ \\
 & 4      & $\NEC{1}{m}{4}{0}$ & $\theta_{4;4}$   & $(\text{$1$ or $X$}, X, X, X, X, X)$ \\
 & 5      & $\NEC{1}{m}{2}{1}$ & $\theta_{4;5}$   & $(\text{$1$ or $X$}, X, X, X, 1, X)$\\
 & 6      & $\NEC{1}{m}{0}{2}$ & $\theta_{4;6,1}$ & $(\text{$1$ or $X$}, X, X, X, X)$\\
 &        &                    & $\theta_{4;6,2}$ & $(1, 1, 1, X, X)$\\
 & 7      & $\NEC{2}{m}{2}{0}$ & $\theta_{4;7}$   & $(\text{$1$ or $X$}, \text{$1$ or $X$}, X, X)$\\
 & 8      & $\NEC{2}{m}{0}{1}$ & $\theta_{4;8}$   & $(1, \text{$1$ or $X$}, 1, X)$\\
 & 9      & $\NEC{3}{m}{0}{0}$ & $\theta_{4;9,1}$ & $(X, 1, 1)$\\
 &        &                    & $\theta_{4;9,3}$ & $(X, 1, X)$\\
 & 10     & $\NEC{1}{p}{0}{1}$ & $\theta_{4;10}$  & $(X, \text{$1$ or $X$}, 1, X)$\\ \hline
5 & 1      & $\NEC{0}{p}{5}{1}$ & $\theta_{5;1}$   & $(X, X, X, X, X, X, X)$ \\
 & 2      & $\NEC{0}{p}{3}{2}$ & $\theta_{5;2}$   & $(X, X, X, X, 1, X, X)$ \\
 & 3      & $\NEC{0}{p}{1}{3}$ & $\theta_{5;3,1}$ & $(X, X, X, X, X, X, X)$ \\
 &      &                    & $\theta_{5;3,2}$ & $(X, X, 1, 1, X, X, X)$ \\
 & 4      & $\NEC{1}{m}{3}{1}$ & $\theta_{5;4}$   & $(\text{$1$ or $X$}, X, X, X, X, X)$ \\
 & 5      & $\NEC{1}{m}{1}{2}$ & $\theta_{5;5}$   & $(\text{$1$ or $X$}, X, X, 1, X, X)$\\
 & 6      & $\NEC{2}{m}{1}{1}$ & $\theta_{5;6}$   & $(\text{$1$ or $X$}, \text{$1$ or $X$}, X, X, X)$\\
 & 7      & $\NEC{1}{p}{1}{1}$ & $\theta_{5;7}$   & $(\text{$1$ or $X$}, \text{$1$ or $X$}, X, X, X)$\\ \hline
\end{tabular}
\caption{List of NSK-maps. The Image column represents the images by NSK-map when the generators corresponding to the signature are arranged in the order $a_1, b_1, \dots, a_g, b_g$, $d_1, \dots, d_g$, $x_1, \dots, x_r$, $e_1, \dots, e_k$, $c_{1,0}, \dots, c_{k,s_k}$. The symbol $X$ denotes the generator of $\Z_2$.}
\label{table:NSKmap}
\end{table}

\subsection{Correspondence between the topological conjugacy classes of involutions and those of NSK-maps of genera up to 5}
\label{section_list_involution}

In this section, we check the correspondence between the topological conjugacy classes of involutions in Figures~\ref{figure_involution_genus2}--\ref{figure_involution_genus5} and those of NSK-maps in Table~\ref{table:NSKmap}.
To give this, we explain a topological interpretation for the condition~(1) in Proposition~\ref{thm_bems}. 
For an involution $\iota $ on a surface $N$, we recall that a \textit{\refline } $\alpha $ of $\iota $ is a simple 
closed curve on $N$ such that the restriction of $\iota $ to $\alpha $ is the identity map on $\alpha $.

The next lemma gives a topological interpretation for the condition~(1) in Proposition~\ref{thm_bems}. 

\begin{lem}\label{theta_e_i}
Let $\theta \colon \Gamma \to \Z _2$ be an NSK-map of genus $g$ 
from an NEC group $\Gamma $ of signature $(h,\varepsilon , [(2)^{r}],\{ (-)^{k}\})$ with $k\geq 1$, $\iota $ the corresponding involution on $N_g$, and $e_i$ $(1\leq i\leq k)$ translations in the generating set for $\Gamma $ as in Convention~\ref{conv:signature}.
Set 
\[
n=\# \{ j \in \{ 1, \ldots, k \} \mid \theta (e_j) = X \}. 
\]
Then $n$ \refline s of $\iota $ are one-sided, and the other $k-n$ \refline s of $\iota $ are two-sided. 
\end{lem}
\begin{proof}
Let $U$ be the universal covering space of $N_g$ and $p$ the natural projection $U \to U / \ker\theta \approx N_g$.
(Note that considering the metric of $N_g$, whether $U$ is the Euclidean plane or the hyperbolic plane depends on if $g = 2$ or $g > 2$.)
We first show that for each mirror action $c_{i,0}$ fixing the geodesic $\tilde{\gamma}_i$, the image $\gamma_i := p(\tilde{\gamma}_i)$ is also fixed by the involution $\iota$ and a simple closed (geodesic) curve.
By the definition of $\iota$, each representative $\tilde{\iota}$ of the nontrivial class of $\Gamma/\ker\theta \cong \Z_2$ induces $\iota$.
That is, $\iota \circ p = p \circ \tilde{\iota}$ holds.
By Lemma~\ref{representatives_z2_action}, each $c_{i,0}$ satisfies $\theta(c_{i,0}) = X$, so $\iota$ fixes each point of the image $\gamma_i$.
By the assumption, the signature of $\Gamma$ defines a relation $c_{i,0} = e_i^{-1} c_{i,0} e_i$ for each $i$. Hence, each translation $e_i$ preserves the geodesic $\tilde{\gamma}_i$ fixed by $c_i$.
In particular, if $U$ is the hyperbolic plane, then the axis of the hyperbolic translation $e_i$ and $\tilde{\gamma}_i$ coincide.
Hence, the image $\gamma_i$ is a simple closed geodesic curve.
In particular, $\gamma_i$ is a \refline \ of $\iota$.

In the case of $\theta(e_i) = 1$, the translation $e_i$ is contained in $\ker\theta$.
Since $e_i$ is an orientation-preserving isometry, the curve $\gamma_i$ is two-sided.
On the other hand, if $\theta(e_i) = X$, the translation $e_i$ is not contained in $\ker\theta$, but the composition $c_{i,0} \circ e_i$ is an element of it.
Since the composition is an orientation-reversing isometry and fixes the geodesic $\tilde{\gamma}_i$, the curve $\gamma_i$ is one-sided.
Therefore, by the assumption, there are exactly $n$ one-sided \refline s, and $k - n$ two-sided \refline s.
\end{proof}

We define involutions $\iota_{g;s}$ and $\iota_{g;s,t}$ on closed non-orientable surfaces described as in Figures~\ref{figure_involution_genus2}--\ref{figure_involution_genus5}.
Recall that for an NSK-map $\theta _{g;s,t}\colon \Gamma \to \Z _2$ in Table~\ref{table:NSKmap}, we abuse notation and simply write $\theta _{g;s,1}=\theta _{g;s}$ when the topological conjugacy class of $\theta _{g;s,t}$ depends on only genus $g$ and the label $s$. 
Similarly, for convenience, we denote $\iota _{g;s,1}=\iota _{g;s}$ when $\iota _{g;s,t}$ for $t\not=1$ is not described in Figures~\ref{figure_involution_genus2}--\ref{figure_involution_genus5}. 
Then we have the following proposition. 

\begin{prop}\label{prop_topconj}
Suppose that $2\leq g\leq 5$. 
Let $\theta _{g;s,t}\colon \Gamma \to \Z _2$ be an NSK-map of genus $g$ defined in Table~\ref{table:NSKmap}, and $\iota $ the corresponding involution on $N_g$. 
Then $\iota $ is \topconj  to $\iota _{g;s,t}$. 
\end{prop}
\begin{proof}
Let $(h,\varepsilon , [(2)^{r}],\{ (-)^{k}\})$ be the signature of the NEC group $\Gamma $. 
We recall that the signature of $\iota $ is also $\sigma (\iota )=(h,\varepsilon , [(2)^{r}],\{ (-)^{k}\})$ and  
\begin{itemize}
    \item[(1)] the genus of $N_g/\iota $ is $h$, 
    \item[(2)] $N_g/\iota  $ is orientable (resp. non-orientable) if $\varepsilon =+$ (resp. $\varepsilon =-$),  
    \item[(3)] $\iota $ has $r$ isolated fixed points and $k$ \refline s by Lemma~\ref{lem_number_fixedpt}. 
\end{itemize} 
These properties are invariant under topological conjugations. 
By Proposition~\ref{prop_topconj_inv-NSK}, the topological conjugacy class of $\iota $ is determined by the topological conjugacy class of $\theta _{g;s,t}$. 
Hence, in the case that the topological conjugacy class of $\theta _{g;s,t}$ depend on only $g$ and $s$, the topological conjugacy class of $\iota $ is determined by the conditions~(1), (2), and (3) above. 
We note that the fixed points of $\iota _{g;s}$ are $r$ isolated fixed points and $k$ simple closed curves. 
Therefore, we can check that $\iota $ is \topconj to $\iota _{g;s}$ except for $(g,s)\in \{(4,2), (4,6), (4,9), (5,3)\}$. 

For the cases $(g,s)=(4,2), (4,6),$ and $(5,3)$, we see that $n=\# \{ j \in \{ 1, \ldots, k \} \mid \theta (e_j) = X \} \not=0$. 
Hence, by Proposition~\ref{thm_bems}, the topological conjugacy class of $\iota $ is determined by the number $n$. 
Then, by Lemma~\ref{theta_e_i}, $\iota $ has $n$ one-sided \refline s\ and $k-n$ two-sided \refline s. 
Thus we can check that $\iota $ is \topconj to $\iota _{g;s,t}$. 
For instance, in the case $(g,s)=(4,2)$, the signature of $\Gamma $ is $\NEC{0}{p}{2}{2}$.  
Hence the quotient space $N_4/\iota $ is a 2-sphere with two boundary components, and $\iota $ has two of each of isolated fixed points and \refline s by Lemma~\ref{lem_number_fixedpt}. 
By Table~\ref{table:NSKmap}, we have
\begin{align*}
    &\theta _{4;2,1}\colon (x_1, x_2, e_1, e_2, c_{1,0}, c_{2,0})\to (X, X, 1, 1, X, X),\\
    &\theta _{4;2,2}\colon (x_1, x_2, e_1, e_2, c_{1,0}, c_{2,0})\to (X, X, X, X, X, X).
\end{align*}
Since $n=0$ if $t=1$ and $n=2$ if $t=2$, by Lemma~\ref{theta_e_i}, $\iota $ has $0$ (resp. 2) one-sided \refline s\ and $2$ (resp. $0$) two-sided \refline s when $t=1$ (resp. $t=2$). 
Thus, $\iota $ is \topconj to $\iota _{4;2,1}$ if $t=1$ and to $\iota _{4;2,2}$ if $t=2$. 

In the case $(g,s)=(4,9)$, the signature of $\Gamma $ is $\NEC{3}{m}{0}{0}$ and $\Gamma $ is generated by three glide reflections $d_1$, $d_2$, and $d_3$. 
Note that the corresponding involution $\iota $ has no fixed points. 
Hence this case holds the assumption of the condition~(2) in Proposition~\ref{thm_bems}, so the topological conjugacy class of $\iota $ is determined by the number $m=\# \{ j \in \{ 1, 2, 3 \} \mid \theta (d_j) = X \}$ and whether $m$ is congruent to 0 modulo 2 or not. 

By Table~\ref{table:NSKmap}, $m=1$ if $t=1$ and $m=2$ if $t=3$. 
For the case that $t=1$, we see that $\theta _{4;9,1}(d_1)=X$ and $\theta _{4;9,1}(d_i)=1$ for $i=2,3$. 
Hence the glide reflections $d_i$ and $d_1d_id_1^{-1}$ for $i=2,3$ are in $\ker\theta_{4;9,1}$.
Then, each of the glide reflections fixes a geodesic of $\mathbb{H}^2$, and the neighborhoods of the images of the geodesics are crosscaps of $N_4$.
Note that, by $d_1 \notin \ker\theta_{4;9,1}$, the glide reflection $d_1$ induces the involution $\iota $.
So, the image of the geodesic fixed by $d_1$ is a separating \refline, and $\iota$ swaps two of the four crosscaps with the other two. 
Hence $\iota $ is \topconj to $\iota _{4;9,1}$. 

For the case that $t=3$, we have $\theta _{4;9,3}(d_i)=X$ for $i=1,3$ and $\theta _{4;9,3}(d_2)=1$. 
Then for a fundamental domain $\mathcal{D}$ of $\Gamma $ as on the left-hand side in Figure~\ref{figure_iota493}, a fundamental domain $\widetilde{\mathcal{D}}$ of $\ker \theta _{4;9,3}\cong \pi _1(N_4)$ is obtained by $\widetilde{\mathcal{D}}=\mathcal{D}\cup d_3(\mathcal{D})$ as on the right-hand side in Figure~\ref{figure_iota493}. 
We can see that the boundary of $\widetilde{\mathcal{D}}$ is identified in the quotient $\widetilde{\mathcal{D}}/\ker \theta _{4;9,3}\approx N_4$ as on the right-hand side in Figure~\ref{figure_iota493} and the quotient $\widetilde{\mathcal{D}}/\ker \theta _{4;9,3}$ is homeomorphic to the surface $N$ as in Figure~\ref{figure_iota493}. 
The involution $\iota $ corresponding to $\theta _{4;9,3}$ is induced by $d_3$ and $\iota $ coincides with the homeomorphism which is induced by the antipodal action on $N$. 
Thus $\iota $ is \topconj to $\iota _{4;9,3}$. 
Therefore, the involution corresponding to $\theta _{g;s,t}$ in Table~\ref{table:NSKmap} is \topconj to $\iota _{g;s,t}$ and we have completed the proof of Proposition~\ref{prop_topconj}. 
\end{proof}

\begin{figure}[ht]
\includegraphics[scale=0.86]{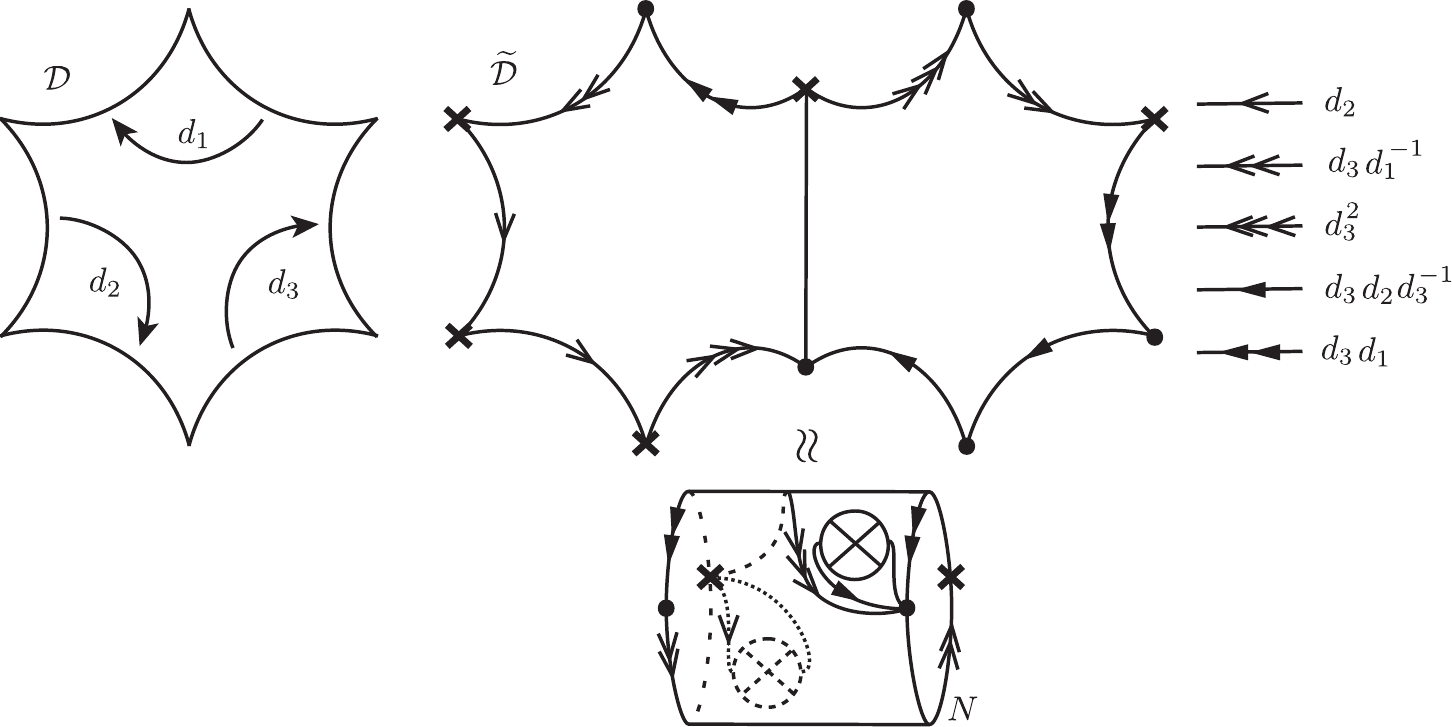}
\caption{Fundamental domains of $\Gamma $ and $\ker \theta _{4;9,3}$.}\label{figure_iota493}
\end{figure}

By Evans-Kolbe's result in~\cite{Evans-Kolbe}, any topological conjugacy classes of involutions on $N_g$ are represented by isometries, and by an argument in Section~\ref{section_NEC-group}, an involution on $N_g$ which is realized as an isometry is corresponding to an NSK-map of genus $g$. 
Hence Proposition~\ref{prop_topconj} implies that involutions described as in Figures~\ref{figure_involution_genus2}--\ref{figure_involution_genus5} give a list of representative of all topological conjugacy classes of involutions on $N_g$ for $2\leq g\leq 5$.

\section{Blowups on surfaces and topological conjugacy classes of involutions}\label{section_blowup}

\subsection{A relationship between signatures via blowup operations}

In this section, we observe some transitions of signatures by blowup operations on $N_g$.

We review the blowup and the blowdown on a surface. 
Let $e:\D \hookrightarrow {\rm int}~S$ be an embedding of the unit disk $\D \subset \C$ to the interior of a surface $S$. 
Put $D=e(\D )$. 
Let $S^\prime $ be the surface obtained from $S-{\rm int}~D$ by identifying $e(z)$ with $e(-z)$ for any $z\in \partial \D $. 
We call the manipulation that gives $S^\prime $ from $S$ the {\it blowup of} $S$ {\it on} $D$ (or at the center point of $D$). 
Note that the image $M\subset S^\prime$ of 
the regular neighborhood of $\partial D$ in $S-{\rm int}~D$ with respect to the blowup of $S$ on $D$ is a crosscap. Conversely, the {\it blowdown of} $S^\prime$ {\it on }$M$ is the following manipulation that gives $S$ from $S^\prime $; we paste a disk on the boundary component which is obtained by cutting $S^\prime $ along the core curve of $M$. The blowdown of $S^\prime $ on $M$ is the inverse manipulation of the blowup of $S$ on $D$.

Let $\iota $ be an involution on $N_g$ which is corresponding to an NSK-map $\theta \colon \Gamma \to \Z _2$ from an NEC group $\Gamma $ of signature $(h,\varepsilon , [(2)^{r}],\{ (-)^{k}\})$. 
By Lemma~\ref{lem_number_fixedpt}, $\iota $ has $r$ isolated fixed points and $k$ \refline s on $N_g$.
Suppose that the involution $\iota $ has a non-empty fixed point set. 
We take a fixed point $x_0\in N_g$ of $\iota $ and an embedding $e:\D \hookrightarrow N_g$ which satisfies the following conditions: 
\begin{enumerate}[label=(\arabic*)]
\item $e(0)=x_0$, and
\item for any $z\in \D $, either
\begin{enumerate}[label=(\alph*)]
\item $\iota (e(z))=e(-z)$ or
\item $\iota (e(z))=e(\overline{z})$,
\end{enumerate}
\end{enumerate}
where $\overline{z}\in \C$ is the complex conjugate of $z$.
The condition~(a) corresponds to the case where $x_0$ is an isolated fixed point of $\iota $ and the condition~(b) corresponds to the case where $x_0$ lies on a \refline \ of $\iota $.
On the upper left-hand side in Figure~\ref{figure_blowup_signature_isolpt} (resp. Figure~\ref{figure_blowup_signature_refline}), we describe the action of $\iota $ around $x_0$ when $x_0$ is an isolated fixed point of $\iota $ (resp. $x_0$ lies on a \refline \ of $\iota $).
Note that the surface $(N_g)^\prime $ which is obtained from $N_g$ by the blowup on $D:=e(\D )$ is homeomorphic to $N_{g+1}$. 
Hence we identify $(N_g)^\prime $ with $N_{g+1}$.
Then we define the self-homeomorphism $\tilde{\iota }$ on $N_{g+1}$ by $\tilde{\iota }([x]):=[\iota (x)]$ for any $x\in N_g-\mathrm{int}~D$.
By the definition and the conditions~(1) and (2) above, 
  for any $z\in \partial \D $, we can check that if $x_0$ is an isolated fixed point of $\iota $, then 
\[
[\iota (e(z))]=[e(-z)]=[e(z)]=[\iota (e(-z))]
\]
in $N_{g+1}$, and  if $x_0$ lies a \refline \ of $\iota $, then 
\[
[\iota (e(z))]=[e(\overline{z})]=[e(-\overline{z})]=[e(\overline{-z})]=[\iota (e(-z))]
\]
in $N_{g+1}$. 
Thus, $\tilde{\iota }$ is well-defined and also an involution on $N_{g+1}$. 
We call the involution $\tilde{\iota }$ the \emph{involution on $N_{g+1}$ obtained from $\iota $ on $N_g$ by the blowup of $N_g$ at the fixed point $x_0$}.
On the upper right-hand side in Figure~\ref{figure_blowup_signature_isolpt} and~\ref{figure_blowup_signature_refline}, we describe the action of $\tilde{\iota }$ around the crosscap obtained from a neighborhood of $\partial D$ in $N_g-\mathrm{int}~D$ by the blowup on $D$. 
Deep gray points represented by cross marks, and solid lines (curves) on the upper surfaces in Figure~\ref{figure_blowup_signature_isolpt} and \ref{figure_blowup_signature_refline} indicate the fixed points of $\iota $ and $\tilde{\iota }$. 
Recall that the signature of an involution is defined by the signature of the NEC group which is the domain of
the corresponding NSK-map, namely, we write $\sigma (\iota )=\sigma (\Gamma )$.  
We have the following proposition about a relationship between $\iota $ and $\tilde{\iota }$ by using the signatures.

\begin{prop}\label{prop_blowup_signature}
Let $\iota $ be an involution on $N_g$ with a signature 
\[
\sigma (\iota )=(h,\varepsilon , [(2)^{r}],\{ (-)^{k}\}), 
\]
$x_0\in N_g$ a fixed point of $\iota $, and $\tilde{\iota }$ the involution on $N_{g+1}$ obtained from $\iota $ by the blowup of $N_g$ at a fixed point $x_0$. Then the signature of $\tilde{\iota }$ is one of the following:
\begin{enumerate}[label=\textup{(\roman*)}]
\item  $\sigma (\tilde{\iota })=(h,\varepsilon , [(2)^{r-1}],\{ (-)^{k+1}\})$\quad if $x_0$ is an isolated fixed point of $\iota $,
\item  $\sigma (\tilde{\iota })=(h,\varepsilon , [(2)^{r+1}],\{ (-)^{k}\})$\quad if $x_0$ lies on a \refline \ of $\iota $.
\end{enumerate}
\end{prop}

\begin{proof}
First, we show that the number of the isolated fixed points and that of the reflection curves do not change before and after the blowup at a fixed point of an involution on $N_g$, except in a neighborhood of the fixed point.
By Lemma~\ref{lem_number_fixedpt}, the integer $r$ is the number of isolated fixed points of $\iota $ and $k$ is the number of \refline s of $\iota $. 
Let $e:\D \hookrightarrow N_g$ be an embedding of the unit disk $\D \subset \C $ such that $e$ defines $\tilde{\iota }$ and $e(0)=x_0$ as above, $D\subset N_g$ the image of $e$, and $\mu \subset N_{g+1}$ the image of $\partial D$ with respect to the blowup on $D$.
Remark that $\iota (D)=D$ and $\mu $ is a one-sided simple closed curve on $N_{g+1}$ which is a core curve of the crosscap obtained by the blowup on $D$. 
Let $p \colon N_g-\mathrm{int}~D\twoheadrightarrow N_{g+1}$ be the quotient map induced by the blowup on $D$. 
By the well-definedness of $\tilde{\iota }$, the following diagram is commutative:
\[
  \xymatrix{
    N_g-\mathrm{int}~D \ar[r]^{\iota |_{N_g-\mathrm{int}~D}} \ar@{>>}[d]_{p } & N_g-\mathrm{int}~D \ar@{>>}[d]^{p } \\
    N_{g+1} \ar[r]_{\tilde{\iota }} & N_{g+1}.
  }
\] 
The restriction of the quotient map $p $ 
to $N_g-D$ is injective.
So, the isolated fixed points (resp. the \refline s) of $\tilde{\iota }$ not intersecting with $\mu $ coincide with the image by $p $ of the isolated fixed points (resp. the \refline s) of $\iota $ not intersecting with $D$. 

Next, we show that the genus and the orientability of the signature of the involution on $N_{g+1}$ obtained by the blowup are the same as those of the original involution on $N_g$.
We denote by $\mathcal{N}\subset N_g$ a regular neighborhood of $D$ in $N_g$ such that $\iota (\mathcal{N})=\mathcal{N}$. 
The images of $\mathcal{N}\subset N_g$ in $N_g/\iota $ and $p (\mathcal{N}-\mathrm{int}~D)\subset N_{g+1}$ in $N_{g+1}/\tilde{\iota }$ are orientable surfaces of genus~0, respectively, as on the lower side in Figures~\ref{figure_blowup_signature_isolpt} and \ref{figure_blowup_signature_refline}. 
Hence $N_g/\iota $ and $N_{g+1}/\tilde{\iota }$ have the same genus and the same orientability.   

Suppose that $x_0\in N_g$ is an isolated fixed point of $\iota $, then the embedding $e$ satisfies the conditions~(1) and (2) (a) above.
Since the point $x_0$ is the unique fixed point of the restriction of $\iota $ to $D$, the restriction $\iota |_{N_g-D}$ has $r-1$ isolated fixed points and $k$ \refline s. 
For any $z\in \partial \D $, we can show that 
\[
\tilde{\iota }([e(z)])=[\iota (e(z))]=[e(-z)]=[e(z)]
\]
in $N_{g+1}$. 
Hence the restriction of $\tilde{\iota }$ to $\mu $ is the identity map on $\mu $, i.e. $\mu $ is a \refline \ of $\tilde{\iota }$. 
Therefore, the involution $\tilde{\iota }$ has $r-1$ isolated fixed points and $k+1$ \refline s, 
so the signature of $\tilde{\iota }$ is $(h,\varepsilon , [(2)^{r-1}],\{ (-)^{k+1}\})$ when $x_0$ is an isolated fixed point of $\iota $. 

Suppose that $x_0\in N_g$ lies on a \refline \ of $\iota $, then the embedding $e$ satisfies the conditions~(1) and (2) (b) above.
Since the subset $e(\{ \mathrm{Re}(z)\mid z\in \D \} )\subset N_g$ is the fixed point set of the restriction of $\iota $ to~$D$, the restriction $\iota |_{N_g-\mathrm{int}~D}$ has $r$ isolated fixed points, $k-1$ \refline s, and one properly embedded arc $\alpha$ fixed by it. 
Then, the endpoints of $\alpha$ are $e(1)$ and $e(-1)$.
For any $z\in \partial \D $, we can show that 
\[
\tilde{\iota }([e(z)])=[\iota (e(z))]=[e(\overline{z})]
\]
in $N_{g+1}$. 
Hence the fixed points of the restriction of $\tilde{\iota }$ to $\mu $ are $[e(1)]$ and $[e(\sqrt{-1})]$ since we have $\overline{\pm 1}=\pm 1$, $[e(\overline{\pm \sqrt{-1}})]=[e(\mp \sqrt{-1})]=[e(\pm \sqrt{-1})]$, and $\overline{z}\not= \pm z$ for any $z\in \partial \D -\{ \pm 1,\ \pm \sqrt{-1}\}$. 
By the definition of the arc $\alpha $, the point $[e(1)]\in N_{g+1}$ lies on the \refline \ $p (\alpha )$ of $\tilde{\iota }$ and $[e(\sqrt{-1})]$ is an isolated fixed point of $\tilde{\iota }$ (see the lower side in Figure~\ref{figure_blowup_signature_refline}). 
Thus the involution $\tilde{\iota }$ has $r+1$ isolated fixed points and $k$ \refline s.
Therefore, the signature of $\tilde{\iota }$ is $(h,\varepsilon , [(2)^{r+1}],\{ (-)^{k}\})$ when $x_0$ lies on a \refline \ of $\iota $. 
We have completed the proof of Proposition~\ref{prop_blowup_signature}.

\end{proof}

\begin{figure}[ht]
\includegraphics[scale=0.8]{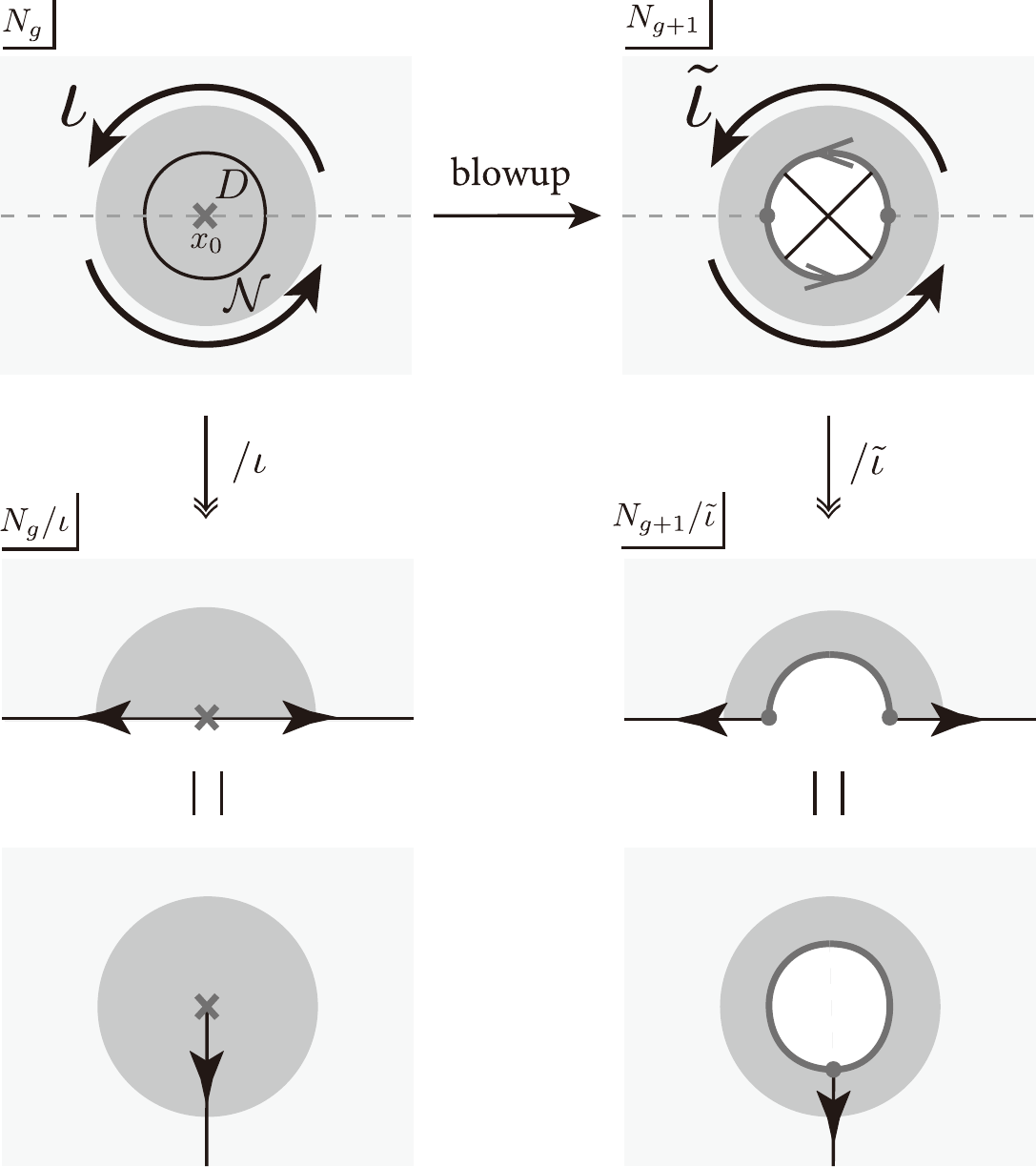}
\caption{Local descriptions of the actions of $\iota $ and $\tilde{\iota }$ around $x_0$ and the crosscap obtained by blowup on $D$, respectively, when $x_0$ is an isolated fixed point of $\iota $.}\label{figure_blowup_signature_isolpt}
\end{figure}

\begin{figure}[ht]
\includegraphics[scale=0.8]{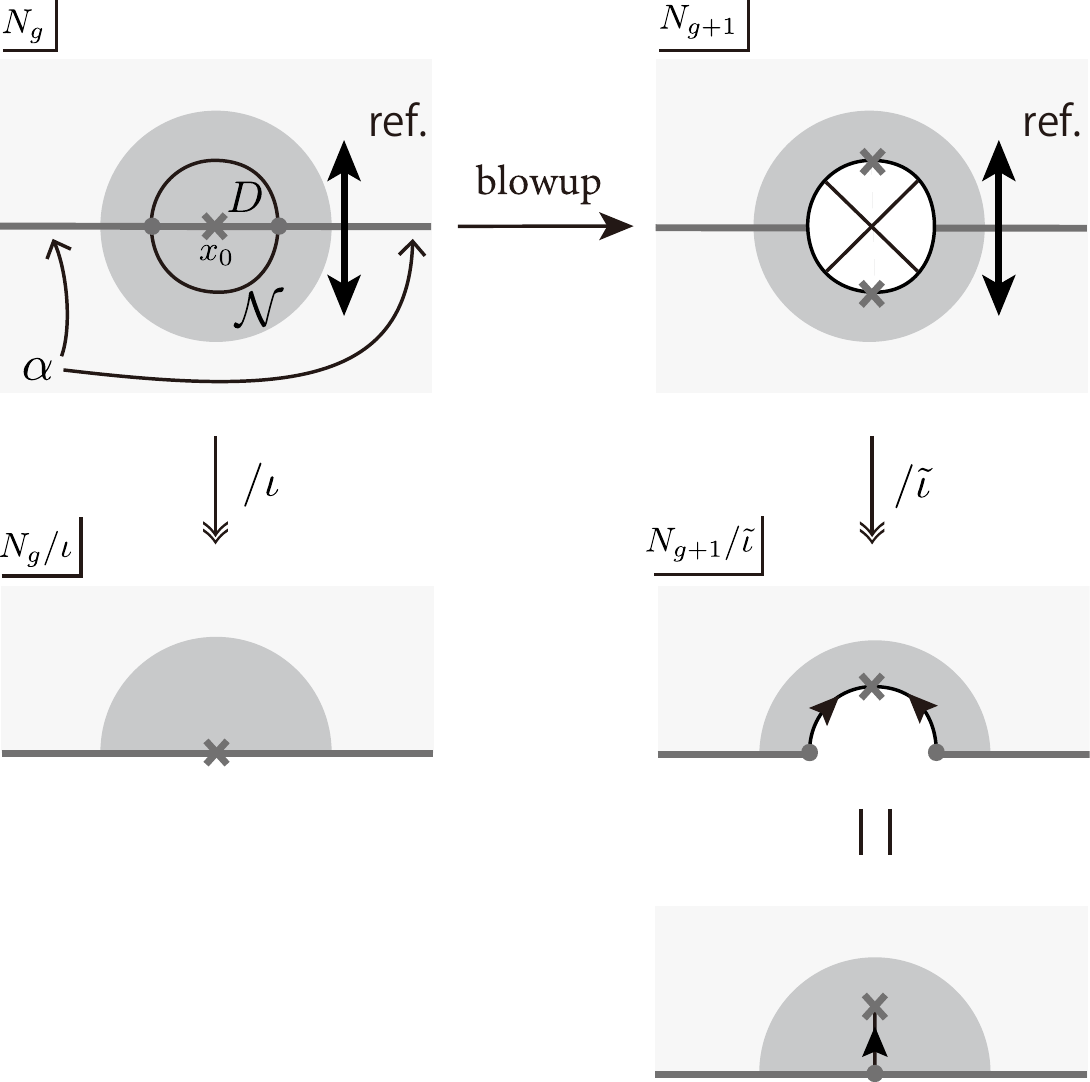}
\caption{Local descriptions of the actions of $\iota $ and $\tilde{\iota }$ around $x_0$ and the crosscap obtained by blowup on $D$, respectively, when $x_0$ lies a \refline \ of $\iota $.}\label{figure_blowup_signature_refline}
\end{figure}

By an argument in the proof of Proposition~\ref{prop_blowup_signature}, we immediately obtain the following corollary.

\begin{cor}\label{cor_blowup_signature_isol}
Let $\iota $ be an involution on $N_g$, $\tilde{\iota }$ the involution on $N_{g+1}$ which is obtained from $\iota $ by the blowup at a fixed point $x_0$ of $\iota $, and $n_1$ (resp. $n_2$) the number of one-sided (resp. two-sided) \refline s of $\iota $. 
\begin{enumerate}[label=\textup{(\arabic*)}]
    \item If $x_0$ is an isolated fixed point of $\iota $, then the number of one-sided (resp. two-sided) \refline s of $\tilde{\iota }$ is $n_1+1$ (resp. $n_2$). 
    \item If $x_0$ lies in a one-sided (resp. two-sided) \refline \ of $\iota $, then the number of one-sided (resp. two-sided) \refline s of $\tilde{\iota }$ is $n_1-1$ (resp. $n_2-1$) and the number of two-sided (resp. one-sided) \refline s of $\tilde{\iota }$ is $n_1+1$ (resp. $n_2+1$). 
\end{enumerate} 
\end{cor}

By using Proposition~\ref{prop_blowup_signature} and Corollary~\ref{cor_blowup_signature_isol}, 
we have the following proposition.

\begin{prop}\label{prop_blowup_even-odd}
Let $\tilde{\iota }$ be an involution on $N_{g+1}$ for $g\in \{ 2, 4\}$. 
Then, there exist an involution $\iota $ on $N_{g}$ and its fixed point $x_0\in N_{g}$ such that $\tilde{\iota }$ is obtained from $\iota $ by the blowup of $N_g$ at $x_0$.
\end{prop}

\begin{proof}
Let $\tilde{\iota }$ be an involution on $N_{3}$. 
By Proposition~\ref{prop_topconj} and Table~\ref{table:NSKmap}, there exists an integer $s\in \{ 1, 2, 3\}$ such that $\iota $ is \topconj to $\iota _{3;s}$. 
Signatures of all topological conjugacy classes of involutions on $N_2$ (resp. $N_3$) are listed on the left-hand (resp. right-hand) side of Figure~\ref{table_blowup_genus2-3}.  
\begin{figure}[ht]
\includegraphics[scale=1.0]{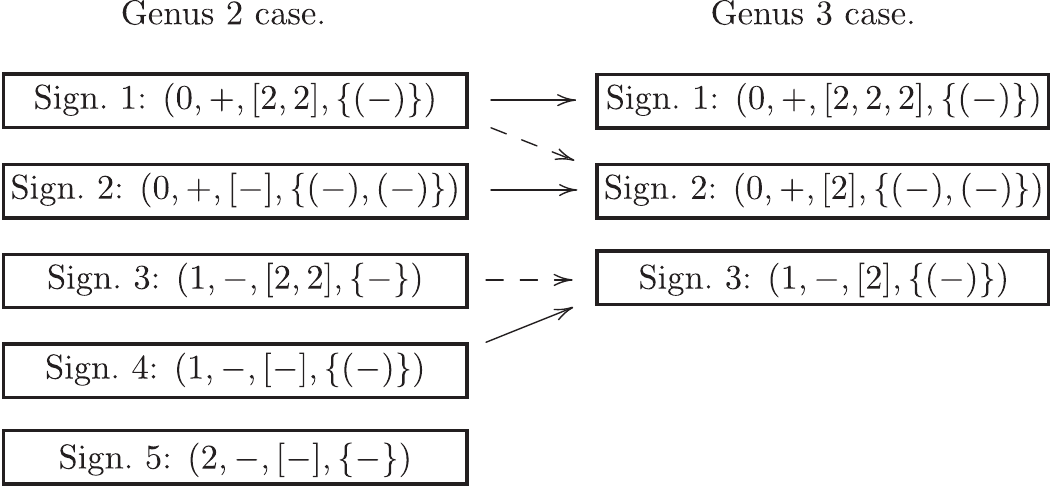}
\caption{The flow by blowups between topological conjugacy classes of involutions on $N_2$ and $N_3$.}\label{table_blowup_genus2-3}
\end{figure}
Recall that the involution $\iota _{g;s}$ on $N_g$ for $g\in \{ 2,3\}$ is defined for $s\in \{ 1,2,3,4,5\} $ when $g=2$, and for $s\in \{1,2,3\} $ when $g=3$ (see Table~\ref{table:NSKmap}). 
Since $\sigma (\iota _{2;1})=(0,+,[2,2],\{ (-)\} )$, $\sigma (\iota _{2;2})=(0,+,[-],\{ (-),(-)\})$, and $\sigma (\iota _{2;3})=(1,-,[2,2],\{ -\})$, by Proposition~\ref{prop_blowup_signature},  we have
\begin{align*}
\sigma (\widetilde{\iota _{2;1}})&=(0,+,[2,2,2],\{ (-)\} ) = \sigma(\iota_{3;1}) \  \text{ when $x_0 \in $ (\refline)},\\
\sigma (\widetilde{\iota _{2;2}})&=(0,+,[2],\{ (-),(-)\} ) = \sigma(\iota_{3;2}), \\
\sigma (\widetilde{\iota _{2;3}})&=(1,-,[2],\{ (-)\} ) = \sigma(\iota_{3;3}). 
\end{align*}
Therefore, $\widetilde{\iota _{2;s}}$ is \topconj to $\iota _{3,s}$ for $s\in \{ 1, 2, 3\}$, namely, the involution $\tilde{\iota }$ is obtained from an involution $\iota $ on $N_2$ by the blowup at a fixed point of $\iota $. 
Each arrow in Figure~\ref{table_blowup_genus2-3} shows that the signature at the terminal is obtained from that at the origin by the blowup at a fixed point.
Whether the fixed point is an isolated fixed point or contained in a reflection curve of the involution of $N_2$ is indicated by whether the arrow is dotted line or solid line.
Figure~\ref{table_blowup_genus2-3} gives the complete flow from topological conjugacy classes of involutions on $N_2$ to topological conjugacy classes of involutions on $N_3$ by blowups. 

Let $\tilde{\iota }$ be an involution on $N_{5}$. 
In this case, we also proceed by an argument similar to the genus $3$ case.  
We can see that a topological conjugacy class of an involution on $N_g$ for $g\in \{ 4,5\}$ is not determined by only the signature of the involution  (see Table~\ref{table:NSKmap}). 
All topological conjugacy classes of NSK-maps (that is corresponding to topological conjugacy classes of involutions by Propositions~\ref{prop_topconj_inv-NSK} and \ref{prop_topconj}) or corresponding signatures of genus $4$ (resp. of genus $5$) are listed on the left-hand (resp. right-hand) side of Figure~\ref{table_blowup_genus4-5}. 
Remark that we omit a symbol of NSK-map in Figure~\ref{table_blowup_genus4-5} if the signature determines the topological conjugacy class of the involution.

\begin{figure}[ht]
\includegraphics[scale=1.0]{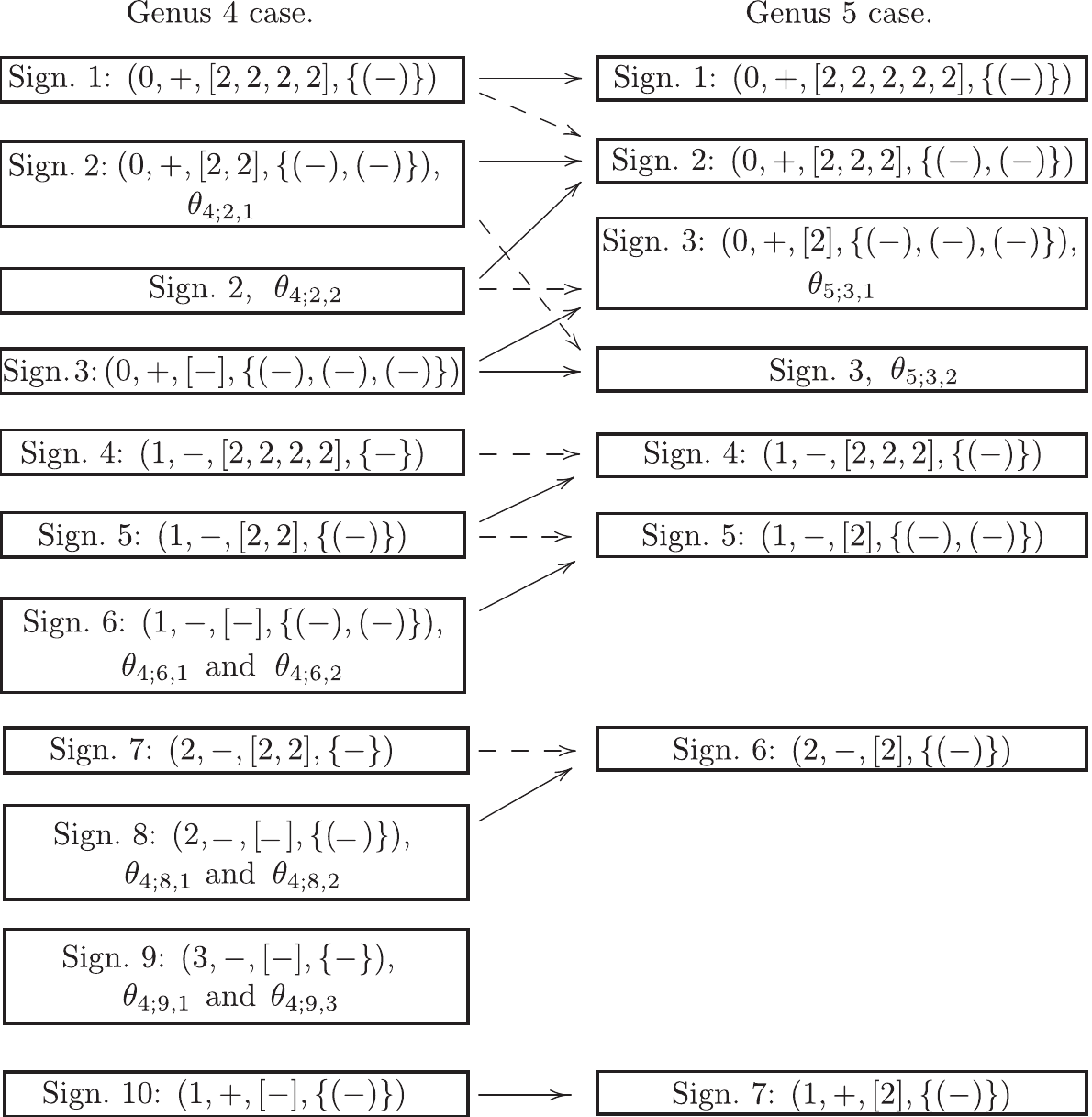}
\caption{The flow by blowups between topological conjugacy classes of involutions on $N_4$ and $N_5$. An arrow by a dotted line (resp. a solid line) indicates that an involution on $N_5$ corresponding to the signature and the NSK-map in the terminal of the arrow is obtained from an involution on $N_4$ corresponding to the signature and the NSK-map in the origin of the arrow by the blowup at an isolated fixed point (resp. a point lying in a \refline ) of the involution on $N_4$.}\label{table_blowup_genus4-5}
\end{figure}
By Proposition~\ref{prop_topconj} and Table~\ref{table:NSKmap}, there exist integers $s\in \{ 1, 2, \dots ,7\}$ and $t\in \{ 1,2,3\}$ such that $\tilde{\iota }$ is topological conjugate to $\iota _{5;s,t}$. 
Recall that we omit $t$ in the notation of $\iota _{g;s,t}$, i.e. we write $\iota _{g;s}$, if the signature determine the topological conjugacy class of the involution.
Since $\sigma (\iota _{4;1})=(0,+,[2,2,2,2],\{ (-)\} )$, $\sigma (\iota _{4;2})=(0,+,[2,2],\{ (-),(-)\})$, $\sigma (\iota _{4;5})=(1,-,[2,2],\{ (-)\})$, $\sigma (\iota _{4;7})=(2,-,[2,2],\{ -\})$, and $\sigma (\iota _{4;10})=(1,+,[-],\{ (-)\})$, by Proposition~\ref{prop_blowup_signature},  we have
\begin{align*}
\sigma (\widetilde{\iota _{4;1}})&=(0,+,[2,2,2,2,2],\{ (-)\} ) =\sigma (\iota _{5;1})\ \text{ when $x_0 \in \text{(\refline)}$},\\
\sigma (\widetilde{\iota _{4;1}})&=(0,+,[2,2,2],\{ (-),(-)\} ) =\sigma (\iota _{5;2})\ \text{ when $x_0$ : an isolated fixed point},\\
\sigma (\widetilde{\iota _{4;2,t}})&=(0,+,[2],\{ (-),(-),(-)\}) =\sigma (\iota _{5;3,t^\prime })\ \text{ when $x_0$ : an isolated fixed point,} \\
\sigma (\widetilde{\iota _{4;5}})&=(1,-,[2,2,2],\{ (-)\} ) =\sigma (\iota _{5;4})\ 
\text{ when $x_0 \in \text{(\refline)}$}, \\
\sigma (\widetilde{\iota _{4;5}})&=(1,-,[2],\{ (-),(-)\} ) =\sigma (\iota _{5;5})\ 
\text{ when $x_0$ : an isolated fixed point}, \\
\sigma (\widetilde{\iota _{4;7}})&=(2,-,[2],\{ (-)\}) =\sigma (\iota _{5;6}), \\
\sigma (\widetilde{\iota _{4;10}})&=(1,+,[2],\{ (-)\}) =\sigma (\iota _{5;7}).
\end{align*}
Topological conjugacy classes of involutions on $N_5$ are determined by their signatures, as shown in Table~\ref{table:NSKmap}, except for $\iota_{5;3,1}$ and $\iota_{5;3,2}$. 
Hence, except for the two classes, the topological conjugacy class of each involution of $N_5$ is obtained from the blowup at a fixed point of a involution of $N_4$.

In the remainder of the proof, we consider the case of $\iota_{5;3,1}$ and $\iota_{5;3,2}$.
By Proposition~\ref{prop_topconj}, these two involutions correspond to the topological conjugacy classes of NSK-maps which are represented by
\begin{align*}
\theta _{5;3,1}&\colon (x_1,e_1,e_2,e_3,c_{1,0},c_{2,0},c_{3,0})\to (X,X,X,X,X,X,X)\ \text{ and}\\
\theta _{5;3,2}&\colon (x_1,e_1,e_2,e_3,c_{1,0},c_{2,0},c_{3,0})\to (X,X,1,1,X,X,X),
\end{align*}
as shown in Table~\ref{table:NSKmap}. 
By Lemma~\ref{theta_e_i}, the involution $\iota _{5;3,1}$ has three one-sided \refline s\ and no two-sided \refline s, and $\iota _{5;3,2}$ has one one-sided \refline \ and two two-sided \refline s.  
Let $\iota $ be an involution on $N_4$ of signature $(0,+,[2,2],\{ (-),(-)\} )$. 
Then $\iota $ is \topconj to either $\iota _{4;2,1}$ or $\iota _{4;2,2}$ by Proposition~\ref{prop_topconj}. 
As shown in Table~\ref{table:NSKmap}, the NSK-maps corresponding to $\iota _{4;2,1}$ and $\iota _{4;2,2}$ are given by the following:
\begin{align*}
\theta _{4;2,1}&\colon (x_1,x_2,e_1,e_2,c_{1,0},c_{2,0})\to (X,X,1,1,X,X),\\
\theta _{4;2,2}&\colon (x_1,x_2,e_1,e_2,c_{1,0},c_{2,0})\to (X,X,X,X,X,X).
\end{align*}
By Lemma~\ref{theta_e_i}, the involution $\iota _{4;2,1}$ has no one-sided \refline s\ and two two-sided \refline s, and $\iota _{4;2,2}$ has two one-sided \refline s\ and no two-sided \refline s.  
Denote by $\tilde{\iota }_i$ $(i=1,2)$ the involution on $N_5$ obtained from $\iota _{4;2,i}$ by the blowup at an isolated fixed point of $\iota _{4;2,i}$. 
By Proposition~\ref{prop_blowup_signature}, each involution $\tilde{\iota }_i$ has the same signature of $\iota_{5;3,1}$ and $\iota_{5;3,2}$.
By Corollary~\ref{cor_blowup_signature_isol}~(1), the involution $\tilde{\iota }_1$ (resp. $\tilde{\iota }_2$) has one one-sided \refline \ and two two-sided \refline s (resp. three one-sided \refline s\ and no two-sided \refline s).  
Thus $\tilde{\iota }_1$ is \topconj to $\iota _{5;3,2}$ and $\tilde{\iota }_2$ is \topconj to $\iota _{5;3,1}$. 
We have completed Proposition~\ref{prop_blowup_even-odd}.
\end{proof}

\section{The proof of Theorem~\ref{main_thm}}\label{section_mainthm}

We give the proof of Theorem~\ref{main_thm} in this section. 

\subsection{Preliminaries for the proof of Theorem~\ref{main_thm}}\label{section_preliminaries}

In this section, we review definitions of the mapping class groups of surfaces, some homomorphisms among mapping class groups and related groups, and relations among Dehn twists and crosscap slides in mapping class groups due to the proof of Theorem~\ref{main_thm}.

The \textit{mapping class group} $\M (N_{g})$ of $N_{g}$ is the group of isotopy classes of self-homeomorphisms on $N_{g}$.  
Let $x_0$ be a point of $N_{g}$ and $\M (N_g,x_0)$ the group of isotopy classes of self-homeomorphisms on $N_g$ fixing the point $x_0$, where isotopies also fix the point $x_0$. 
For mapping classes $f=[\varphi ]$ and $g=[\psi ]$ on $N_g$, we define $gf :=[\psi \circ \varphi ]$, i.e. $f$ is applied first. 
Then the natural correspondence $\M (N_g,x_0)\ni [\varphi ]\to [\varphi ]\in \M (N_{g})$ indeuces the epimorphism 
\[
\F \colon \M (N_g,x_0) \to \M (N_g).
\]
The homomorphism $\F $ is called the \textit{forgetful homomorphism}.

Then the {\it point pushing map} 
\[
\Delta \colon \pi _1(N_{g},x_0)\rightarrow \M (N_g,x_0)
\]
is an antihomomorphism that is defined as follows. For $\gamma \in \pi _1(N_g,x_0)$, the image $\Delta (\gamma )\in \M (N_g,x_0)$ is described as the result of pushing the point $x_0$ once along $\gamma $. 
Note that for $\gamma _1$, $\gamma _2\in \pi _1(N_{g},x_0)$, the product $\gamma _1\gamma _2$ means $\gamma _1\gamma _2(t)=\gamma _1(2t)$ for $0\leq t\leq \frac{1}{2}$ and $\gamma _1\gamma _2(t)=\gamma _2(2t-1)$ for $\frac{1}{2}\leq t\leq 1$. 
The forgetful homomorphism $\F $ and the point pushing map $\Delta $ give the following short exact sequence:
\begin{eqnarray}\label{exact1}
\pi _1(N_{g},x_0) \stackrel{\Delta }{\longrightarrow }\M (N_{g},x_0)\stackrel{\mathcal{F}}{\longrightarrow }\M (N_{g})\longrightarrow 1.
\end{eqnarray}
The exact sequence~(\ref{exact1}) is called the \textit{Birman exact sequence}. 
By Murasugi~\cite{Murasugi}, the center of $\pi _1(N_{g},x_0)$ is trivial for $g\geq 3$. 
Thus, for $g\geq 3$, the homomorphism $\Delta $ is injective by Birman~\cite{Birman}*{Corollary~1.2}.

Let $e:\D \hookrightarrow N_g$ be an embedding of the unit disk $\D \subset \C $ such that $e(0)=x_0$. 
Put $D=e(\D )$. 
Then the {\it blowup homomorphism} 
\[
\Phi \colon \M (N_g,x_0)\rightarrow \M (N_{g+1})
\]
is defined as follows. 
For $h \in \mathcal{M}(N_g,x_0)$, we take a representative homeomorphism $\omega \in h$ of the mapping class $h$ which satisfies either of the following conditions: (a) $\omega |_{D}$ is the identity map on $D$, (b) $\omega (x)=e(\overline{e^{-1}(x)})$ for $x\in D$, where $\overline{e^{-1}(x)}$ is the complex conjugate of $e^{-1}(x)\in \C $. Such $\omega $ is compatible with the blowup of $N_g$ on $D$, thus $\Phi (h)\in \M (N_{g})$ is induced and well-defined (c.f. \cite{Szepietowski1}*{Section~2.3}). 
Following Szepietowski~\cite{Szepietowski1}, we define the composition: 
\[
\Psi =\Phi \circ \Delta \colon \pi _1(N_g, x_0)\rightarrow \M (N_{g+1}).
\]
Under the condition above, the next two lemmas follow from the description of the point pushing map (see \cite{Korkmaz2}*{Lemma~2.2 and Lemma~2.3}). 

\begin{lem}\label{pushing1}
Let $\gamma \in \pi _1(N_g,x_0)$ be a homotopy class of a one-sided simple loop on $N_g$ and $\mu $ a one-sided simple closed curve on $N_{g+1}$ which is obtained from $\partial D$ by the blowup of $N_g$ on $D$. 
Denote by $\alpha $ the two-sided simple closed curve on $N_{g+1}$ which is obtained from $\gamma $ by the blowup of $N_g$ on $D$. Then we have
\[
\Psi (\gamma )=Y_{\mu ,\alpha }.
\]
\end{lem}

\begin{lem}\label{pushing2}
Let $\gamma \in \pi _1(N_g,x_0)$ be a homotopy class of a two-sided simple loop on $N_g$ and $\bar{\delta}_1$ (resp. $\bar{\delta}_2$) a parallel curve of $\gamma $ on the right-hand (resp. left-hand) side of $\gamma $ with respect to an orientation of the regular neighborhood of $\gamma $ in $N_g$. 
Denote by $\delta _i$ $(i=1,2)$ the simple closed curve on $N_{g+1}$ which is obtained from $\bar{\delta}_i$ by the blowup of $N_g$ on $D$ (see Figure~\ref{crosscap_pushing_twist2}). 
We give the positive direction of $t_{\delta _i}$ $(i=1,2)$ which is induced from the orientation of the regular neighborhood of $\gamma $ in $N_g$. 
Then we have 
\[
\Psi (\gamma )=t_{\delta _1}t_{\delta _2}^{-1}.
\] 
\end{lem}
\begin{figure}[ht]
\includegraphics[scale=1.0]{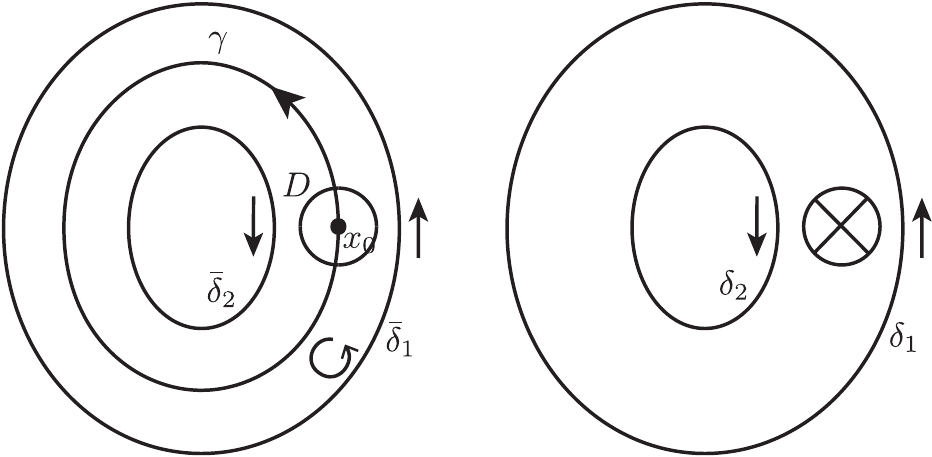}
\caption{Simple closed curves $\gamma $, $\bar{\delta }_1$, and $\bar{\delta }_2$ on $N_g$ on the left-hand side, and simple closed curves $\delta _1$ and $\delta _2$ on $N_{g+1}$ on the right-hand side.}\label{crosscap_pushing_twist2}
\end{figure}

The following lemma give well-known relations in mapping class groups.   
\begin{lem}\label{braid_original}
Let $t_\gamma $ be the Dehn twist along a two-sided simple closed curve $\gamma $ on $N_g$ and $Y_{\mu , \alpha }$ the crosscap slide about a one-sided simple closed curve $\mu $ and an oriented two-sided simple closed curve $\alpha $ on $N_g$. 
For any self-homeomorphism $\varphi $ on $N_g$, we have the following relations.
 \[
\left\{ \begin{array}{lll}
\mathrm{(i)} &\varphi t_{c}\varphi ^{-1} = t_{\varphi (c)}^\varepsilon ,\\
\mathrm{(ii)} &\varphi Y_{\mu ,\alpha }\varphi ^{-1} = Y_{\varphi (\mu ),\varphi (\alpha )},
 \end{array} \right.
 \]
where $\varepsilon =1$ if $\varphi $ preserves the orientations of regular neighborhoods of $\gamma $ and $\varphi (\gamma )$, and $\varepsilon =-1$ for the other case.
\end{lem}
The next lemma follows from Lemma~\ref{braid_original} (see for instance~\cite{Farb-Margalit}).
\begin{lem}\label{braid_rel}
Let $t_{\gamma _i}$ $(i=1,2)$ be the Dehn twist along a two-sided simple closed curve $\gamma _i$ on $N_g$, and $Y_{\mu , \alpha }$ and $Y_{\mu ^\prime , \alpha ^\prime }$ the crosscap slides for one-sided simple closed curves $\mu $ and $\mu ^\prime $, and oriented two-sided simple closed curves $\alpha $ and $\alpha ^\prime $ on $N_g$. 
When simple closed curves $\gamma _1$ and $\gamma _2$ and unions $\mu \cup \alpha $ and $\mu ^\prime \cup \alpha ^\prime $ are mutually disjoint, we have the following relations.
\begin{enumerate}[label=\textup{(\arabic*)},start=0]
\item $t_{\alpha }Y_{\mu , \alpha }=Y_{\mu , \alpha }t_{\alpha }^{-1}$,
\item $t_{\gamma _1}t_{\gamma _2}=t_{\gamma _2}t_{\gamma _1}$, 
\item $Y_{\mu , \alpha }Y_{\mu ^\prime , \alpha ^\prime }=Y_{\mu ^\prime , \alpha ^\prime }Y_{\mu , \alpha }$, 
\item $t_{\gamma _1}Y_{\mu , \alpha }=Y_{\mu , \alpha }t_{\gamma _1}$. 
\end{enumerate}
When $\gamma _1$ and $\gamma _2$ intersect transversely at one point, we have the relation
\begin{enumerate}[label=\textup{(\arabic*)},start=4]
\item $t_{\gamma _1}t_{\gamma _2}t_{\gamma _1}=t_{\gamma _2}t_{\gamma _1}t_{\gamma _2}$.
\end{enumerate}
\end{lem}
The relations~(1), (2), and (3) in Lemma~\ref{braid_rel} are called the \textit{disjointness relations} and the relation~(4) in Lemma~\ref{braid_rel} is called the \textit{braid relation}. 

We prepare explicit simple closed curves on $N_g$ and denote some mapping classes on $N_g$ to construct \DC s. 
Recall that $\alpha _i$ $(i=1,\dots ,g-1)$, $\beta $, and $\mu $ are simple closed curves on $N_g$ as in Figure~\ref{scc_closed_nonorisurf}. 
We take a point $x_0$ on the right-hand side of the g-th crosscap of $N_g$ as in Figure~\ref{scc_gamma_i1i2il}~and~\ref{generators_fundamental_grp}. 
For positive integers $1\leq i_1<i_2<\cdots <i_l\leq g$, we denote by $\gamma _{\{ i_1, i_2, \dots , i_l\}}=\gamma _{i_1, i_2, \dots , i_l}$ the simple closed curve on $N_g$ which passes through the $i_1$, $i_2$, $\dots $, $i_l$-th crosscaps as in Figure~\ref{scc_gamma_i1i2il}. 
We remark that $\alpha _i=\gamma _{i,i+1}$ $(i=1,\dots ,g-1)$, $\beta =\gamma _{1,2,3,4}$, and $\mu _1=\gamma _{1}$.
Then we denote by $Y_{i,j}$ the crosscap slide $Y_{\gamma _{i}, \gamma _{i,j}}$ for $1\leq i,j\leq g$. 
Recall that we express $i=t_{\alpha _i}$ $(i=1,\dots ,g-1)$, $b=t_{\beta }$, and $y=Y_{\mu _1,\alpha _1}=Y_{1,2}$.  
By a basic calculation, we have the following lemma. 
\begin{lem}\label{rel_y_ij}
The following relations hold in $\mathcal{M}(N_g,x_0)$ (also hold in $\mathcal{M}(N_g)$): 
\begin{enumerate}[label=\textup{(\arabic*)}]
\item $Y_{i,i+1}=\bigl((i-1)i\cdots 2312\bigr)y^{(-1)^{i+1}}\bigl(\bar{2}\bar{1}\bar{3}\bar{2}\cdots \bar{i}(\overline{i-1})\bigr)$ \quad for $1\leq i\leq g-1$,
\item $Y_{i+1,i}=\bigl((i-1)i\cdots 23121\bigr)y^{(-1)^{i+1}}\bigl(\bar{1}\bar{2}\bar{1}\bar{3}\bar{2}\cdots \bar{i}(\overline{i-1})\bigr)$ \quad for $1\leq i\leq g-1$.
\end{enumerate}
The following relation holds in $\mathcal{M}(N_g)$: 
\begin{enumerate}[label=\textup{(\arabic*)},start=3]
\item $Y_{g,1}=\bigl((\overline{g-1})\cdots \bar{2}\bar{1}\bigr)\bar{y}\bigl(12\cdots (g-1)\bigr)$,
\item $Y_{g,g-1}=\bigl((\overline{g-2})(\overline{g-1})\cdots \bar{2}\bar{3}\bar{1}\bar{2}\bigr)y\bigl(2132\cdots (g-1)(g-2)\bigr)$.
\end{enumerate}
\end{lem} 
\begin{figure}[ht]
\includegraphics[scale=1.2]{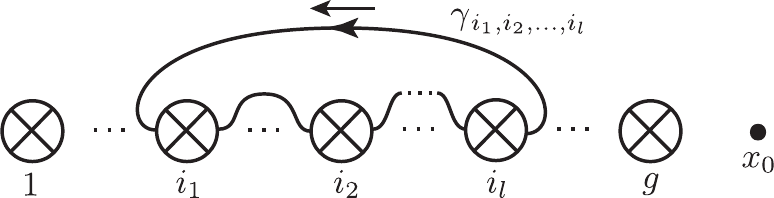}
\caption{A simple closed curve $\gamma _{\{ i_1, i_2, \dots , i_l\}}$ on $N_g$.}\label{scc_gamma_i1i2il}
\end{figure}
\begin{figure}[ht]
\includegraphics[scale=1.0]{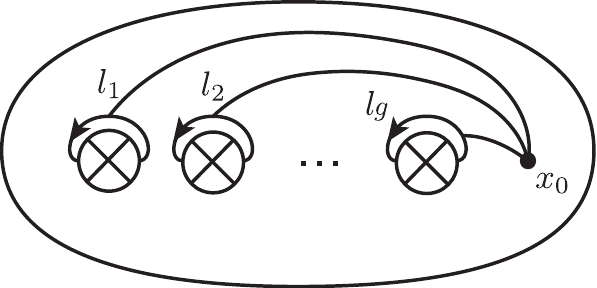}
\caption{Loops $l_{1},\ l_2, \dots ,\ l_{g}$ on $N_g$ based at $x_0$.}\label{generators_fundamental_grp}
\end{figure}

Let $l_i$ $(i=1,\ \dots ,\ g)$ be a loop on $N_g$ based at $x_0$ which passes through the $i$-th crosscap as in Figure~\ref{generators_fundamental_grp}. 
Loops $l_{1},\ l_2, \dots ,\ l_{g}$ generate $\pi _1(N_g,x_0)$. 
By Lemma~\ref{pushing1}, we have the following lemma:
\begin{lem}\label{rel_y_5i}
The relation $\Psi (l_i)=Y_{g+1,i}$ $(1\leq i\leq g)$ holds in $\mathcal{M}(N_{g+1})$. 
\end{lem} 
 
Finally, we review the Alexander method. 
Let $N$ be a closed non-orientable surface, possibly with finitely many marked points.  
A simple closed curve $\gamma $ on $N$ is \textit{essential} if $\gamma $ does not bounds a disk and a disk with one marked point. 
For a collection of essential simple closed curves $\gamma _1,\ \dots ,\ \gamma _n$ on $N$, then the collection $\gamma _1,\ \dots ,\ \gamma _n$ \textit{fills} $N$ if the surface obtained from $N$ by cutting along $\gamma _1\cup \cdots \cup \gamma _n$ is a disjoint union of disks, once-punctured disks, and M\"{o}bius bands.   
The following proposition is an analogy of Proposition~2.8 in~\cite{Farb-Margalit} for the case of non-orientable surfaces.  
\begin{prop}[Alexander method]\label{alexander_method}
Let $N$ be a compact non-orientable surface, possibly with finitely many marked points, $\varphi $ a self-homeomorphism on $N$, and $\gamma _1,\ \dots ,\ \gamma _n$ a collection of essential simple closed curves and essential proper simple arcs on $N$ with the following properties:
\begin{enumerate}[label=\textup{(\roman*)}]
\item For $1\leq i<j\leq n$, $\gamma _i$ is not isotopic to $\gamma _j$ and intersects with $\gamma _j$ at minimal intersection up to isotopy,  
\item For $1\leq i<j<k\leq n$, at least one of $\gamma _i\cap \gamma _j$, $\gamma _i\cap \gamma _k$, or $\gamma _j\cap \gamma _k$ is empty,
\item The collection $\gamma _1,\ \dots ,\ \gamma _n$ fills $N$.
\end{enumerate}
Then, the following hold. 
\begin{enumerate}[label=\textup{(\arabic*)}]
\item If $\varphi (\gamma _i)$ is isotopic to $\gamma _{i}$ for any $1\leq i\leq n$ relative to $\partial N$, then the union $\varphi (\cup_{i=1}^{n} \gamma _i)$ is isotopic to $\cup_{i=1}^{n} \gamma _i$ relative to $\partial N$. 
\end{enumerate}

We regard $\cup_{i=1}^{n} \gamma _i$ as an oriented graph $G$ whose vertices are intersection points and ends of arcs, and the orientations of edges are induced by the orientations of $\gamma _1, \dots , \gamma _n$. 
Then the composition of $\varphi $ with the isotopy in (1) gives an automorphism $\varphi _\ast $ on $G$. 
 
\begin{enumerate}[label=\textup{(\arabic*)},start=2]
\item If the automorphism $\varphi _\ast $ is the identity map on $G$ and 
$\varphi $ preserves any connected component of $N-\cup_{i=1}^{n} \gamma _i$, then $\varphi $ is isotopic to the identity relative to $\partial N$. 
\end{enumerate}
\end{prop}
Remark that the assumption that $\varphi $ preserves each component of the complement of $\cup_{i=1}^{n} \gamma _i$ in (2) is not necessary for the orientable case as in Proposition~2.8 in~\cite{Farb-Margalit}. 
Since we consider an orientation-preserving homeomorphism $\varphi $ on an oriented surface in Proposition~2.8 of \cite{Farb-Margalit}, the condition 
follows from the assumption that $\varphi _\ast $ is the identity map on the oriented graph $G$. 
We do not know whether this condition for $\varphi $ is necessary or not.

\subsection{Outline of the proof of Theorem~\ref{main_thm}}\label{section_outline}

In this section, we explain an outline of the proof of Theorem~\ref{main_thm}. 
The precise calculations for \DC s are given in Section~\ref{dc-pres_genus2-3} and \ref{dc-pres_genus4-5}. 
Recall that $\alpha _i$ $(i=1,\dots ,g-1)$, $\beta $, and $\mu _1$ are simple closed curves on $N_g$ defined as in Figure~\ref{scc_closed_nonorisurf}, and we express the homeomorphisms $t_{\alpha _i}$ $(i=1,\dots ,g-1)$, $t_\beta $, $Y_{\mu _1,\alpha _1}$, and $f^{-1}$ for any homeomorphism $f$ by $a_i=i$, $b$, $y$, and $\bar{f}$, respectively. 
By Theorem~\ref{gen_mcg}, $\M (N_g)$ is generated by $a_1,$ $\dots ,$ $a_{g-1}$, $b$, and $y$. 

We start to explain the outline of the proof of Theorem~\ref{main_thm}. 
By Proposition~\ref{prop_topconj}, any involution on $N_g$ for $2\leq g\leq 5$ is \topconj to an involution $\iota _{g;s,t}$ on a non-orientable surface $N$ described as in Figures~\ref{figure_involution_genus2}--\ref{figure_involution_genus5}, where we write $\iota _{g;s,t}=\iota _{g;s}$ when the topological conjugacy class is determined by the signature~$s$. 
We take a homeomorphism $F\colon N\to N_g$ and simple closed curves $\gamma _1^\prime ,\ \dots ,\ \gamma _n^\prime $ on $N$ which satisfies the conditions~(i)--(iii) in Proposition~\ref{alexander_method} and $\iota _{g;s,t}(\gamma _1^\prime \cup \cdots \cup \gamma _n^\prime )=\gamma _1^\prime \cup \cdots \cup \gamma _n^\prime $. 
Then there exists an element $\sigma $ of the symmetric group $S_n$ of degree $n$ and $\varepsilon _j\in \{ -1,1\}$ for $1\leq j\leq n$ such that $\iota _{g;s,t}(\gamma _j^\prime )={(\gamma _{\sigma (j)}^\prime )}^{\varepsilon _j}$. 
Put $\gamma _j=F(\gamma _j^\prime )\subset N_g$ for $1\leq j\leq n$. 
By Proposition~\ref{alexander_method}, a self-homeomorphism $\varphi $ on $N_g$ which satisfies $\varphi (\gamma _j)=\gamma _{\sigma (j)}^{\varepsilon _j}$ for any $1\leq j\leq n$ and preserves each component of $N_g-\gamma _1\cup \cdots \cup \gamma _n$ is determined up to isotopies, and then $\varphi $ is isotopic to $F\circ \iota _{g;s,t}\circ F^{-1}$. 
Thus if we construct a product $w$ of $a_1,\ \dots ,\ a_{g-1}$, $y$, and $b$ %, and their inverses
 such that $w(\gamma _j)=\gamma _{\sigma (j)}^{\varepsilon _j}$ for any $1\leq j\leq n$ up to isotopies and $w$ preserves each component of $N_g-\gamma _1\cup \cdots \cup \gamma _n$, then we have a \DC \ $w$ for $\iota _{g;s,t}$ with respect to generators for $\M (N_g)$ in Theorem~\ref{gen_mcg}. 

Due to simplify the calculations of \DC s, we use the fact by Proposition~\ref{prop_blowup_even-odd} that any involution on $N_{g+1}$ for $g\in \{2, 4\}$ is obtained from an involution on $N_{g}$ by the blowup of $N_{g}$ at one of its fixed points. 
Fix $g\in \{ 2,4\}$ and assume that the involution $\iota _{g;s,t}$ on $N$ has a fixed point. 
We take a point $x_0$ on the right-hand side of the $g$-th crosscap of $N_g$ as in Figures~\ref{scc_gamma_i1i2il} and~\ref{generators_fundamental_grp}, and a homeomorphism $F\colon N\to N_g$ such that the image of a fixed point by $F$ coincides with $x_0$. 
Set $\iota =F\circ \iota _{g;s,t}\circ F^{-1}\colon N_g\to N_g$, and then the point $x_0$ is a fixed point of $\iota $. 
Since the group $\M (N_g, x_0)$ is generated by $a_i$ $(i=1,\dots ,g-1)$, $y$, $b$ (for $g\geq 4$), and $\Delta (l_i)$ $(i=1,\dots ,g)$ by the exact sequence~(\ref{exact1}) and Theorem~\ref{gen_mcg},  there exist a product $w$ of $a_i$ $(i=1,\dots ,g-1)$, $y$, $b\in \M (N_g,x_0)$ and an element $l$ in $\pi _1(N_g,x_0)$ such that $\iota =\Delta (l)w$ in $\M (N_g,x_0)$. 
We construct such an explicit product $w$ and an explicit element $l\in \pi _1(N_g,x_0)$. 
To prove $\iota =\Delta (l)w$ in $\M (N_g,x_0)$, we check $\iota (\gamma _j)=\Delta (l)w(\gamma _j)$ relative to $x_0$. 
By using the forgetful homomorphism $\F \colon \M (N_g,x_0) \to \M (N_g)$, we have a \DC \ $\iota =\F (\Delta (l)w)=w$ on $N_g$. 
Let $\tilde{\iota }$ be the involution on $N_{g+1}\approx (N_g)^\prime $ which is obtained from $\iota $ by the blowup at $x_0$, and $\Phi \colon \M (N_g,x_0)\rightarrow \M (N_{g+1})$ the blowup homomorphism induced by the blowup at $x_0$. 
Then we have 
\[
\tilde{\iota }=\Phi (\iota )=\Phi (\Delta (l)w)=\Phi (\Delta (l))w.
\]
Thus we obtain a \DC \ of $\tilde{\iota }$ if we construct an explicit product $w^\prime $ of $a_i$ $(i=1,\dots ,g)$, $y$, and $b$ such that $\Phi (\Delta (l))=w^\prime $.

For instance, we will construct a \DC \ for $\iota _{4;2,1}$, and give a \DC \ for $\iota _{5;3,2}$ from the \DC \ for $\iota _{4;2,1}$ by a blowup at a fixed point of $\iota _{4;2,1}$. 
Let $N$ be the surface which is homeomorphic to $N_4$ as on the left-hand side in Figure~\ref{representative1_genus4_sgn2-1}, and $\iota _{4;2,1}$ the involution on $N$ induced by the reflection across the $xy$-plane. 
The involution $\iota _{4;2,1}$ has two isolated fixed points, two two-sided \refline s (they are described as the gray points and curves on $N$ on the left-hand side in Figure~\ref{representative1_genus4_sgn2-1}), and no one-sided \refline s.
Let $\alpha _1^\prime $, $\alpha _2^\prime $, $\alpha _3^\prime $, and $\mu _1^\prime $ be simple closed curves on $N$ as on the right-hand side in Figure~\ref{representative1_genus4_sgn2-1}, 
 and $x_{0,2}^\prime $ an isolated fixed point of $\iota _{4;2,1}$ as in Figure~\ref{representative1_genus4_sgn2-1}. 
Recall that we took the model of $N_4$ by the surface as in Figure~\ref{scc_closed_nonorisurf} (see also Figure~\ref{scc_genus4_sgn2-1}). 
The complements $N-(\alpha _1^\prime \cup \alpha _2^\prime \cup \alpha _3^\prime \cup \mu _1^\prime )$ and $N_4-(\alpha _1\cup \alpha _2\cup \alpha _3\cup \mu _1)$ are disks, each intersection of $\alpha _i^\prime $ and $\alpha _{i+1}^\prime $, $\alpha _i$ and $\alpha _{i+1}$ $(i=1, 2)$, $\alpha _1^\prime $ and $\mu _{1}^\prime $, and $\alpha _1$ and $\mu _{1}$ consists of one point, and other pairs of simple closed curves on $N$ or $N_4$ in Figure~\ref{representative1_genus4_sgn2-1} and \ref{scc_genus4_sgn2-1} are disjoint. 
Thus there exists a homeomorphism $F\colon N\to N_4$ such that $F(\alpha _i^\prime )=\alpha _i$ $(i=1,2,3)$, $F(\mu _1^\prime )=\mu _1$, and $F(x_{0,2}^\prime )=x_{0}$, respectively. 
The composition $F\circ \iota _{4;2,1} \circ F^{-1}$ is also an involution on $N_4$ which is \topconj to $\iota _{4;2,1}$. 
We write $F\circ \iota _{4;2,1} \circ F^{-1}$ as $\iota _{4;2,1}$ for convenience. 

\begin{figure}[ht]
\includegraphics[scale=0.88]{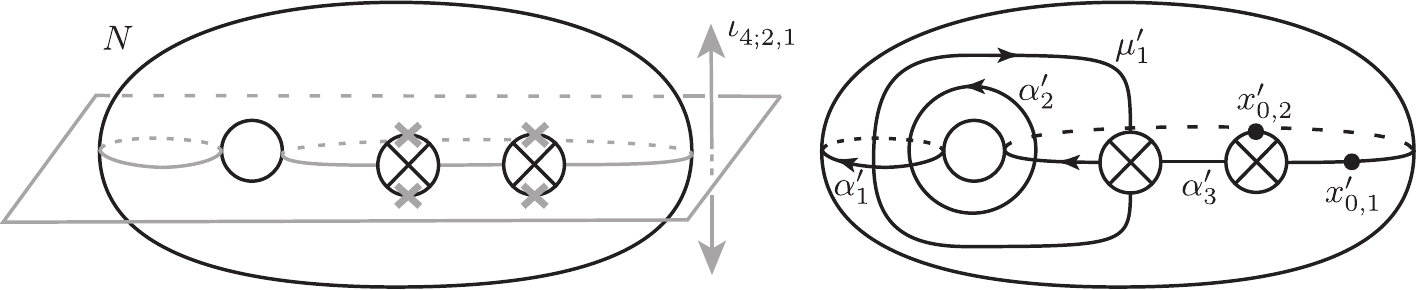}
\caption{The involution $\iota _{4;2,1}$ on $N$, simple closed curves $\alpha _i^\prime $ $(i=1,2,3)$ and $\mu _1^\prime $ on $N$, and points $x_{0,1}^\prime $ and $x_{0,2}^\prime $ of $N$.}\label{representative1_genus4_sgn2-1}
\end{figure}

\begin{figure}[ht]
\includegraphics[scale=1.0]{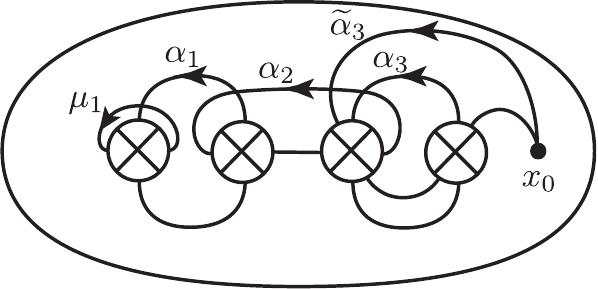}
\caption{A simple closed curve $\widetilde{\alpha}_3$ on $N_4$ based at $x_0$.}\label{scc_genus4_sgn2-1}
\end{figure}

By the definition of $\iota _{4;2,1}$, 
we have
\begin{eqnarray*}
\iota _{4;2,1}(\alpha _i)=\alpha _i \quad (i=1,3),\quad  
\iota _{4;2,1}(\alpha _2)=\alpha _2^{-1}, \quad  
\iota _{4;2,1}(\mu _1)=\mu _1^{-1}, 
\end{eqnarray*}
and $\iota _{4;2,1}(x_0)=x_0$. 
Hence we regard $\iota _{4;2,1}$ as an element of $\M (N_4,x_0)$. 
We can check that $Y_{4,3}\bar{y}(\alpha _i)=\alpha _i$ $(i=1,3)$ and $Y_{4,3}\bar{y}(\mu _1)=\mu _1^{-1}$ in $\M (N_4,x_0)$, and $Y_{4,3}\bar{y}(\alpha _2)$ is isotopic to the simple closed curve on $N_4$ as on the left-hand side in Figure~\ref{calc_genus4_sgn2-1} relative to $x_0$. 
The product $l_3l_4$ in $\pi _1(N_4,x_{0})$ is represented by the simple loop which is described as the dotted loop based at $x_{0}$ on the left-hand side in Figure~\ref{calc_genus4_sgn2-1}. 
The mapping class $\Delta (l_3l_4)\in \M (N_4,x_{0})$ fixes the union $\alpha _1\cup \alpha _3\cup \mu _1$ pointwise, and the image of $Y_{4,3}\bar{y}(\alpha _2)$ by $\Delta (l_3l_4)$ is the simple closed curve on $N_4$ as on the right-hand side in Figure~\ref{calc_genus4_sgn2-1}. 
The image $\Delta (l_3l_4)Y_{4,3}\bar{y}(\alpha _2)$ is isotopic to $\alpha _2^{-1}$ relative to $x_{0}$. 
Thus we have $\Delta (l_3l_4)Y_{4,3}\bar{y}=\iota _{4;2,1}$ in $\M (N_4,x_{0})$ by Proposition~\ref{alexander_method}. 
Using the forgetful homomorphism $\F \colon \M (N_4,x_0) \to \M (N_4)$, we have $\iota _{4;2,1}=\F (\Delta (l_3l_4)Y_{4,3}\bar{y})=Y_{4,3}\bar{y}$ in $\M (N_4)$. 
By Lemma~\ref{rel_y_ij}, the relation $Y_{4,3}=23121y\bar{1}\bar{2}\bar{1}\bar{3}\bar{2}$ holds in $\M (N_4)$. 
Thus we have
\begin{eqnarray*}
\iota _{4;2,1}&=&Y_{4,3}\bar{y}\\
&=&2\underline{31}21y\bar{1}\bar{2}\underline{\bar{1}\bar{3}}\bar{2}\bar{y} \\
&\stackrel{\text{DIS}}{=}&21321y\bar{1}\bar{2}\bar{3}\bar{1}\bar{2}\bar{y} 
\end{eqnarray*}
in $\M (N_4)$, where ``$\stackrel{\text{DIS}}{=}$'' means a deformation of an expressions by the relations~(0), (1), (2), and (3) in Lemma~\ref{braid_rel}. 
Thus we obtain the \DC \ $21321y\bar{1}\bar{2}\bar{3}\bar{1}\bar{2}\bar{y}$ for $\iota _{4;2,1}$. 

\begin{figure}[ht]
\includegraphics[scale=1.1]{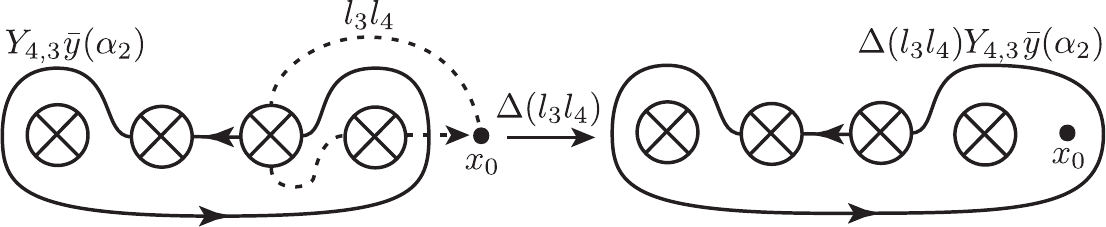}
\caption{Simple closed curves $Y_{4,3}\bar{y}(\alpha _2)$ and $\Delta (l_3l_4)Y_{4,3}\bar{y}(\alpha _2)$ on $N_4$.}\label{calc_genus4_sgn2-1}
\end{figure}

By an argument in the proof of Proposition~\ref{prop_blowup_even-odd} (see also Figure~\ref{table_blowup_genus4-5}), the involution on $N_5$ which is obtained from $\iota _{4;2,1}$ by the blowup at the isolated fixed point $x_0$ of $\iota _{4;2,1}$ is \topconj to $\iota _{5;3,2}$. 
Since $\overline{Y_{5,4}}(\alpha _3)$ is the simple closed curve on $N_5$ as in Figure~\ref{calc_genus5_sgn3-2}, by Lemma~\ref{pushing2}, the mapping class $\Phi (\Delta (l_3l_4))=\Psi (l_3l_4)\in \M (N_5)$ coincides with the product $\bar{3}\cdot \overline{Y_{5,4}}3Y_{5,4}$. 
By Lemma~\ref{rel_y_ij}, the relation $Y_{5,4}=3423121\bar{y}\bar{1}\bar{2}\bar{1}\bar{3}\bar{2}\bar{4}\bar{3}$ holds in $\M (N_5)$. 
Thus we have 
\begin{eqnarray*}
\iota _{5;3,2}&=&\Phi (\iota _{4;2,1})=\Phi (\Delta (l_3l_4)Y_{4,3}\bar{y})\\
&=&\bar{3}\cdot \overline{Y_{5,4}}\cdot 3\cdot Y_{5,4}\cdot Y_{4,3}\bar{y}\\
&=&\bar{3}\cdot 3423121y\underline{\bar{1}\bar{2}\bar{1}}\bar{3}\bar{2}\underline{\bar{4}\bar{3}\cdot 3\cdot 34}23\underline{121}\bar{y}\bar{1}\bar{2}\bar{1}\bar{3}\bar{2}\bar{4}\underline{\bar{3}\cdot 23}\ \underline{121}y\bar{1}\bar{2}\bar{1}\bar{3}\bar{2}\bar{y} \\
&\stackrel{\text{BR}}{=}&423121y\bar{2}\bar{1}\underline{\bar{2}\bar{3}\bar{2}}34\bar{3}\underline{232}12\bar{y}\bar{1}\bar{2}\bar{1}\bar{3}\bar{2}\bar{4}2312y\bar{1}\bar{2}\bar{1}\bar{3}\bar{2}\bar{y} \\
&\stackrel{\text{BR}}{=}&423121y\bar{2}\bar{1}\bar{3}\bar{2}\underline{42}312\bar{y}\bar{1}\bar{2}\bar{1}\bar{3}\underline{\bar{2}\bar{4}}2312y\bar{1}\bar{2}\bar{1}\bar{3}\bar{2}\bar{y} \\
&\stackrel{\text{DIS}}{=}&423121y\bar{2}\bar{1}\underline{\bar{3}43}12\bar{y}\bar{1}\bar{2}\bar{1}\underline{\bar{3}\bar{4}3}12y\bar{1}\bar{2}\bar{1}\bar{3}\bar{2}\bar{y} \\
&\stackrel{\text{BR}}{=}&\underline{42}\ \underline{31}21y\bar{2}\underline{\bar{1}43\bar{4}}12\bar{y}\bar{1}\bar{2}\underline{\bar{1}4\bar{3}\bar{4}}12y\bar{1}\bar{2}\bar{1}\bar{3}\bar{2}\bar{y} \\
&\stackrel{\text{DIS}}{=}&2\underline{41}321y\underline{\bar{2}4}3\underline{\bar{4}2\bar{y}\bar{1}\bar{2}}4\bar{3}\underline{\bar{4}2}y\bar{1}\bar{2}\underline{\bar{1}\bar{3}}\bar{2}\bar{y} \\
&\stackrel{\text{DIS}}{=}&214321y4\underline{\bar{2}32}\bar{y}\bar{1}\underline{\bar{2}\bar{3}2}\bar{4}y\bar{1}\bar{2}\bar{3}\bar{1}\bar{2}\bar{y} \\
&\stackrel{\text{BR}}{=}&214321y432\underline{\bar{3}\bar{y}\bar{1}}3\bar{2}\bar{3}\bar{4}y\bar{1}\bar{2}\bar{3}\bar{1}\bar{2}\bar{y} \\
&\stackrel{\text{DIS}}{=}&214321y432\underline{\bar{y}\bar{1}}\bar{2}\bar{3}\bar{4}y\bar{1}\bar{2}\bar{3}\bar{1}\bar{2}\bar{y} \\
&\stackrel{\text{DIS}}{=}&\underline{21}4321y4321\bar{y}\bar{2}\bar{3}\bar{4}y\bar{1}\bar{2}\bar{3}\bar{1}\bar{2}\bar{y} \\
&\stackrel{\text{CONJ}}{\to }&4321y4321\bar{y}\bar{2}\bar{3}\bar{4}y\bar{1}\bar{2}\bar{3}\bar{1}\bar{2}\bar{y}21, 
\end{eqnarray*}
where ``$\stackrel{\text{CONJ}}{\to }$'' and ``$\stackrel{\text{BR}}{=}$'' means deformations of expressions by a conjugation and the relation~(4) in Lemma~\ref{braid_rel}, respectively. 
Therefore, we obtain the \DC \ $4321y4321\bar{y}\bar{2}\bar{3}\bar{4}y\bar{1}\bar{2}\bar{3}\bar{1}\bar{2}\bar{y}21$ for $\iota _{5;3,2}$. 
For the other cases, we give precise calculations for the proof of Theorem~\ref{main_thm} in Sections~\ref{dc-pres_genus2-3} and \ref{dc-pres_genus4-5}. 

\begin{figure}[ht]
\includegraphics[scale=1.1]{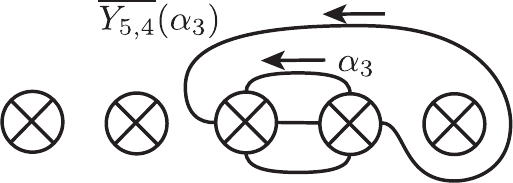}
\caption{The simple closed curve $\overline{Y_{5,4}}(\alpha _3)$ on $N_5$.}\label{calc_genus5_sgn3-2}
\end{figure}

\subsection{The genus 2 and 3 cases}\label{dc-pres_genus2-3}

In Sections~\ref{dc-pres_genus2-3} and \ref{dc-pres_genus4-5}, ``$\stackrel{\text{CONJ}}{\to }$'', ``$\stackrel{\text{(0)}}{=}$'', ``$\stackrel{\text{BR}}{=}$'', and ``$\stackrel{\text{DIS}}{=}$'' in equations mean deformations of expressions by a conjugation, the relation~(0) in Lemma~\ref{braid_rel}, the relation~(4) in Lemma~\ref{braid_rel}, and the relations~(1), (2), and (3) in Lemma~\ref{braid_rel}, respectively.  
In this section, we give calculations of \DC s for involutions on $N_g$ for $g\in \{ 2,3\}$. 
The \DC \ for $\iota _{3;1}$ in Theorem~\ref{main_thm} is given by Stukow~\cite{Stukow2}. 
We remark that the mapping class group $\M (N_2)$ is isomorphic to $\Z _2\oplus \Z _2$ generated by $t_{\alpha _1}$ and $y$ by Lickorish~\cite{Lickorish1}. 
The involution $\iota _{2;2}$ is isotopic to the identity map on $N_2$. 

\subsubsection{\DC s for $\iota _{2;1}$ and $\iota _{3;2}$}\label{section_DC_2-1_3-2}

We construct a \DC \ for $\iota _{2;1}$, and give a \DC \ for $\iota _{3;2}$ from the \DC \ for $\iota _{2;1}$ by a blowup at an isolated fixed point of $\iota _{2;1}$. 
Let $\iota _{2;1}$ be the involution on $N_2$ which is induced by the reflection across the $xy$-plane as on the left-hand side in Figure~\ref{scc_genus2_sgn1} (in this case, we can already regard $\iota _{2;1}$ as an involution on $N_2$). 
The involution $\iota _{2;1}$ has two isolated fixed points and we denote by $x_0$ the isolated fixed point of $\iota _{2;1}$ which lies in the second crosscap. 
From here, we depict only fixed points of an involution in figures which are necessary in the argument.  
Then by an isotopy on $N_2$ which fixes the simple closed curves $\alpha _1$ and $\mu _1$ as in Figure~\ref{scc_genus2_sgn1}, we move $x_0$ into the right-hand side of the second crosscap. 
We regard the resulting homeomorphism and the fixed point by the isotopy as also $\iota _{2;1}$ and $x_0$, respectively.  

\begin{figure}[ht]
\includegraphics[scale=0.85]{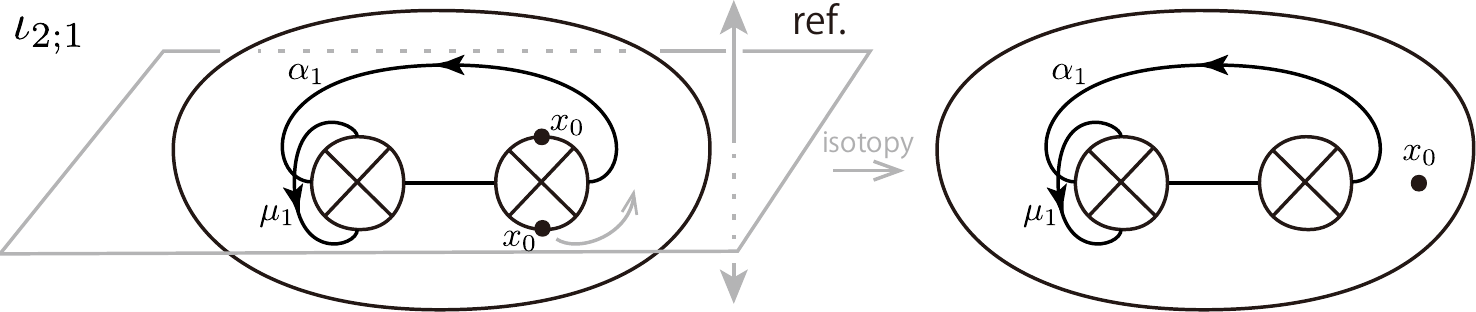}
\caption{An isotopy on $N_2$. }\label{scc_genus2_sgn1}
\end{figure}

By the definition of $\iota _{2;1}$, we have $\iota _{2;1}(\mu _1)=\mu _1^{-1}$, $\iota _{2;1}(\alpha _1)=\alpha _1$, and $\iota _{2;1}(x_0)=x_0$, hence we regard $\iota _{2;1}$ as an element of $\M (N_2,x_0)$. 
Since we can check that $y(\mu _1)=\mu _1^{-1}$ and $y(\alpha _1)=\alpha _1$ in $N_2$ relative to $x_0$, we have $y=\iota _{2;1}$ in $\M (N_2,x_{0})$ by Proposition~\ref{alexander_method}. 
Using the forgetful homomorphism, we have $\iota _{2;1}=\F (y)=y$ in $\M (N_2)$. 
Thus we obtain the \DC \ $y$ for $\iota _{2;1}$. 

By an argument in the proof of Proposition~\ref{prop_blowup_even-odd} (see also Figure~\ref{table_blowup_genus2-3}), the involution on $N_3$ which is obtained from $\iota _{2;1}$ by the blowup at the isolated fixed point $x_0$ of $\iota _{2;1}$ is \topconj to $\iota _{3;2}$. 
Thus we have $\iota _{3;2}=\Phi (\iota _{2;1})=\Phi (y)=y$. 
Therefore, we obtain the \DC \ $y$ for $\iota _{3;2}$.

\subsubsection{\DC s for $\iota _{2;3}$ and $\iota _{3;3}$}\label{section_DC_2-3_3-3}

We construct a \DC \ for $\iota _{2;3}$, and give a \DC \ for $\iota _{3;3}$ from the \DC \ for $\iota _{2;3}$ by a blowup at an isolated fixed point of $\iota _{2;3}$. 
Let $\iota _{2;3}$ be the involution on $N_2$ which is induced by the $\pi $-rotation of $N_2$ on the gray axis as in Figure~\ref{figure_involution_genus2} (in this case, we can already regard $\iota _{2;3}$ as an involution on $N_2$). 
The involution $\iota _{2;3}$ has two isolated fixed points and we denote by $x_0$ an isolated fixed point of $\iota _{2;3}$. 
Then we can regard the simple closed curves $\alpha _1$ and $\mu _1$ as simple closed curves on the complement of $x_0$. 

By the definition of $\iota _{2;3}$, we have $\iota _{2;3}(\mu _1)=\gamma _{2}$, $\iota _{2;3}(\alpha _1)=\alpha _{1}$, and $\iota _{2;3}(x_0)=x_0$. 
Hence we regard $\iota _{2;3}$ as an element of $\M (N_2,x_0)$. 
Since we can check that $1y(\mu _1)=\gamma _{\{ 2\}}$ in $\M (N_2,x_0)$, we have $1y=\iota _{2;3}$ in $\M (N_2,x_{0})$ by Proposition~\ref{alexander_method}. 
Using the forgetful homomorphism, we have $\iota _{2;3}=\F (1y)=1y$ in $\M (N_2)$. 
Thus we obtain the \DC \ $1y$ for $\iota _{2;3}$. 

By an argument in the proof of Proposition~\ref{prop_blowup_even-odd} (see also Figure~\ref{table_blowup_genus2-3}), the involution on $N_3$ which is obtained from $\iota _{2;3}$ by the blowup at the isolated fixed point $x_0$ of $\iota _{2;3}$ is \topconj to $\iota _{3;3}$. 
Thus we have $\iota _{3;3}=\Phi (\iota _{2;3})=\Phi (1y)=1y$. 
Therefore, we obtain the \DC \ $1y$ for $\iota _{3;3}$.

\subsubsection{\DC s for $\iota _{2;4}$ and $\iota _{2;5}$}\label{section_DC_2-4_2-5}

We construct \DC s for $\iota _{2;4}$ and $\iota _{2;5}$. 
Since the involution $\iota _{2;5}$ coincides with the composition $\iota _{2;2}\circ \iota _{2;4}$ and $\iota _{2;2}$ is isotopic to the identity map on $N_4$, the involutions $\iota _{2;4}$ and $\iota _{2;5}$ have the same \DC . 

The involution $\iota _{2;4}$ coincides with the composition $\iota _{2;3}\circ \iota _{2;1}$ and $y$ is an order 2 element in $\M (N_2)$. 
Thus we have $\iota _{2;4}=\iota _{2;3}\cdot \iota _{2;1}=1y\cdot y=1$. 
Therefore, we obtain the \DC \ $1$ for $\iota _{2;4}$ and $\iota _{2;5}$.

\subsection{The genus 4 and 5 cases}\label{dc-pres_genus4-5}
In this section, we give calculations of \DC s for involutions on $N_g$ for $g\in \{ 4,5\}$. 
The \DC s for $\iota _{4;1}$ and $\iota _{5;1}$ in Theorem~\ref{main_thm} are given by Stukow~\cite{Stukow}. 
The \DC s for $\iota _{4;2,1}$ and $\iota _{5;3,2}$ are given in Section~\ref{section_outline}.

\subsubsection{A \DC \ for  $\iota _{5;2}$}\label{section_DC_5-2}

We construct a \DC \ for $\iota _{5;2}$ from the \DC \ for $\iota _{4;2,1}$ by a blowup at a fixed point which lies in the \refline \ of $\iota _{4;2,1}$. 
Recall that $\iota _{4;2,1}$ is the involution on $N$ which is induced by the reflection across the $xy$-plane as on the left-hand side in Figure~\ref{representative1_genus4_sgn2-1}, $\alpha _1^\prime $, $\alpha _2^\prime $, $\alpha _3^\prime $, and $\mu _1^\prime $ are the simple closed curves on $N$ as on the right-hand side in Figure~\ref{representative1_genus4_sgn2-1}, and $x_0$ is the point in $N_4$ which lies on the right-hand side of 4-th crosscap as in Figure~\ref{scc_genus4_sgn2-1}. 
Let $x_{0,1}^\prime $ be a fixed point of $\iota _{4;2,1}$ which lies in the \refline \ $\alpha _3^\prime $ of $\iota _{4;2,1}$ as in Figure~\ref{representative1_genus4_sgn2-1} and $\widetilde{\alpha}_3$ the simple closed curve on $N_4$ based at $x_0$ as in Figure~\ref{scc_genus4_sgn2-1}. 
Then there exists a homeomorphism $F\colon N\to N_4$ such that $F(\alpha _i^\prime )=\alpha _i$ $(i=1,2)$, $F(\alpha _3^\prime )=\widetilde{\alpha}_3$, $F(\mu _1^\prime )=\mu _1$, and $F(x_{0,1}^\prime )=x_{0}$, respectively. 
The composition $F\circ \iota _{4;2,1} \circ F^{-1}$ is also an involution on $N_4$ which is \topconj to $\iota _{4;2,1}$. 
We regard $F\circ \iota _{4;2,1} \circ F^{-1}$ as $\iota _{4;2,1}$ for convenience. 

By the definition of $\iota _{4;2,1}$, 
we have
\begin{eqnarray*}
\iota _{4;2,1}(\alpha _1)=\alpha _1 ,\quad  
\iota _{4;2,1}(\alpha _2)=\alpha _2^{-1}, \quad  
\iota _{4;2,1}(\widetilde{\alpha}_3)=\widetilde{\alpha}_3,\quad  
\iota _{4;2,1}(\mu _1)=\mu _1^{-1}, 
\end{eqnarray*}
and $\iota _{4;2,1}(x_0)=x_0$. 
Hence we regard $\iota _{4;2,1}$ as an element of $\M (N_4,x_0)$. 
We can check that $Y_{4,3}\bar{y}(\alpha _1)=\alpha _1$ and $Y_{4,3}\bar{y}(\mu _1)=\mu _1^{-1}$ in $\M (N_4,x_0)$, and $Y_{4,3}\bar{y}(\alpha _2)$ and $Y_{4,3}\bar{y}(\widetilde{\alpha}_3)$ are isotopic to the simple closed curves on $N_4$ as on the left-hand side in Figure~\ref{calc_genus5_sgn2} relative to $x_0$. 
The mapping class $\Delta (l_3)\in \M (N_4,x_{0})$ fixes the union $\alpha _1\cup \mu _1$ pointwise, and the images of $Y_{4,3}\bar{y}(\alpha _2)$ and $Y_{4,3}\bar{y}(\widetilde{\alpha}_3)$ by $\Delta (l_3)$ are the simple closed curves on $N_4$ as on the right-hand side in Figure~\ref{calc_genus5_sgn2}. 
The simple closed curves $\Delta (l_3)Y_{4,3}\bar{y}(\alpha _2)$ and $\Delta (l_3)Y_{4,3}\bar{y}(\widetilde{\alpha}_3)$ are isotopic to $\alpha _2^{-1}$ and $\widetilde{\alpha}_3$ relative to $x_{0}$, respectively. 
Thus we have $\Delta (l_3)Y_{4,3}\bar{y}=\iota _{4;2,1}$ in $\M (N_4,x_{0})$ by Proposition~\ref{alexander_method}. 

\begin{figure}[ht]
\includegraphics[scale=1.1]{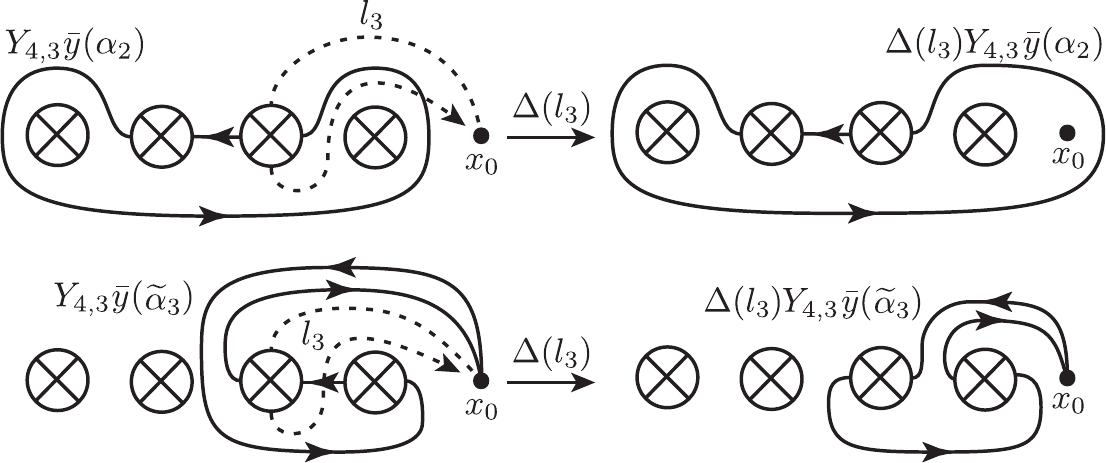}
\caption{Simple closed curves $Y_{4,3}\bar{y}(\widetilde{\alpha}_3)$, $\Delta (l_3)Y_{4,3}\bar{y}(\alpha _2)$, and $\Delta (l_3)Y_{4,3}\bar{y}(\widetilde{\alpha}_3)$ on $N_4$.}\label{calc_genus5_sgn2}
\end{figure}

By an argument in the proof of Proposition~\ref{prop_blowup_even-odd} (see also Figure~\ref{table_blowup_genus4-5}), the involution on $N_5$ which is obtained from $\iota _{4;2,1}$ by the blowup at $x_0$ is \topconj to $\iota _{5;2}$. 
By Lemma~\ref{pushing1}, $\Phi (\Delta (l_3))=\Psi (l_3)=Y_{5,3}$. 
Thus we have $[\iota _{5;2}]=\Phi (\iota _{4;2,1})=\Phi (\Delta (l_3)Y_{4,3}\bar{y})=Y_{5,3}Y_{4,3}\bar{y}$. 
Since $3234123y21(\mu _1)=\gamma _{3}$ and $3234123y21(\alpha _1)=\gamma _{3,5}^{-1}$, we have $Y_{5,3}=3234123y21\bar{y}\bar{1}\bar{2}\bar{y}\bar{3}\bar{2}\bar{1}\bar{4}\bar{3}\bar{2}\bar{3}$. 
By Lemma~\ref{rel_y_ij}, the relation $Y_{4,3}=23121y\bar{1}\bar{2}\bar{1}\bar{3}\bar{2}$ holds in $\M (N_5)$. 
Thus we have 
\begin{eqnarray*}
\iota _{5;2}&=&Y_{5,3}Y_{4,3}\bar{y}\\
&=&3234123y21\bar{y}\bar{1}\bar{2}\bar{y}\bar{3}\bar{2}\bar{1}\bar{4}\bar{3}\bar{2}\underline{\bar{3}\cdot 23}\ \underline{121}y\bar{1}\bar{2}\bar{1}\bar{3}\bar{2}\cdot \bar{y} \\
&\stackrel{\text{BR}}{=}&3234123y21\bar{y}\bar{1}\bar{2}\bar{y}\bar{3}\bar{2}\bar{1}\bar{4}\underline{\bar{3}\bar{2}23\bar{2}2}12y\bar{1}\bar{2}\bar{1}\bar{3}\bar{2}\bar{y} \\
&=&3234123y21\bar{y}\bar{1}\bar{2}\bar{y}\bar{3}\underline{\bar{2}\bar{1}\bar{4}12}y\bar{1}\bar{2}\bar{1}\bar{3}\bar{2}\bar{y} \\
&\stackrel{\text{DIS}}{=}&3234123y21\bar{y}\bar{1}\bar{2}\underline{\bar{y}\bar{3}\bar{4}y}\bar{1}\bar{2}\bar{1}\bar{3}\bar{2}\bar{y} \\
&\stackrel{\text{DIS}}{=}&\underline{3}234123y21\bar{y}\bar{1}\bar{2}\bar{3}\bar{4}\bar{1}\bar{2}\bar{1}\bar{3}\bar{2}\bar{y} \\
&\stackrel{\text{CONJ}}{\to }&234123y21\bar{y}\bar{1}\bar{2}\bar{3}\bar{4}\bar{1}\bar{2}\bar{1}\bar{3}\bar{2}\bar{y}3 \\
&\stackrel{\text{DIS}}{=}&234123y21\bar{y}\bar{1}\bar{2}\bar{3}\bar{4}\underline{\bar{1}\bar{2}\bar{1}}\ \underline{\bar{3}\bar{2}3}\bar{y} \\
&\stackrel{\text{BR}}{=}&234123y21\bar{y}\bar{1}\bar{2}\bar{3}\bar{4}\bar{2}\bar{1}\underline{\bar{2}2}\bar{3}\bar{2}\bar{y} \\
&=&23\underline{4123y}21\bar{y}\bar{1}\bar{2}\bar{3}\bar{4}\bar{2}\bar{1}\bar{3}\bar{2}\bar{y} \\
&\stackrel{\text{DIS}}{=}&2312y4321\bar{y}\bar{1}\bar{2}\bar{3}\bar{4}\bar{2}\bar{1}\bar{3}\bar{2}\bar{y} 
\end{eqnarray*}
Therefore, we obtain the \DC \ $2312y4321\bar{y}\bar{1}\bar{2}\bar{3}\bar{4}\bar{2}\bar{1}\bar{3}\bar{2}\bar{y}$ for $\iota _{5;2}$.

\subsubsection{\DC s for $\iota _{4;2,2}$ and $\iota _{5;3,1}$}\label{section_DC_4-2-2_5-3-1}

We construct a \DC \ for $\iota _{4;2,2}$, and give a \DC \ for $\iota _{5;3,1}$ from the \DC \ for $\iota _{4;2,2}$ by a blowup at an isolated fixed point of $\iota _{4;2,2}$. 
Let $N$ be the surface which is homeomorphic to $N_4$ as on the left-hand side in Figure~\ref{scc_genus4_sgn2-2} and $\iota _{4;2,2}$ the involution on $N$ which is induced by the $\pi $-rotation of $N$ on the gray axis. 
The involution $\iota _{4;2,2}$ has two isolated fixed points and two one-sided \refline s. 
Let $x_{0}^\prime $ be an isolated fixed point of $\iota _{4;2,2}$, $\alpha _1^\prime $, $\alpha _3^\prime $, $\mu _1^\prime $, and $\mu ^\prime $ simple closed curves on $N$ which are preserved by $\iota _{4;2,2}$ setwise as in Figure~\ref{scc_genus4_sgn2-2}, and $\widetilde{\alpha }_1$ and $\mu $ simple closed curves on $N_4$ as on the right-hand side in Figure~\ref{scc_genus4_sgn2-2}. 
Then there exists a homeomorphism $F\colon N\to N_4$ such that $F(\alpha _i^\prime )=\alpha _i$ $(i=1,3)$, $F(\mu _1^\prime )=\mu _1$, $F(\mu ^\prime )=\mu $, and $F(x_{0}^\prime )=x_{0}$, respectively. 
Since the composition $F\circ \iota _{4;2,2} \circ F^{-1}$ is \topconj to $\iota _{4;2,2}$, we regard $F\circ \iota _{4;2,2} \circ F^{-1}$ as $\iota _{4;2,2}$.

\begin{figure}[ht]
\includegraphics[scale=1.1]{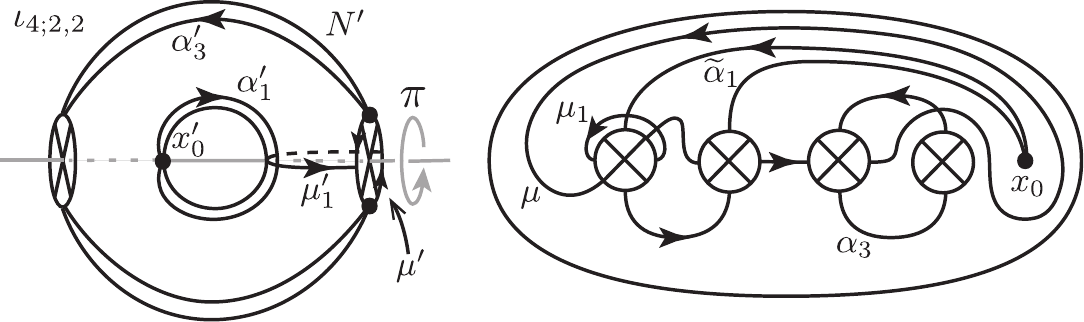}
\caption{Simple closed curves $\alpha _i^\prime $ $(i=1,3)$, $\mu _1^\prime $, and $\mu ^\prime $ on $N$, a point $x_{0}^\prime $ in $N$, and simple closed curves $\widetilde{\alpha}_1$ and $\mu $ on $N_4$.}\label{scc_genus4_sgn2-2}
\end{figure}

By the definition of $\iota _{4;2,2}$, 
we have
\begin{eqnarray*}
\iota _{4;2,2}(\widetilde{\alpha }_1)=\widetilde{\alpha }_1^{-1},\quad  
\iota _{4;2,2}(\alpha _3)=\alpha _3^{-1},\quad  
\iota _{4;2,2}(\mu _1)=\mu _1^{-1}, \quad  
\iota _{4;2,2}(\mu )=\mu , 
\end{eqnarray*}
and $\iota _{4;2,2}(x_0)=x_0$. 
Hence we regard $\iota _{4;2,2}$ as an element of $\M (N_4,x_0)$. 
Let $\alpha$ be a simple closed curve on $N_4$ as in Figure~\ref{scc_genus4_sgn2-2_2}. 
We can check that $Y_{\mu_ 1,\alpha }Y_{2,3}(\mu _1)=\mu _1^{-1}$ in $\M (N_4,x_0)$ and images of $\widetilde{\alpha }_1$, $\alpha _3$, and $\mu $ by $Y_{\mu_ 1,\alpha }Y_{2,3}$ are isotopic to simple closed curves on $N_4$ as in Figure~\ref{calc_genus4_sgn2-2} relative to $x_0$. 
The mapping class $\Delta (l_1^2l_2^2l_3l_1^2l_2)=\Delta (l_1^2l_2)\Delta (l_1^2l_2^2l_3)\in \M (N_4,x_{0})$ fixes $\mu _1$ pointwise, and the images of $Y_{\mu_ 1,\alpha }Y_{2,3}(\widetilde{\alpha }_1)$, $Y_{\mu_ 1,\alpha }Y_{2,3}(\alpha _3)$, and $Y_{\mu_ 1,\alpha }Y_{2,3}(\mu )$ by $\Delta (l_1^2l_2^2l_3l_1^2l_2)$ are isotopic to $\widetilde{\alpha }_1^{-1}$, $\alpha _3^{-1}$, and $\mu $ relative to $x_{0}$, respectively (Figure~\ref{calc_genus4_sgn2-2_2} indicates the calculation for $\widetilde{\alpha }_1$). 
Thus, by Proposition~\ref{alexander_method}, we have $\Delta (l_1^2l_2)\Delta (l_1^2l_2^2l_3)Y_{\mu_ 1,\alpha }Y_{2,3}=\iota _{4;2,2}$ in $\M (N_4,x_{0})$. 
Using the forgetful homomorphism, we have $\iota _{4;2,2}=\F (\Delta (l_1^2l_2)\Delta (l_1^2l_2^2l_3)Y_{\mu_ 1,\alpha }Y_{2,3})=Y_{\mu_ 1,\alpha }Y_{2,3}$ in $\M (N_4)$. 
The relation $Y_{2,3}=12\bar{y}\bar{2}\bar{1}$ holds in $\M (N_4,x_0)$ by Lemma~\ref{rel_y_ij}, and we can check that $Y_{\mu_ 1,\alpha }=\bar{2}y2$ in $\M (N_4,x_0)$. 
Thus we have
\begin{eqnarray*}
\iota _{4;2,1}&=&Y_{\mu_ 1,\alpha }Y_{2,3}=\bar{2}y\underline{2\cdot 12}\bar{y}\bar{2}\bar{1}\\
&\stackrel{\text{BR}}{=}&\bar{2}\underline{y1}21\bar{y}\bar{2}\bar{1}=\bar{2}\bar{1}y21\bar{y}\bar{2}\bar{1}\\ 
&\stackrel{\text{inverse}}{=}&12y\bar{1}\bar{2}\bar{y}12
\end{eqnarray*}
in $\M (N_4)$. 
Thus we obtain the \DC \ $12y\bar{1}\bar{2}\bar{y}12$ for $\iota _{4;2,2}$. 

\begin{figure}[ht]
\includegraphics[scale=1.1]{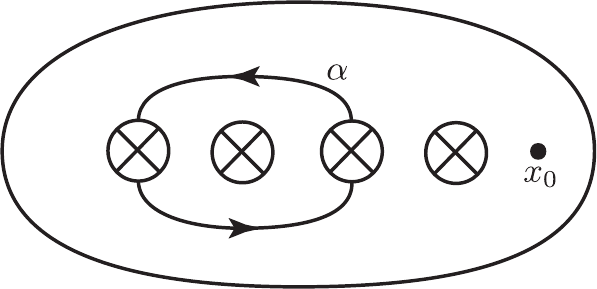}
\caption{A simple closed curve $\alpha $ on $N_4$.}\label{scc_genus4_sgn2-2_2}
\end{figure}

\begin{figure}[ht]
\includegraphics[scale=1.03]{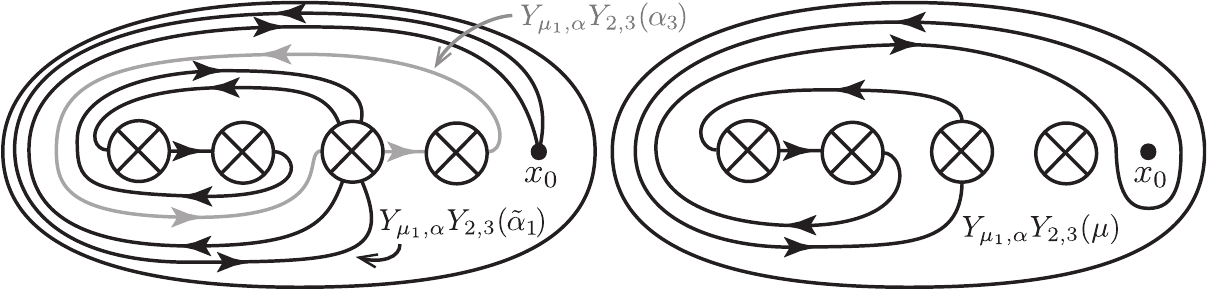}
\caption{Images of $\widetilde{\alpha }_1$, $\alpha _3$, and $\mu $ by $Y_{\mu_ 1,\alpha }Y_{2,3}$. The gray simple closed curve is $Y_{\mu_ 1,\alpha }Y_{2,3}(\alpha _3)$.}\label{calc_genus4_sgn2-2}
\end{figure}

\begin{figure}[ht]
\includegraphics[scale=1.03]{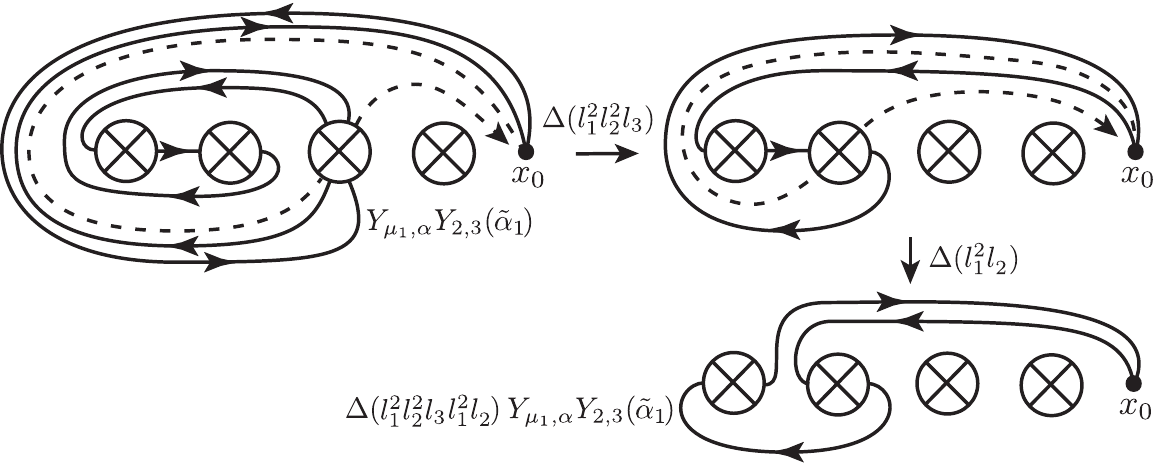}
\caption{Calculation for the image of $\widetilde{\alpha }_1$ by the mapping class $\Delta (l_1^2l_2^2l_3l_1^2l_2)Y_{\mu_ 1,\alpha }Y_{2,3}$. }\label{calc_genus4_sgn2-2_2}
\end{figure}

By an argument in the proof of Proposition~\ref{prop_blowup_even-odd} (see also Figure~\ref{table_blowup_genus4-5}), the involution on $N_5$ which is obtained from $\iota _{4;2,2}$ by the blowup at the isolated fixed point $x_0$ of $\iota _{4;2,2}$ is \topconj to $\iota _{5;3,1}$. 
Using the blowup homomorphism, we can check that $\Phi (\Delta (l_1^2l_2))=\bar{1}\bar{4}\bar{3}\bar{2}\bar{1}\bar{y}12341$ and $\Phi (\Delta (l_1^2l_2^2l_3))=\bar{2}\bar{1}\bar{4}\bar{3}\bar{2}\bar{1}\bar{y}123412$ in $\M (N_5)$. 
Thus we have 
\begin{eqnarray*}
\iota _{5;3,2}&=&\Phi (\iota _{4;2,1})=\Phi (\Delta (l_1^2l_2)\Delta (l_1^2l_2^2l_3)Y_{\mu_ 1,\alpha }Y_{2,3})\\
&=&\bar{1}\bar{4}\bar{3}\bar{2}\bar{1}\bar{y}12341\cdot \bar{2}\bar{1}\bar{4}\bar{3}\bar{2}\bar{1}\bar{y}1234\underline{12\cdot \bar{2}\bar{1}}y21\bar{y}\bar{2}\bar{1}\\
&=&\bar{1}\bar{4}\bar{3}\bar{2}\bar{1}\bar{y}1234\underline{1\bar{2}\bar{1}}\bar{4}\bar{3}\bar{2}\bar{1}\bar{y}1234y21\bar{y}\bar{2}\bar{1}\\
&\stackrel{\text{BR}}{=}&\underline{\bar{1}\bar{4}}\bar{3}\bar{2}\bar{1}\bar{y}123\underline{4\bar{2}}\bar{1}\underline{2\bar{4}}\bar{3}\bar{2}\bar{1}\bar{y}123\underline{4y21\bar{y}\bar{2}\bar{1}}\\
&\stackrel{\text{DIS}}{=}&\bar{4}\bar{1}\bar{3}\bar{2}\bar{1}\bar{y}1\underline{23\bar{2}}4\bar{1}\bar{4}\underline{2\bar{3}\bar{2}}\bar{1}\bar{y}123y21\bar{y}\bar{2}\bar{1}4\\
&\stackrel{\text{BR}}{=}&\bar{4}\bar{1}\bar{3}\bar{2}\bar{1}\bar{y}1\bar{3}234\underline{\bar{1}\bar{4}\bar{3}}\bar{2}3\bar{1}\bar{y}123y21\bar{y}\bar{2}\bar{1}4\\
&\stackrel{\text{DIS}}{=}&\bar{4}\bar{1}\bar{3}\bar{2}\bar{1}\bar{y}1\bar{3}\underline{2\bar{1}\bar{2}}3\bar{1}\bar{y}123y21\bar{y}\bar{2}\bar{1}4\\
&\stackrel{\text{BR}}{=}&\bar{4}\bar{1}\bar{3}\bar{2}\bar{1}\bar{y}1\underline{\bar{3}\bar{1}}\bar{2}\underline{13}\bar{1}\bar{y}123y21\bar{y}\bar{2}\bar{1}4\\
&\stackrel{\text{DIS}}{=}&\bar{4}\bar{1}\bar{3}\bar{2}\bar{1}\bar{y}\bar{3}\bar{2}\underline{3\bar{y}1}2\underline{3y}21\bar{y}\bar{2}\bar{1}4\\
&\overset{\text{DIS}}{\underset{\text{(0)}}{=}}&\underline{\bar{4}\bar{1}}\bar{3}\bar{2}\bar{1}\bar{y}\bar{3}\bar{2}\bar{1}\bar{y}32y321\bar{y}\bar{2}\bar{1}4\\
&\stackrel{\text{CONJ}}{\to }&\bar{3}\bar{2}\bar{1}\bar{y}\bar{3}\bar{2}\bar{1}\bar{y}32y321\bar{y}\bar{2}\bar{1}\bar{1}\\
&\stackrel{\text{inverse}}{=}&1\underline{12y\bar{1}\bar{2}\bar{3}\bar{y}\bar{2}}\bar{3}y123y123\\
&=&12y\bar{1}\bar{2}\bar{3}\bar{y}\bar{2}y123y123,
\end{eqnarray*}
where we use the relation $(2y321\bar{y}\bar{2})1(2y\bar{1}\bar{2}\bar{3}\bar{y}\bar{2})=3$ in the last equality. 
Therefore, we obtain the \DC \ $12y\bar{1}\bar{2}\bar{3}\bar{y}\bar{2}y123y123$ for $\iota _{5;3,1}$.

\subsubsection{A \DC \ for  $\iota _{4;3}$}\label{section_DC_4-3}

We construct a \DC \ for $\iota _{4;3}$. 
Let $N$ be the surface which is homeomorphic to $N_4$ as on the left-hand side in Figure~\ref{scc_genus4_sgn3} and $\iota _{4;3}$ the involution on $N$ which is induced by the reflection across the $xy$-plane. 
Let $\gamma _i^\prime $ $(i=1,2,3)$ and $\alpha _3^\prime $ be simple closed curves on $N$ which are preserved by $\iota _{4;3}$ setwise as in Figure~\ref{scc_genus4_sgn3}, and $\gamma _i$ $(i=1,2,3)$ a simple closed curve on $N_4$ as on the right-hand side in Figure~\ref{scc_genus4_sgn3}. 
Remark that $\gamma _1$ and $\gamma _1^\prime $ are one-sided simple closed curves on non-orientable surfaces, and the other simple closed curves are two-sided. 
Then there exists a homeomorphism $F\colon N\to N_4$ such that $F(\gamma _i^\prime )=\gamma _i$ $(i=1,2,3)$ and $F(\alpha _3^\prime )=\alpha _3$. 
Since the composition $F\circ \iota _{4;3} \circ F^{-1}$ is \topconj to $\iota _{4;3}$, we regard $F\circ \iota _{4;3} \circ F^{-1}$ as $\iota _{4;3}$. 

\begin{figure}[ht]
\includegraphics[scale=1.02]{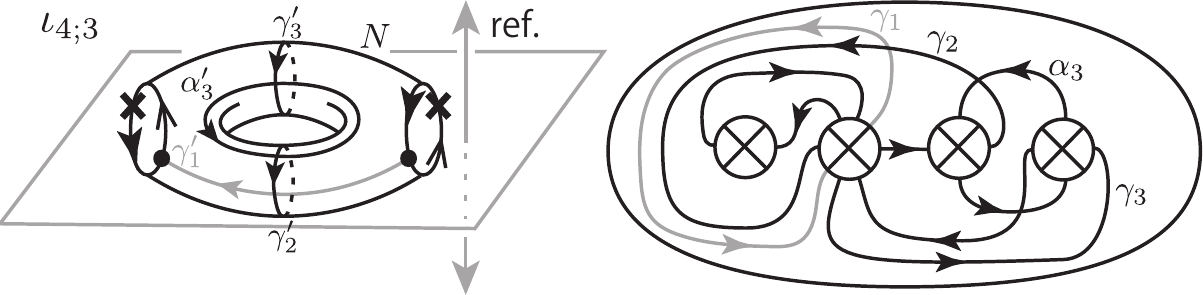}
\caption{Simple closed curves $\gamma _i^\prime $ $(i=1,2,3)$ and $\alpha _3^\prime $ on $N$, and $\gamma _i$ $(i=1,2,3)$ on $N_4$. 
The surface $N$ is obtained from the torus with two boundary components by identifying the boundary along the oriented edges as on the left-hand side. }\label{scc_genus4_sgn3}
\end{figure}

By the definition of $\iota _{4;3}$, we have
\begin{eqnarray*}
\iota _{4;3}(\gamma _1)=\gamma _1,\quad  
\iota _{4;3}(\gamma _i)=\gamma _i^{-1}\quad (i=2,3),\quad  
\iota _{4;3}(\alpha _3)=\alpha _3. 
\end{eqnarray*}
We can check that $23y(23)^2Y_{2,3}\bar{y}(\gamma _1)=\gamma _1$, $23y(23)^2Y_{2,3}\bar{y}(\gamma _i)=\gamma _i^{-1}$ $(i=2,3)$, and $23y(23)^2Y_{2,3}\bar{y}(\alpha _3)=\alpha _3$. 
Thus, by Proposition~\ref{alexander_method}, we have $23y(23)^2Y_{2,3}\bar{y}=\iota _{4;3}$ in $\M (N_4)$. 
By Lemma~\ref{rel_y_ij}, the relation $Y_{2,3}=12\bar{y}\bar{2}\bar{1}$ holds in $\M (N_4)$. 
Thus we have
\begin{eqnarray*}
\iota _{4;3}&=&23y(23)^2Y_{2,3}\bar{y}\\
&=&23y(23)^212\bar{y}\bar{2}\bar{1}\bar{y}
\end{eqnarray*}
in $\M (N_4)$. 
Thus we obtain the \DC \ $23y(23)^212\bar{y}\bar{2}\bar{1}\bar{y}$ for $\iota _{4;3}$.

\subsubsection{\DC s for $\iota _{4;4}$ and $\iota _{5;4}$}\label{section_DC_4-4_5-4}

We construct a \DC \ for $\iota _{4;4}$, and give a \DC \ for $\iota _{5;4}$ from the \DC \ for $\iota _{4;4}$ by a blowup at an isolated fixed point of $\iota _{4;4}$. 
Let $N$ be the surface which is homeomorphic to $N_4$ as on the left-hand side in Figure~\ref{scc_genus4_sgn4} and $\iota _{4;4}$ the involution on $N$ which is induced by the $\pi $-rotation of $N$ on the gray axis. 
The involution $\iota _{4;4}$ has four isolated fixed points. 
Let $x_{0}^\prime $ be an isolated fixed point of $\iota _{4;2,2}$, $\gamma ^\prime $, $\alpha _i^\prime $ $(i=2,3)$, $\mu _1^\prime $, and $\mu ^\prime $ simple closed curves on $N$ whose union is preserved by $\iota _{4;3}$ setwise as in Figure~\ref{scc_genus4_sgn4}, and $\gamma $ a simple closed curve on $N_4$ as on the right-hand side in Figure~\ref{scc_genus4_sgn4}. 
Then there exists a homeomorphism $F\colon N\to N_4$ such that $F(\gamma ^\prime )=\gamma $, $F(\alpha _i^\prime )=\alpha _i$ $(i=2,3)$, $F(\mu _1^\prime )=\mu _1$, $F(\mu ^\prime )=\gamma _{\{ 2,3,4\}}$, and $F(x_{0}^\prime )=x_{0}$, respectively. 
Since the composition $F\circ \iota _{4;4} \circ F^{-1}$ is \topconj to $\iota _{4;4}$, we regard $F\circ \iota _{4;4} \circ F^{-1}$ as $\iota _{4;4}$. 
\begin{figure}[ht]
\includegraphics[scale=1.1]{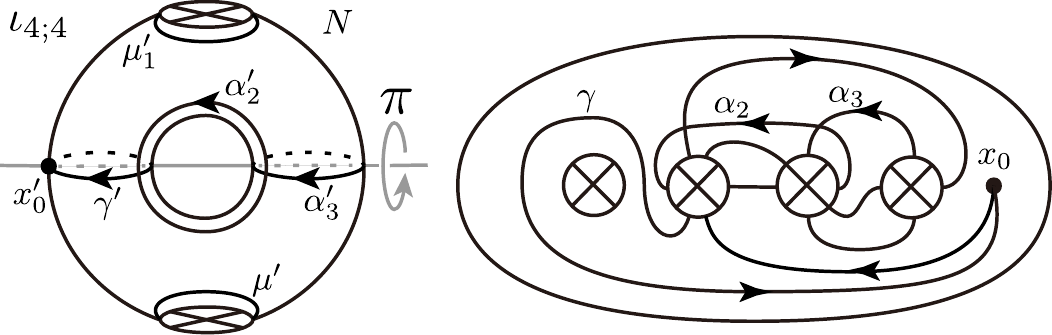}
\caption{Simple closed curves $\gamma ^\prime $, $\alpha _i^\prime $ $(i=2,3)$, $\mu _1^\prime $, and $\mu ^\prime $ and a point $x_0^\prime $ on $N$, and a simple closed curve $\gamma $ on $N_4$.}\label{scc_genus4_sgn4}
\end{figure}

By the definition of $\iota _{4;4}$,  
we have
\begin{eqnarray*}
\iota _{4;4}(\gamma )=\gamma ^{-1}\quad \text{and}\quad   
\iota _{4;4}(\alpha _i)=\alpha _i^{-1}\quad (i=2,3) \quad \text{with orientations},  
\end{eqnarray*}
$\iota _{4;4}(x_0)=x_0$, and $\iota _{4;4}(\mu _1)=\gamma _{2,3,4}$ without orientations. 
Hence we regard $\iota _{4;4}$ as an element of $\M (N_4,x_0)$. 
Let $\beta ^\prime $ be an oriented simple closed closed curves on $N_4$ as in Figure~\ref{scc_genus4_sgn4_2}. 
We can check that $bY_{\mu_ 1,\beta ^\prime }(\bar{2}\bar{3})^3(\gamma )=\gamma ^{-1}$ and $bY_{\mu_ 1,\beta ^\prime }(\bar{2}\bar{3})^3(\alpha _i)=\alpha _i^{-1}$ $(i=2,3)$ in $\M (N_4,x_0)$ and  $bY_{\mu_ 1,\beta ^\prime }(\bar{2}\bar{3})^3(\mu _1)=\gamma _{\{ 2,3,4\}}$ without orientations in $\M (N_4,x_0)$, respectively. 
Thus, by Proposition~\ref{alexander_method}, we have $bY_{\mu_ 1,\beta ^\prime }(\bar{2}\bar{3})^3=\iota _{4;4}$ in $\M (N_4,x_{0})$. 
Let $\alpha $ be a two-sided simple closed closed curve on $N_4$ whose regular neighborhood has an orientation as in Figure~\ref{scc_genus4_sgn4_2}. 
By regarding $Y_{\mu_ 1,\beta ^\prime }$ as a pushing map of first crosscap and Lemma~\ref{pushing2}, the relation $Y_{\mu_ 1,\beta ^\prime }=\overline{t_\alpha }3=\bar{2}\bar{3}\bar{y}\bar{2}y32\cdot 3$ holds in $\M (N_4,x_0)$. 
Thus we have
\begin{eqnarray*}
\iota _{4;4}&=&bY_{\mu_ 1,\beta ^\prime }(\bar{2}\bar{3})^3=b\bar{2}\bar{3}\bar{y}\bar{2}y323\cdot \underline{\bar{2}\bar{3}\bar{2}}\bar{3}\bar{2}\bar{3}\\
&\stackrel{\text{BR}}{=}&b\bar{2}\bar{3}\bar{y}\bar{2}\underline{y\bar{3}}\bar{2}\bar{3}\\ 
&\stackrel{\text{DIS}}{=}&b\bar{2}\bar{3}\bar{y}\bar{2}\bar{3}y\bar{2}\bar{3}
\end{eqnarray*}
in $\M (N_4,x_0)$. 
Using the forgetful homomorphism, 
we obtain the \DC \ $\iota _{4;4}=\F (b\bar{2}\bar{3}\bar{y}\bar{2}\bar{3}y\bar{2}\bar{3})=b\bar{2}\bar{3}\bar{y}\bar{2}\bar{3}y\bar{2}\bar{3}$ for $\iota _{4;4}$. 
\begin{figure}[ht]
\includegraphics[scale=1.1]{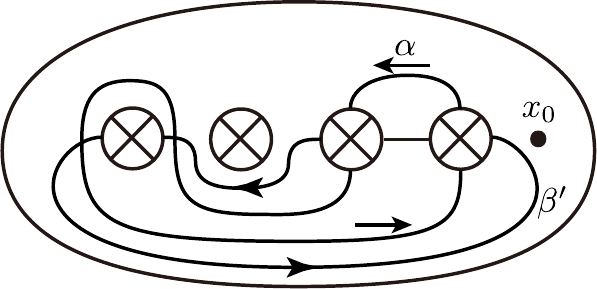}
\caption{Simple closed curves $\beta ^\prime $ and $\alpha $ on $N_4$.}\label{scc_genus4_sgn4_2}
\end{figure}

By an argument in the proof of Proposition~\ref{prop_blowup_even-odd} (see also Figure~\ref{table_blowup_genus4-5}), the involution on $N_5$ which is obtained from $\iota _{4;4}$ by the blowup at the isolated fixed point $x_0$ of $\iota _{4;4}$ is \topconj to $\iota _{5;4}$. 
Thus we have 
\begin{eqnarray*}
\iota _{5;4}&=&\Phi (\iota _{4;4})=\Phi (b\bar{2}\bar{3}\bar{y}\bar{2}\bar{3}y\bar{2}\bar{3})\\
&=&b\bar{2}\bar{3}\bar{y}\bar{2}\bar{3}y\bar{2}\bar{3}. 
\end{eqnarray*}
Therefore, we obtain the \DC \ $b\bar{2}\bar{3}\bar{y}\bar{2}\bar{3}y\bar{2}\bar{3}$ for $\iota _{5;4}$.

\subsubsection{\DC s for $\iota _{4;5}$ and $\iota _{5;5}$}\label{section_DC_4-5_5-5}

We construct a \DC \ for $\iota _{4;5}$, and give a \DC \ for $\iota _{5;5}$ from the \DC \ for $\iota _{4;5}$ by a blowup at an isolated fixed point of $\iota _{4;5}$. 
Let $N$ be the surface which is homeomorphic to $N_4$ as on the left-hand side in Figure~\ref{scc_genus4_sgn5} and $\iota _{4;5}$ the involution on $N$ which is induced by the $\pi $-rotation of $N$ on the gray axis. 
The involution $\iota _{4;5}$ has two isolated fixed points and one two-sided \refline . 
Let $x_{0}^\prime $ be an isolated fixed point of $\iota _{4;5}$, $\alpha _1^\prime $, $\mu _1^\prime $, $\beta ^\prime $, and $\gamma ^\prime $ simple closed curves on $N$ which are preserved by $\iota _{4;5}$ without orientations as on the left-hand side in Figure~\ref{scc_genus4_sgn5}, and $\gamma $ a simple closed curve on $N_4$ as on the right-hand side in Figure~\ref{scc_genus4_sgn5}. 
Then there exists a homeomorphism $F\colon N\to N_4$ such that $F(\alpha _1^\prime )=\alpha _1$, $F(\mu _1^\prime )=\mu _1$, $F(\beta ^\prime )=\beta $, $F(\gamma ^\prime )=\gamma $, and $F(x_{0}^\prime )=x_{0}$, respectively. 
Since the composition $F\circ \iota _{4;5} \circ F^{-1}$ is \topconj to $\iota _{4;5}$, we regard $F\circ \iota _{4;5} \circ F^{-1}$ as $\iota _{4;5}$. 

\begin{figure}[ht]
\includegraphics[scale=1.1]{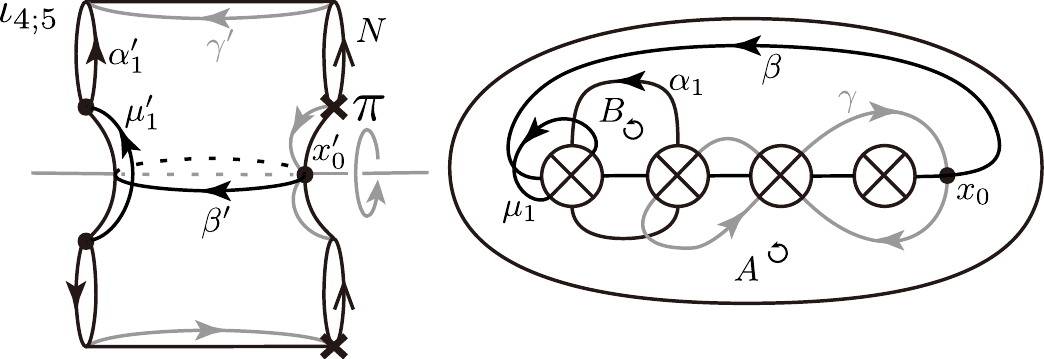}
\caption{Simple closed curves $\alpha _1^\prime $, $\mu _1^\prime $, $\beta ^\prime $, and $\gamma ^\prime $ on $N$, a point $x_{0}^\prime $ in $N$, and a simple closed curve $\gamma $ on $N_4$. 
The surface $N$ is obtained from the 2-sphere with four boundary components by identifying the boundary along the oriented edges as on the left-hand side.}\label{scc_genus4_sgn5}
\end{figure}

By the definition of $\iota _{4;5}$, we have
\begin{eqnarray*}
\iota _{4;5}(\alpha _1)=\alpha _1,\quad  
\iota _{4;5}(\mu _1)=\mu _1^{-1}, \quad  
\iota _{4;5}(\beta )=\beta ^{-1},\quad  
\iota _{4;5}(\gamma )=\gamma ^{-1} , 
\end{eqnarray*}
and $\iota _{4;5}(x_0)=x_0$. 
Hence we regard $\iota _{4;5}$ as an element of $\M (N_4,x_0)$.
We can check that $3\overline{Y_{4,3}}y(\alpha _1)=\alpha _1$ and $3\overline{Y_{4,3}}y(\mu _1)=\mu _1^{-1}$ in $\M (N_4,x_0)$, and images of $\beta $ and $\gamma $ by $3\overline{Y_{4,3}}y$ are isotopic to simple closed curves on $N_4$ as in Figure~\ref{calc_genus4_sgn5} relative to $x_0$, respectively. 
The product $l_1^2l_2^2l_3l_4$ in $\pi _1(N_4,x_{0})$ is represented by the simple loop which is described as the dotted loop based at $x_{0}$ on the left-hand side in Figure~\ref{calc_genus4_sgn5}. 
The mapping class $\Delta (l_1^2l_2^2l_3l_4)\in \M (N_4,x_{0})$ fixes $\alpha _1$ and $\mu _1$ pointwise, and the images of $3\overline{Y_{4,3}}y(\beta )$ and $3\overline{Y_{4,3}}y(\gamma )$ by $\Delta (l_1^2l_2^2l_3l_4)$ are isotopic to $\beta ^{-1}$ and $\gamma ^{-1}$ relative to $x_{0}$, respectively. 
We regard the union $\alpha _1\cup \mu_1\cup \beta \cup \gamma $ as an oriented graph $G$ on $N_4$. 
In $G$, each $e\in \{ \alpha _1, \mu _1, \beta , \gamma \}$ is separated in two edges. 
We denote by $e^1$ and $e^2$ the two edges of $G$ which are obtained from $e\in \{\alpha _1, \mu_1, \beta , \gamma \}$.  
By Proposition~\ref{alexander_method}~(1), the composition $f=\iota _{4;5}\Delta (l_1^2l_2^2l_3l_4)3\overline{Y_{4,3}}y$ induces an automorphism $f_\ast $ on $G$. 
Since each one of simple closed curves $\alpha _1$, $\mu_1$, $\beta $, and $\gamma $ is fixed by $f$ with orientations, if there exists a simple closed curve $e\in \{\alpha _1, \mu_1, \beta , \gamma \}$ such that $f_\ast (e^1)=e^2$, then $f$ switch two simple closed curves which intersect with $e$. 
Thus $f$ fixes each oriented edge of $G$ and we can assume that $f$ fixes $\alpha _1\cup \mu_1\cup \beta \cup \gamma $ pointwise.  
The complement of $\alpha _1\cup \mu_1\cup \beta \cup \gamma $ in $N_4$ is a disjoint union of two disks $A$ and $B$ as in Figure~\ref{complement_disks_genus4_sgn5}. 
Since $f$ fixes the boundary of the disks $A$ and $B$ pointwise and the edge $\beta ^1$ lies in the disk $A$ and does not lie in the disk $B$, $f$ also fixes $A$ and $B$, respectively. 
By Proposition~\ref{alexander_method}, we have $\Delta (l_1^2l_2^2l_3l_4)3\overline{Y_{4,3}}y=\iota _{4;5}$ in $\M (N_4,x_{0})$. 
Using the forgetful homomorphism, we have $\iota _{4;5}=\F (\Delta (l_1^2l_2^2l_3l_4)3\overline{Y_{4,3}}y)=3\overline{Y_{4,3}}y$ in $\M (N_4)$. 
The relation $\overline{Y_{4,3}}=23121\bar{y}\bar{1}\bar{2}\bar{1}\bar{3}\bar{2}$ holds in $\M (N_4,x_0)$ by Lemma~\ref{rel_y_ij}. 
Thus we have
\begin{eqnarray*}
\iota _{4;5}&=&3\overline{Y_{4,3}}y=3\cdot 2\underline{31}21\bar{y}\bar{1}\bar{2}\underline{\bar{1}\bar{3}}\bar{2}\cdot y\\
&\stackrel{\text{DIS}}{=}&321321\bar{y}\bar{1}\bar{2}\bar{3}\bar{1}\bar{2}\underline{y}\\ 
&\stackrel{\text{CONJ}}{\to }&y(321)^2\bar{y}\bar{1}\bar{2}\bar{3}\bar{1}\bar{2}
\end{eqnarray*}
in $\M (N_4)$. 
Thus we obtain the \DC \ $y(321)^2\bar{y}\bar{1}\bar{2}\bar{3}\bar{1}\bar{2}$ for $\iota _{4;5}$. 

\begin{figure}[ht]
\includegraphics[scale=1.03]{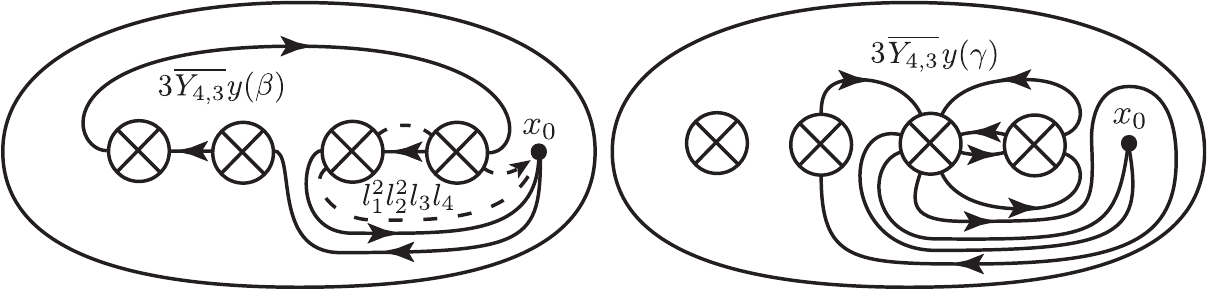}
\caption{Images of $\beta $ and $\gamma $ by $3\overline{Y_{4,3}}y$ and a representative of $l_1^2l_2^2l_3l_4$.}\label{calc_genus4_sgn5}
\end{figure}

\begin{figure}[ht]
\includegraphics[scale=1.2]{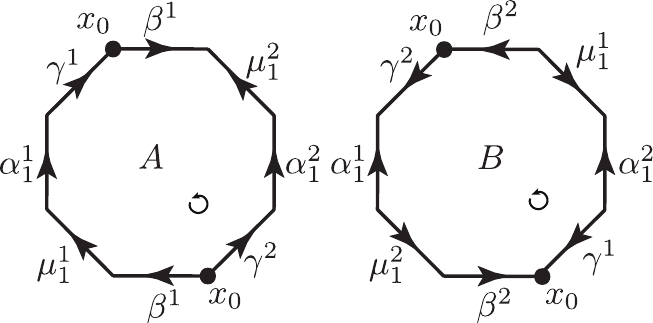}
\caption{Two octagons which are obtained from $N_4$ by cutting $\alpha _1\cup \mu_1\cup \beta \cup \gamma $.}\label{complement_disks_genus4_sgn5}
\end{figure}

By an argument in the proof of Proposition~\ref{prop_blowup_even-odd} (see also Figure~\ref{table_blowup_genus4-5}), the involution on $N_5$ which is obtained from $\iota _{4;5}$ by the blowup at the isolated fixed point $x_0$ of $\iota _{4;5}$ is \topconj to $\iota _{5;5}$. 
The image $Y_{5,4}(\alpha _3)$ is the simple closed curve on $N_5$ as in Figure~\ref{calc_genus5_sgn5}. 
By Lemma~\ref{pushing2}, the mapping class $\Phi (\Delta (l_1^2l_2^2l_3l_4))=\Psi (l_1^2l_2^2l_3l_4)\in \M (N_5)$ coincides with the product $Y_{5,4}3\overline{Y_{5,4}}\cdot \bar{3}$. 
By Lemma~\ref{rel_y_ij}, the relation $Y_{5,4}=3423121\bar{y}\bar{1}\bar{2}\bar{1}\bar{3}\bar{2}\bar{4}\bar{3}$ holds in $\M (N_5)$. 
Thus we have 
\begin{eqnarray*}
\iota _{5;5}&=&\Phi (\iota _{4;5})=\Phi (\Delta (l_1^2l_2^2l_3l_4)3\overline{Y_{4,3}}y)=Y_{5,4}3\overline{Y_{5,4}}\underline{\bar{3}\cdot 3}\overline{Y_{4,3}}y\\
&=&3423121\bar{y}\bar{1}\bar{2}\bar{1}\bar{3}\bar{2}\bar{4}\cdot \underline{3\cdot 423121}y\bar{1}\bar{2}\bar{1}\bar{3}\bar{2}\bar{4}\underline{\bar{3}\cdot 23}\ \underline{121}\bar{y}\bar{1}\bar{2}\bar{1}\bar{3}\bar{2}\cdot y\\
&\stackrel{\text{BR}}{=}&3423121\bar{y}\bar{1}\bar{2}\bar{1}\bar{3}\bar{2}\underline{\bar{4}4}342312y\bar{1}\bar{2}\bar{1}\bar{3}\bar{2}\bar{4}23\underline{\bar{2}2}12\bar{y}\bar{1}\bar{2}\bar{1}\bar{3}\bar{2}y\\
&=&3423121\bar{y}\underline{\bar{1}\bar{2}\bar{1}}\ \underline{\bar{3}\bar{2}3}42312y\bar{1}\bar{2}\bar{1}\bar{3}\bar{2}\bar{4}2312\bar{y}\underline{\bar{1}\bar{2}\bar{1}\bar{3}\bar{2}}y\\
&\stackrel{\text{BR}}{=}&3423121\bar{y}\bar{2}\underline{\bar{1}\bar{3}}\bar{2}\underline{4231}2y\bar{1}\bar{2}\bar{1}\bar{3}\underline{\bar{2}\bar{4}}2312\bar{y}\bar{2}\bar{1}\bar{3}\bar{2}\underline{\bar{3}y}\\
&\stackrel{\text{DIS}}{=}&3423121\bar{y}\bar{2}\bar{3}\underline{\bar{1}\bar{2}21}432y\bar{1}\bar{2}\bar{1}\bar{3}\bar{4}\underline{\bar{2}2}312\bar{y}\bar{2}\bar{1}\bar{3}\bar{2}y\bar{3}\\
&=&3423121\bar{y}\bar{2}\bar{3}432y\bar{1}\bar{2}\bar{1}\underline{\bar{3}\bar{4}31}2\bar{y}\bar{2}\bar{1}\bar{3}\bar{2}y\bar{3}\\
&\stackrel{\text{DIS}}{=}&3423121\bar{y}\bar{2}\bar{3}432y\bar{1}\bar{2}\underline{\bar{1}1}\bar{3}\bar{4}32\bar{y}\bar{2}\bar{1}\bar{3}\bar{2}y\bar{3}\\
&=&3423121\bar{y}\bar{2}\underline{\bar{3}43}2y\underline{\bar{1}\bar{2}\bar{3}\bar{4}32}\bar{y}\bar{2}\bar{1}\bar{3}\bar{2}y\bar{3}\\
&\stackrel{\text{BR, DIS}}{=}&3\underline{4231}21\bar{y}\underline{\bar{2}4}3\bar{4}\underline{2y43}\bar{1}\bar{2}\bar{3}\bar{4}\bar{y}\bar{2}\underline{\bar{1}\bar{3}}\bar{2}y\bar{3}\\
&\stackrel{\text{DIS}}{=}&3214321\bar{y}4\bar{2}\underline{323}y\bar{1}\bar{2}\bar{3}\bar{4}\bar{y}\bar{2}\bar{3}\bar{1}\bar{2}y\bar{3}\\
&\stackrel{\text{BR}}{=}&\underline{321}4321\bar{y}432y\bar{1}\bar{2}\bar{3}\bar{4}\bar{y}\bar{2}\bar{3}\bar{1}\bar{2}y\bar{3}\\
&\stackrel{\text{CONJ}}{\to }&4321\bar{y}432y\bar{1}\bar{2}\bar{3}\bar{4}\bar{y}\bar{2}\bar{3}\bar{1}\bar{2}y21
\end{eqnarray*}
Therefore, we obtain the \DC \ $4321\bar{y}432y\bar{1}\bar{2}\bar{3}\bar{4}\bar{y}\bar{2}\bar{3}\bar{1}\bar{2}y21$ for $\iota _{5;5}$. 

\begin{figure}[ht]
\includegraphics[scale=1.1]{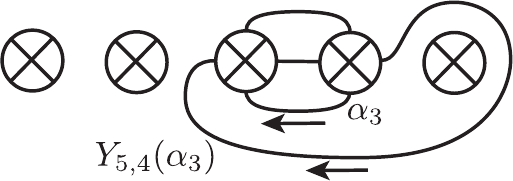}
\caption{The simple closed curve $Y_{5,4}(\alpha _3)$ on $N_5$.}\label{calc_genus5_sgn5}
\end{figure}

\subsubsection{A \DC \ for  $\iota _{4;6,1}$}\label{section_DC_4-6-1}

We construct a \DC \ for $\iota _{4;6,1}$. 
Let $N$ be the surface which is homeomorphic to $N_4$ as on the left-hand side in Figure~\ref{scc_genus4_sgn6-1} and $\iota _{4;6,1}$ the involution on $N$ which is induced by the reflection across the $yz$-plane. 
Let $\alpha _1^\prime $, $\mu _1^\prime $, and $\mu ^\prime $ be simple closed curves on $N$ which are preserved by $\iota _{4;6,1}$ setwise as in Figure~\ref{scc_genus4_sgn6-1}, and $\mu $ a simple closed curve on $N_4$ as on the right-hand side in Figure~\ref{scc_genus4_sgn6-1}. 
Then there exists a homeomorphism $F\colon N\to N_4$ such that $F(\alpha _1^\prime )=\alpha _1$, $F(\mu _1^\prime )=\mu _1$, and $F(\mu ^\prime )=\mu $. 
Since the composition $F\circ \iota _{4;6,1} \circ F^{-1}$ is \topconj to $\iota _{4;6,1}$, we regard $F\circ \iota _{4;6,1} \circ F^{-1}$ as $\iota _{4;6,1}$. 

\begin{figure}[ht]
\includegraphics[scale=1.02]{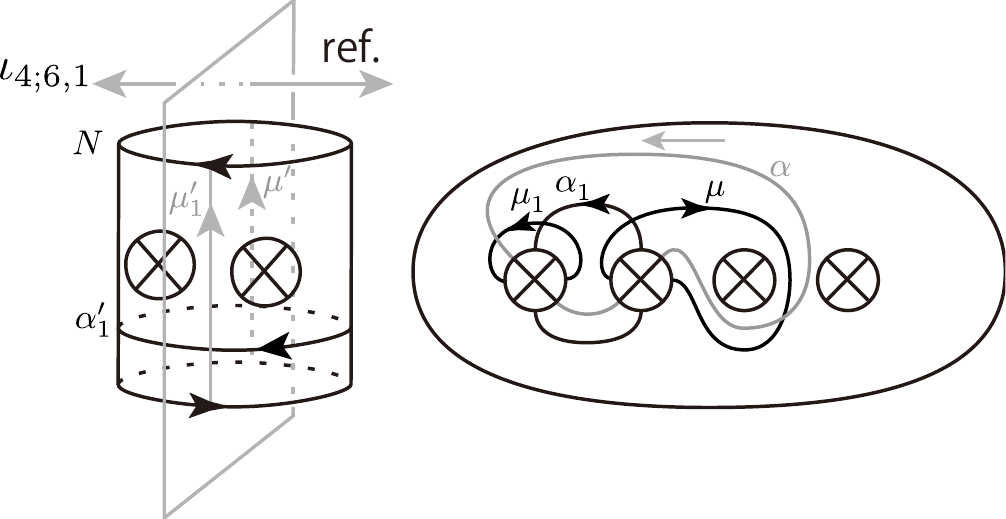}
\caption{Simple closed curves $\alpha _1^\prime $, $\mu _1^\prime $, and $\mu ^\prime $ on $N$ and simple closed curves $\mu $ and $\alpha $ on $N_4$. }\label{scc_genus4_sgn6-1}
\end{figure}

By the definition of $\iota _{4;6,1}$, we have
\begin{eqnarray*}
\iota _{4;6,1}(\alpha _1)=\alpha _1^{-1},\quad  
\iota _{4;6,1}(\mu _1)=\mu _1,\quad  
\iota _{4;6,1}(\mu )=\mu  \quad \text{with orientations}, 
\end{eqnarray*}
and $\iota _{4;6,1}(\gamma _{3})=\gamma _{4}$ without orientations. 
Let $\alpha$ be a simple closed curve on $N_4$ whose regular neighborhood has an orientation as in Figure~\ref{scc_genus4_sgn6-1}. 
We can check that $t_{\alpha }\bar{1}y21123(\alpha _1)=\alpha _1^{-1}$, $t_{\alpha }\bar{1}y21123(\mu _1)=\mu _1$, and $t_{\alpha }\bar{1}y21123(\mu )=\mu $ with orientations, and $t_{\alpha }\bar{1}y21123(\gamma _{3})=\gamma _{4}$ without orientations. 
Thus, by Proposition~\ref{alexander_method}, we have $t_{\alpha }\bar{1}y21123=\iota _{4;6,1}$ in $\M (N_4)$. 
Since $21y(\alpha _2)=\alpha $, the relation $t_{\alpha }=21y2\bar{y}\bar{1}\bar{2}$ holds in $\M (N_4)$. 
Thus we have
\begin{eqnarray*}
\iota _{4;6,1}&=&t_{\alpha }\bar{1}y21123=21y2\underline{\bar{y}\bar{1}}\bar{2}\cdot \bar{1}y21123\\
&=&21y21\underline{\bar{y}\bar{2}\bar{1}y21}123\\
&=&21y\underline{212}\bar{y}\bar{2}\bar{1}y\underline{212}3\\
&\stackrel{\text{BR}}{=}&21\underline{y1}21\bar{y}\bar{2}\bar{1}y1213\\
&\stackrel{\text{(0)}}{=}&2y21\bar{y}\bar{2}\bar{1}y12\underline{13}\\
&\stackrel{\text{DIS}}{=}&2y21\bar{y}\bar{2}\bar{1}y123\underline{1}\\
&\stackrel{\text{CONJ}}{\to }&12y21\bar{y}\bar{2}\bar{1}y123
\end{eqnarray*}
in $\M (N_4)$, where we use the relation $(\bar{y}\bar{2}\bar{1}y2)1(\bar{y}\bar{2}\bar{1}y2)^{-1}=2$ in the fourth equality. 
Thus we obtain the \DC \ $12y21\bar{y}\bar{2}\bar{1}y123$ for $\iota _{4;6,1}$.

\subsubsection{A \DC \ for  $\iota _{4;6,2}$}\label{section_DC_4-6-2}

We construct a \DC \ for $\iota _{4;6,2}$. 
Let $N$ be the surface which is homeomorphic to $N_4$ as on the left-hand side in Figure~\ref{scc_genus4_sgn6-2} and $\iota _{4;6,2}$ the involution on $N$ which is induced by the reflection across the $yz$-plane. 
Let $\alpha _i^\prime $ $(i=2, 3)$ and $\gamma ^\prime $ be simple closed curves on $N$ which are preserved by $\iota _{4;6,2}$ setwise as in Figure~\ref{scc_genus4_sgn6-2}, and $\gamma $ a simple closed curve on $N_4$ as on the right-hand side in Figure~\ref{scc_genus4_sgn6-2}. 
Then there exists a homeomorphism $F\colon N\to N_4$ such that $F(\alpha _i^\prime )=\alpha _i$ $(i=2,3)$ and $F(\gamma ^\prime )=\gamma $. 
Since the composition $F\circ \iota _{4;6,2} \circ F^{-1}$ is \topconj to $\iota _{4;6,2}$, we regard $F\circ \iota _{4;6,2} \circ F^{-1}$ as $\iota _{4;6,2}$. 

\begin{figure}[ht]
\includegraphics[scale=1.02]{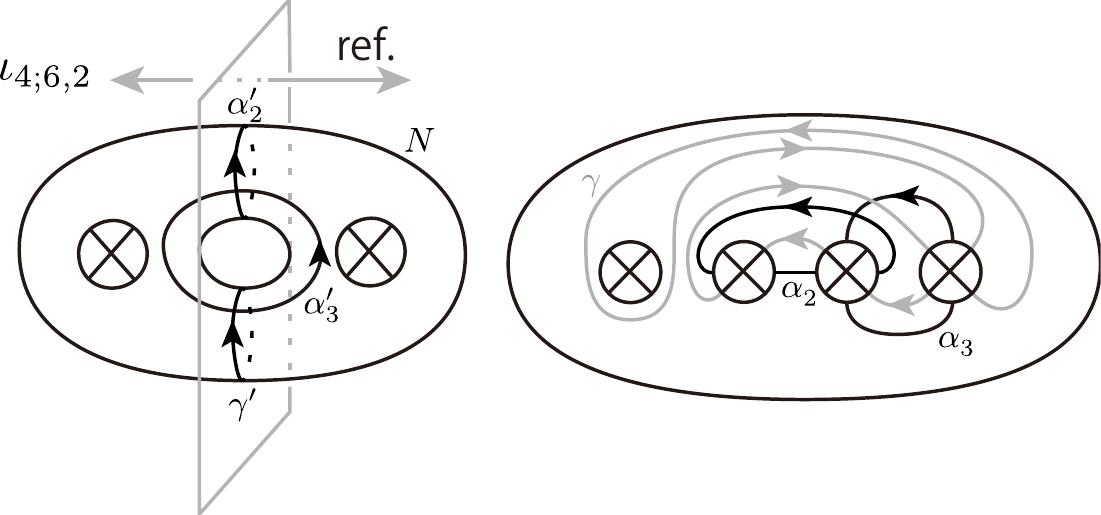}
\caption{Simple closed curves $\alpha _i^\prime $ $(i=2, 3)$ and $\gamma ^\prime $ on $N$ and a simple closed curve $\gamma $ on $N_4$. }\label{scc_genus4_sgn6-2}
\end{figure}

By the definition of $\iota _{4;6,2}$, we have
\begin{eqnarray*}
\iota _{4;6,2}(\alpha _2)=\alpha _2,\quad  
\iota _{4;6,2}(\alpha _3)=\alpha _3^{-1},\quad  
\iota _{4;6,2}(\gamma )=\gamma  \quad \text{with orientations}, 
\end{eqnarray*}
and $\iota _{4;6,2}(\mu _1)=\gamma _{2,3,4}$ without orientations. 
We can check that $bY_{\mu _1, \beta }\bar{y}Y_{4,3}3^223^22(\alpha _2)=\alpha _2$, $bY_{\mu _1, \beta }\bar{y}Y_{4,3}3^223^22(\alpha _3)=\alpha _3^{-1}$, and $bY_{\mu _1, \beta }\bar{y}Y_{4,3}3^223^22(\gamma )=\gamma ^{-1}$ with orientations, and $bY_{\mu _1, \beta }\bar{y}Y_{4,3}3^223^22(\mu _1)=\gamma _{2,3,4}$ without orientations. 
Thus, by Proposition~\ref{alexander_method}, we have $bY_{\mu _1, \beta }\bar{y}Y_{4,3}3^223^22=\iota _{4;6,2}$ in $\M (N_4)$. 
By Lemma~\ref{rel_y_ij}, the relation $Y_{4,3}=23121y\bar{1}\bar{2}\bar{1}\bar{3}\bar{2}$ holds in $\M (N_4)$. 
By regarding $Y_{\mu_ 1,\beta }$ as a pushing map of first crosscap and Lemma~\ref{pushing2}, the relation $Y_{\mu_ 1,\beta }=Y_{\mu _1, \gamma _{1,3,4}}\cdot y=(\bar{2}\bar{3}y2\bar{y}32\cdot \bar{3})\cdot y$ holds in $\M (N_4)$. 
Thus we have
\begin{eqnarray*}
\iota _{4;6,2}&=&bY_{\mu _1, \beta }\bar{y}Y_{4,3}3^223^22=b\cdot \bar{2}\bar{3}y2\bar{y}32\bar{3}y\cdot \bar{y}\cdot \underline{23121}y\underline{\bar{1}\bar{2}\bar{1}\bar{3}\bar{2}}\cdot 3\underline{323}32\\
&\stackrel{\text{BR}}{=}&b\bar{2}\bar{3}y2\bar{y}32\underline{\bar{3}3}2312y\bar{2}\bar{1}\underline{\bar{3}\bar{2}\bar{3}323}232\\
&=&b\bar{2}\bar{3}y2\bar{y}322312y\underline{\bar{2}\bar{1}2}32\\
&\stackrel{\text{BR}}{=}&b\bar{2}\bar{3}y2\bar{y}322312\underline{y1}\bar{2}\bar{1}32\\
&=&b\bar{2}\bar{3}y2\bar{y}32\underline{2312\bar{1}}y\bar{2}\bar{1}32\\
&\stackrel{\text{BR}}{=}&b\bar{2}\bar{3}y2\underline{\bar{y}3}2\bar{3}2312y\bar{2}\bar{1}32\\
&\stackrel{\text{DIS}}{=}&b\bar{2}\bar{3}y23\bar{y}2\bar{3}2312y\bar{2}\underline{\bar{1}32}\\
&\stackrel{\text{CONJ}}{\to }&\underline{\bar{1}32b\bar{2}\bar{3}y}23\bar{y}2\bar{3}2312y\bar{2}\\
&\overset{\text{DIS}}{\underset{\text{(0)}}{=}}&by123\bar{y}2\bar{3}2312y\bar{2}
\end{eqnarray*}
in $\M (N_4)$. 
Thus we obtain the \DC \ $by123\bar{y}2\bar{3}2312y\bar{2}$ for $\iota _{4;6,2}$.

\subsubsection{\DC s for $\iota _{4;7}$ and $\iota _{5;6}$}\label{section_DC_4-7_5-6}

We construct a \DC \ for $\iota _{4;7}$, and give a \DC \ for $\iota _{5;6}$ from the \DC \ for $\iota _{4;7}$ by a blowup at an isolated fixed point of $\iota _{4;7}$. 
Let $\iota _{4;7}$ be the involution on $N_4$ which is induced by the $\pi $-rotation of $N_4$ on the gray axis as in Figure~\ref{figure_involution_genus4} (in this case, we can already regard $\iota _{4;7}$ as an involution on $N_4$). 
The involution $\iota _{4;7}$ has two isolated fixed points and we denote by $x_0$ the top isolated fixed point of $\iota _{4;7}$. 

By the definition of $\iota _{4;7}$, we have
\begin{eqnarray*}
\iota _{4;7}(\gamma _{i})=\gamma _{i+2}\quad (i=1,2),\quad  
\iota _{4;7}(\gamma _{3})=\gamma _{1}=\mu _1,\quad 
\iota _{4;7}(\beta )=\beta ,  
\end{eqnarray*}
and $\iota _{4;7}(x_0)=x_0$. 
Hence we regard $\iota _{4;7}$ as an element of $\M (N_4,x_0)$.
We can check that $(123)^2(\gamma _{i})=\gamma _{i+2}$ $(i=1,2)$ and $(123)^2(\beta )=\beta $ in $\M (N_4,x_0)$, and the image of $\gamma _{3}$ by $(123)^2$ is isotopic to a simple closed curve on $N_4$ as on the left-hand side in Figure~\ref{calc_genus4_sgn7} relative to $x_0$. 
The product $l_1l_2$ in $\pi _1(N_4,x_{0})$ is represented by the simple loop which is described as the dotted loop based at $x_{0}$ on the left-hand side in Figure~\ref{calc_genus4_sgn7}. 
The mapping class $\Delta (l_1l_2)\in \M (N_4,x_{0})$ fixes $\gamma _{i}$ $(i=3,4)$ and $\beta $ pointwise, and the images of $(123)^2(\gamma _{3})$ by $\Delta (l_1l_2)$ are isotopic to $\mu _1$ relative to $x_{0}$. 
By Proposition~\ref{alexander_method}, we have $\Delta (l_1l_2)(123)^2=\iota _{4;7}$ in $\M (N_4,x_{0})$. 
Using the forgetful homomorphism, we have $\iota _{4;7}=\F (\Delta (l_1l_2)(123)^2)=(123)^2$ in $\M (N_4)$. 
Thus we obtain the \DC \ $(123)^2$ for $\iota _{4;7}$. 

\begin{figure}[ht]
\includegraphics[scale=1.0]{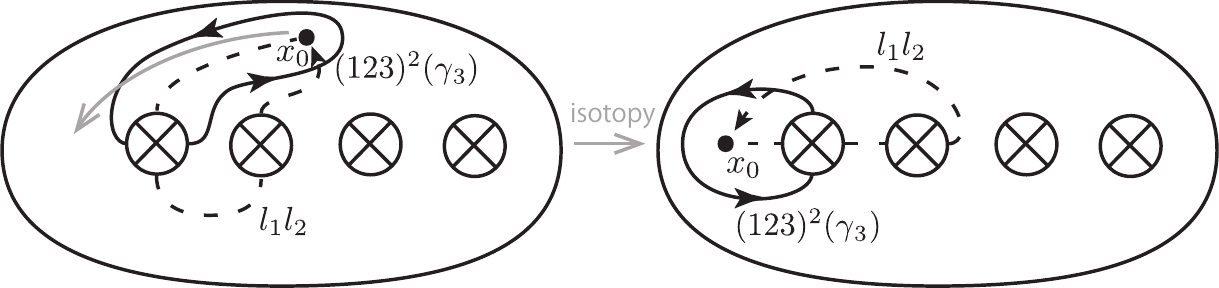}
\caption{The image of $\gamma _{3}$ by $(123)^2$ and a representative of $l_1l_2$.}\label{calc_genus4_sgn7}
\end{figure}

By an argument in the proof of Proposition~\ref{prop_blowup_even-odd} (see also Figure~\ref{table_blowup_genus4-5}), the involution on $N_5$ which is obtained from $\iota _{4;7}$ by the blowup at the isolated fixed point $x_0$ of $\iota _{4;7}$ is \topconj to $\iota _{5;6}$. 
In this case, by an isotopy on $N_4$ as in Figure~\ref{calc_genus4_sgn7}, we especially regard the crosscap which is obtained from a neighborhood of $x_0$ by the blowup at $x_0$ as the first crosscap on $N_5$. 
Let $\Phi \colon \M (N_4, x_0)\to \M (N_5)$ be the blowup homomorphism induced by the blowup at $x_0$. 
Remark that ,under this homomorphism $\Phi $, we have $\Phi (t_{\alpha _i})=t_{\alpha _{i+1}}$ for $i=1,2,3$. 
By Lemma~\ref{pushing2}, the mapping class $\Phi (\Delta (l_1l_2))=\Psi (l_1l_2)\in \M (N_5)$ coincides with the product $y2\bar{y}\cdot \bar{2}$. 
Thus we have 
\begin{eqnarray*}
\iota _{5;6}&=&\Phi (\iota _{4;7})=\Phi (\Delta (l_1l_2)(123)^2)=y2\bar{y}\bar{2}\cdot (234)^2\\
&=&y2\underline{\bar{y}34}234\\
&\stackrel{\text{DIS}}{=}&y234\bar{y}234
\end{eqnarray*}
Therefore, we obtain the \DC \ $y234\bar{y}234$ for $\iota _{5;6}$.

\subsubsection{A \DC \ for  $\iota _{4;8,1}$}\label{section_DC_4-8-1}

We construct a \DC \ for $\iota _{4;8,1}$. 
Let $\iota _{4;8,1}$ be the involution on $N_4$ which is induced by the reflection across the $yz$-plane as in Figure~\ref{figure_involution_genus4} (in this case, $\iota _{4;8,1}$ is already an involution on $N_4$). 
By the definition of $\iota _{4;8,1}$, we have
\begin{eqnarray*}
\iota _{4;8,1}(\gamma _{i})=\gamma _{5-i}^{-1}\quad (i=1,2,3),\quad  
\iota _{4;8,1}(\beta )=\beta ^{-1}.  
\end{eqnarray*}
We regard $\iota _{4;8,1}$ as an element of $\M (N_4)$. 
We can check that $(\bar{y}\bar{2}y)12\bar{y}321(\gamma _{i})=\gamma _{5-i}^{-1}$ $(i=1,2,3)$ and $(\bar{y}\bar{2}y)12\bar{y}321(\beta )=\beta ^{-1}$ in $\M (N_4)$. 
Thus, by Proposition~\ref{alexander_method}, we have 
\begin{eqnarray*}
\iota _{4;8,1}=\bar{y}\bar{2}y12\bar{y}321
\end{eqnarray*}
in $\M (N_4)$. 
Thus we obtain the \DC \ $\bar{y}\bar{2}y12\bar{y}321$ for $\iota _{4;8,1}$.

\subsubsection{A \DC \ for  $\iota _{4;8,2}$}\label{section_DC_4-8-2}

We construct a \DC \ for $\iota _{4;8,2}$. 
Let $N$ be the surface which is homeomorphic to $N_4$ as on the left-hand side in Figure~\ref{scc_genus4_sgn8-2} and $\iota _{4;8,2}$ the involution on $N$ which is induced by the reflection across the $yz$-plane. 
Let $\gamma ^\prime $, $\gamma _{1,3}^\prime $, $\mu $, and $\mu ^\prime $ be simple closed curves on $N$ which are preserved by $\iota _{4;6,1}$ setwise as in Figure~\ref{scc_genus4_sgn8-2}, and $\gamma $ a simple closed curve on $N_4$ as on the right-hand side in Figure~\ref{scc_genus4_sgn8-2}. 
Then there exists a homeomorphism $F\colon N\to N_4$ such that $F(\gamma ^\prime )=\gamma $, $F(\gamma _{1,3}^\prime )=\gamma _{1,3}$, $F(\mu )=\gamma _{1,2,3}$, and $F(\mu ^\prime )=\gamma _{4}$. 
Since the composition $F\circ \iota _{4;8,2} \circ F^{-1}$ is \topconj to $\iota _{4;8,2}$, we regard $F\circ \iota _{4;8,2} \circ F^{-1}$ as $\iota _{4;8,2}$. 

\begin{figure}[ht]
\includegraphics[scale=1.02]{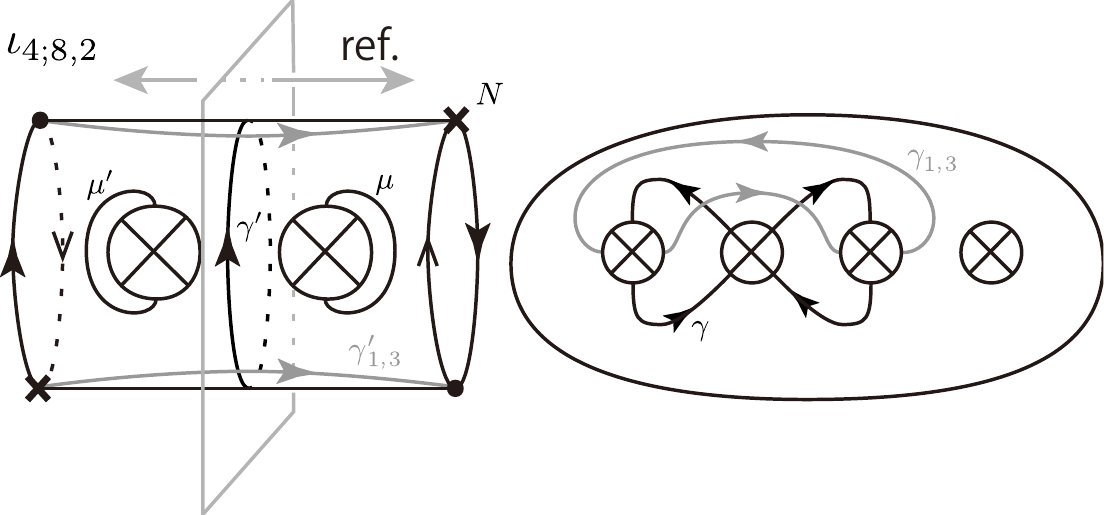}
\caption{Simple closed curves $\gamma ^\prime $, $\gamma _{1,3}^\prime $, $\mu $, and $\mu ^\prime $ on $N$ and a simple closed curve $\gamma $ on $N_4$. }\label{scc_genus4_sgn8-2}
\end{figure}

By the definition of $\iota _{4;8,2}$, we have
\begin{eqnarray*}
\iota _{4;8,2}(\gamma )=\gamma \quad \text{and} \quad  
\iota _{4;8,2}(\gamma _{1,3})=\gamma _{1,3}^{-1}\quad \text{with orientations, }
\end{eqnarray*}
and $\iota _{4;8,2}(\gamma _{1,2,3})=\gamma _{4}$ without orientations. 
We can check that $\bar{b}(2y2\bar{y}\bar{2})(\bar{2}y2)yY_{3,2}(\gamma )=\gamma $ and $\bar{b}(2y2\bar{y}\bar{2})(\bar{2}y2)yY_{3,2}(\gamma _{1,3})=\gamma _{1,3}^{-1}$ with orientations, and $\bar{b}(2y2\bar{y}\bar{2})(\bar{2}y2)yY_{3,2}(\gamma _{1,2,3})=\gamma _{4}$ without orientations. 
Thus, by Proposition~\ref{alexander_method}, we have $\bar{b}2y2\bar{y}\bar{2}\bar{2}y2yY_{3,2}=\iota _{4;8,2}$ in $\M (N_4)$. 
Remark that $y2\bar{y}$ commutes with $\bar{2}$. 
By Lemma~\ref{rel_y_ij}, the relation $Y_{3,2}=121y\bar{1}\bar{2}\bar{1}$ holds in $\M (N_4)$. 
Thus we have
\begin{eqnarray*}
\iota _{4;8,2}&=&\bar{b}2y2\bar{y}\bar{2}\bar{2}y2yY_{3,2}=\bar{b}2\underline{y2\bar{y}\bar{2}\bar{2}}y2y\cdot 121y\bar{1}\bar{2}\bar{1}\\
&\stackrel{\text{DIS}}{=}&\bar{b}\underline{2\bar{2}}\bar{2}y2\underline{\bar{y}y}2y121y\bar{1}\bar{2}\bar{1}\\
&=&\underline{\bar{b}\bar{2}}y22y121\underline{y\bar{1}}\bar{2}\bar{1}\\
&\stackrel{\text{DIS}}{=}&\bar{2}\bar{b}y22y\underline{121}1y\bar{2}\bar{1}\\
&\stackrel{\text{BR}}{=}&\bar{2}\bar{b}y22y2\underline{121}y\bar{2}\bar{1}\\
&\stackrel{\text{BR}}{=}&\underline{\bar{2}}\bar{b}(y2^2)^212y\bar{2}\bar{1}\\
&\stackrel{\text{CONJ}}{\to }&\bar{b}(y2^2)^212y\underline{\bar{2}\bar{1}\bar{2}}\\
&\stackrel{\text{BR}}{=}&\bar{b}(y2^2)^212\underline{y\bar{1}}\bar{2}\bar{1}\\
&\stackrel{\text{BR}}{=}&\bar{b}(y2^2)^2121y\bar{2}\bar{1}
\end{eqnarray*}
in $\M (N_4)$. 
Thus we obtain the \DC \ $\bar{b}(y2^2)^2121y\bar{2}\bar{1}$ for $\iota _{4;8,2}$.

\subsubsection{A \DC \ for  $\iota _{4;9,1}$}\label{section_DC_4-9-1}

We construct a \DC \ for $\iota _{4;9,1}$. 
Let $\iota _{4;9,1}$ be the involution on $N_4$ which is induced by the antipodal action on $N_4$ as in Figure~\ref{scc_genus4_sgn9-1} (in this case, $\iota _{4;9,1}$ is already an involution on $N_4$). 
Remark that $\iota _{4;9,1}$ has no fixed point. 
We regard the simple closed curves $\beta $ and $\gamma _{i}$ $(i=1,2,3,4)$ as the simple closed curves on $N_4$ in Figure~\ref{scc_genus4_sgn9-1}. 

\begin{figure}[ht]
\includegraphics[scale=1.02]{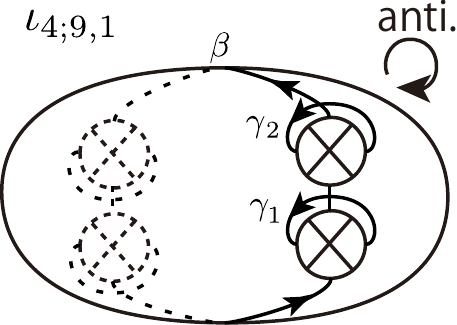}
\caption{Simple closed curves $\beta $ and $\gamma _{i}$ $(i=1,2,3,4)$  on $N_4$. }\label{scc_genus4_sgn9-1}
\end{figure}

By the definition of $\iota _{4;9,1}$, we have
\begin{eqnarray*}
\iota _{4;9,1}(\beta )=\beta , \quad  
\iota _{4;9,1}(\gamma _{i})=\gamma _{i+2}^{-1}\quad (i=1,2),\quad  
\iota _{4;9,1}(\gamma _{3})=\gamma _{1}^{-1}.
\end{eqnarray*}
We can check that $(1\bar{y})23(y2\bar{y}\cdot \bar{2})\bar{1}\bar{2}\bar{3}(\beta )=\beta $, $(1\bar{y})23(y2\bar{y}\cdot \bar{2})\bar{1}\bar{2}\bar{3}(\gamma _{i})=\gamma _{i+2}^{-1}$ $(i=1,2)$, and $(1\bar{y})23(y2\bar{y}\cdot \bar{2})\bar{1}\bar{2}\bar{3}(\gamma _{3})=\gamma _{1}^{-1}$. 
Thus, by Proposition~\ref{alexander_method}, we have 
\begin{eqnarray*}
\iota _{4;9,1}&=&1\bar{y}23y2\bar{y}\underline{\bar{2}\bar{1}\bar{2}}\bar{3}\\
&\stackrel{\text{BR}}{=}&1\bar{y}23y2\underline{\bar{y}\bar{1}}\bar{2}\underline{\bar{1}\bar{3}}\\
&\stackrel{\text{DIS}}{=}&1\bar{y}23y21\underline{\bar{y}\bar{2}\bar{3}\bar{1}}\\
&\stackrel{\text{CONJ}}{\to }&\bar{y}\bar{2}\bar{3}\bar{y}23y21
\end{eqnarray*}
in $\M (N_4)$. 
Thus we obtain the \DC \ $\bar{y}\bar{2}\bar{3}\bar{y}23y21$ for $\iota _{4;9,1}$.

\subsubsection{A \DC \ for  $\iota _{4;9,3}$}\label{section_DC_4-9-2}

We construct a \DC \ for $\iota _{4;9,3}$. 
Let $N$ be the surface which is homeomorphic to $N_4$ as on the left-hand side in Figure~\ref{scc_genus4_sgn9-2} and $\iota _{4;9,3}$ the involution on $N$ which is induced by the antipodal action on $N$. 
Remark that $\iota _{4;9,3}$ has no fixed point. 
Let $\alpha _i^\prime $ $(i=2, 3)$ and $\gamma ^\prime $ be simple closed curves on $N$ which are preserved by $\iota _{4;9,3}$ setwise as in Figure~\ref{scc_genus4_sgn9-2}, and $\gamma $ a simple closed curve on $N_4$ as on the right-hand side in Figure~\ref{scc_genus4_sgn9-2}. 
Then there exists a homeomorphism $F\colon N\to N_4$ such that $F(\alpha _i^\prime )=\alpha _i$ $(i=2,3)$ and $F(\gamma ^\prime )=\gamma $. 
Since the composition $F\circ \iota _{4;9,3} \circ F^{-1}$ is \topconj to $\iota _{4;9,3}$, we regard $F\circ \iota _{4;9,3} \circ F^{-1}$ as $\iota _{4;9,3}$. 

\begin{figure}[ht]
\includegraphics[scale=1.02]{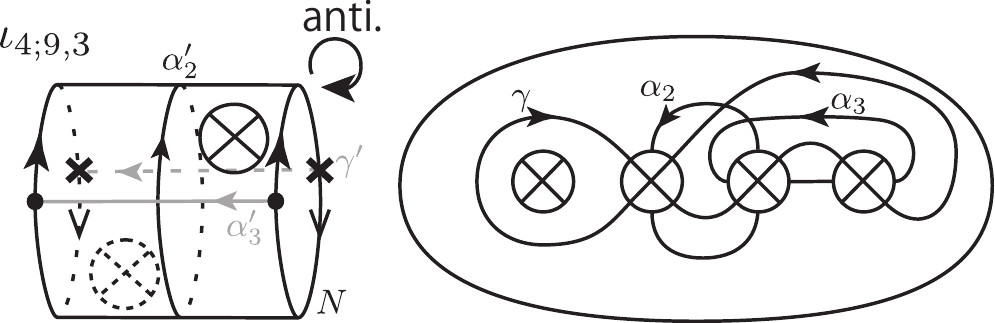}
\caption{Simple closed curves $\alpha _i^\prime $ $(i=2, 3)$ and $\gamma ^\prime $ on $N$ and a simple closed curve $\gamma $ on $N_4$. }\label{scc_genus4_sgn9-2}
\end{figure}

By the definition of $\iota _{4;9,3}$, we have
\begin{eqnarray*}
\iota _{4;9,3}(\alpha _2)=\alpha _2, \quad  
\iota _{4;9,3}(\alpha _3)=\gamma ^{-1},\quad \text{and}\quad  
\iota _{4;9,3}(\gamma )=\alpha _3^{-1}\quad \text{with orientations},
\end{eqnarray*}
and $\iota _{4;9,3}(\mu )=\gamma _{2,3,4}$ without orientations. 
We can check that $\bar{y}(\bar{2}y2)Y_{2,3}(\bar{2}\cdot y2\bar{y})\bar{b}(\alpha _2)=\alpha _2$, $\bar{y}(\bar{2}y2)Y_{2,3}(\bar{2}\cdot y2\bar{y})\bar{b}(\alpha _3)=\gamma ^{-1}$, and $\bar{y}(\bar{2}y2)Y_{2,3}(\bar{2}\cdot y2\bar{y})\bar{b}(\gamma )=\alpha _3^{-1}$ with orientations, and $\bar{y}(\bar{2}y2)Y_{2,3}(\bar{2}\cdot y2\bar{y})\bar{b}(\mu )=\gamma _{2,3,4}$ without orientations. 
Thus, by Proposition~\ref{alexander_method}, we have $\bar{y}\bar{2}y2Y_{2,3}\bar{2}y2\bar{y}\bar{b}=\iota _{4;9,3}$ in $\M (N_4)$. 
By Lemma~\ref{rel_y_ij}, the relation $Y_{2,3}=12\bar{y}\bar{2}\bar{1}$ holds in $\M (N_4)$. 
Thus we have
\begin{eqnarray*}
\iota _{4;9,3}&=&\bar{y}\bar{2}y2Y_{2,3}\bar{2}y2\bar{y}\bar{b}=\bar{y}\bar{2}y\underline{2\cdot 12}\bar{y}\underline{\bar{2}\bar{1}\cdot \bar{2}}y2\bar{y}\bar{b}\\
&\stackrel{\text{BR}}{=}&\bar{y}\bar{2}\underline{y1}21\bar{y}\bar{1}\bar{2}\underline{\bar{1}y}2\bar{y}\bar{b}\\
&=&\bar{y}\bar{2}\bar{1}y21\bar{y}\bar{1}\bar{2}y12\bar{y}\bar{b}\\
\end{eqnarray*}
in $\M (N_4)$. 
Thus we obtain the \DC \ $\bar{y}\bar{2}\bar{1}y21\bar{y}\bar{1}\bar{2}y12\bar{y}\bar{b}$ for $\iota _{4;9,3}$.

\subsubsection{\DC s for $\iota _{4;10}$ and $\iota _{5;7}$}\label{section_DC_4-10_5-7}

We construct a \DC \ for $\iota _{4;10}$, and give a \DC \ for $\iota _{5;7}$ from the \DC \ for $\iota _{4;10}$ by a blowup at a fixed point which lies in a \refline \ of $\iota _{4;10}$. 
Let $N$ be the surface which is homeomorphic to $N_4$ as on the left-hand side in Figure~\ref{scc_genus4_sgn10} and $\iota _{4;10}$ the involution on $N$ which is induced by the $\pi $-rotation of $N$ on the gray axis. 
The involution $\iota _{4;10}$ has one two-sided \refline . 
Let $x_{0}^\prime $ be a fixed point of $\iota _{4;10}$ which lies in the two-sided \refline , $\alpha _i^\prime $ $(i=1,2,3)$ and $\gamma ^\prime $ simple closed curves on $N$ which are preserved by $\iota _{4;10}$ setwise as on the left-hand side in Figure~\ref{scc_genus4_sgn10}, and $\gamma $ a simple closed curve on $N_4$ as on the right-hand side in Figure~\ref{scc_genus4_sgn10}. 
Then there exists a homeomorphism $F\colon N\to N_4$ such that $F(\alpha _i^\prime )=\alpha _i$ $(i=1,2,3)$, $F(\gamma ^\prime )=\gamma $, and $F(x_{0}^\prime )=x_{0}$, respectively. 
Since the composition $F\circ \iota _{4;10} \circ F^{-1}$ is \topconj to $\iota _{4;10}$, we regard $F\circ \iota _{4;10} \circ F^{-1}$ as $\iota _{4;10}$. 

\begin{figure}[ht]
\includegraphics[scale=1.1]{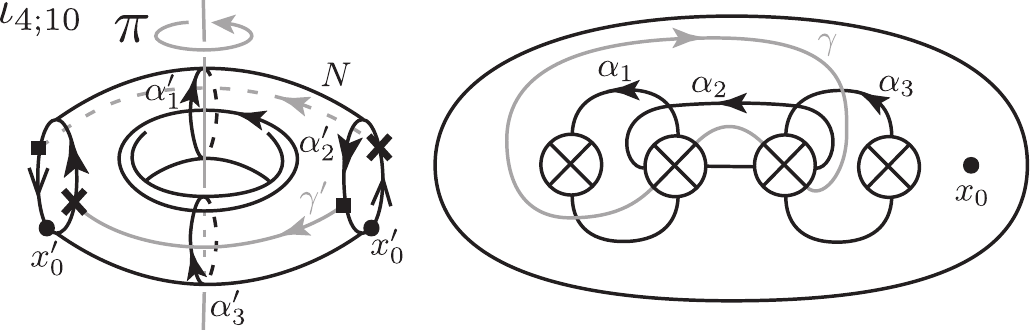}
\caption{Simple closed curves $\alpha _i^\prime $ $(i=1,2,3)$ and $\gamma ^\prime $ on $N$, a point $x_{0}^\prime $ in $N$, and a simple closed curve $\gamma $ on $N_4$. 
The surface $N$ is obtained from the torus with two boundary components by identifying the boundary along the oriented edges as on the left-hand side.}\label{scc_genus4_sgn10}
\end{figure}

By the definition of $\iota _{4;10}$, we have
\begin{eqnarray*}
\iota _{4;10}(\alpha _1)=\alpha _3^{-1},\quad  
\iota _{4;10}(\alpha _2)=\alpha _2, \quad  
\iota _{4;10}(\alpha _3)=\alpha _1^{-1},\quad  
\iota _{4;10}(\gamma )=\gamma ^{-1} , 
\end{eqnarray*}
and $\iota _{4;10}(x_0)=x_0$. 
Hence we regard $\iota _{4;10}$ as an element of $\M (N_4,x_0)$.
We can check that $\bar{b}\Delta (l_1^{-1})\cdot \bar{3}\bar{2}\bar{y}23\cdot y2\bar{y}\cdot 32\bar{1}\bar{2}\bar{3}(\alpha _1)=\alpha _3^{-1}$, $\bar{b}\Delta (l_1^{-1})\cdot \bar{3}\bar{2}\bar{y}23\cdot y2\bar{y}\cdot 32\bar{1}\bar{2}\bar{3}(\alpha _2)=\alpha _2$, $\bar{b}\Delta (l_1^{-1})\cdot \bar{3}\bar{2}\bar{y}23\cdot y2\bar{y}\cdot 32\bar{1}\bar{2}\bar{3}(\alpha _3)=\alpha _1^{-1}$, and $\bar{b}\Delta (l_1^{-1})\cdot \bar{3}\bar{2}\bar{y}23\cdot y2\bar{y}\cdot 32\bar{1}\bar{2}\bar{3}(\gamma )=\gamma ^{-1}$ relative to $x_0$, respectively. 
The complement of $\alpha _1\cup \alpha _2 \cup \alpha _3 \cup \gamma $ in $N_4$ is a disjoint union of two disks and the point $x_0$ lies in the interior of one of the disks. 
Since the mapping classes $\iota _{4;10}$ and $\bar{b}\Delta (l_1^{-1})\cdot \bar{3}\bar{2}\bar{y}23\cdot y2\bar{y}\cdot 32\bar{1}\bar{2}\bar{3}$ fix $x_0$, they preserve each connected component of $N_4-\alpha _1\cup \alpha _2 \cup \alpha _3 \cup \gamma $. 
Thus, by Proposition~\ref{alexander_method}, we have $\bar{b}\Delta (l_1^{-1})\bar{3}\bar{2}\bar{y}23y2\bar{y}32\bar{1}\bar{2}\bar{3}=\iota _{4;10}$ in $\M (N_4,x_{0})$. 
Remark that $\bar{2}\bar{y}2\cdot 3\cdot y2\bar{y}\cdot 3(\alpha _2)=\alpha _3$ and $\bar{y}23\cdot y2\bar{y}(\alpha _3)=\alpha _2$ with orientations of regular neighborhoods of curves. 
Thus we have 
\begin{eqnarray*}
\iota _{4;10}&=&\bar{b}\Delta (l_1^{-1})\bar{3}\underline{\bar{2}\bar{y}23y2\bar{y}32}\bar{1}\bar{2}\bar{3}\\
&=&\bar{b}\Delta (l_1^{-1})\underline{\bar{3}3}\bar{2}\bar{y}23y2\bar{y}3\bar{1}\bar{2}\bar{3}\\ 
&=&\bar{b}\Delta (l_1^{-1})\bar{2}\underline{\bar{y}23y2\bar{y}3}\bar{1}\bar{2}\bar{3}\\
&=&\bar{b}\Delta (l_1^{-1})\underline{\bar{2}2}\bar{y}23y2\bar{y}\bar{1}\bar{2}\bar{3}\\
&=&\bar{b}\Delta (l_1^{-1})\bar{y}23y2\bar{y}\bar{1}\bar{2}\bar{3}
\end{eqnarray*}
in $\M (N_4,x_0)$. 
Using the forgetful homomorphism, we obtain the \DC \ $\bar{b}\bar{y}23y2\bar{y}\bar{1}\bar{2}\bar{3}$ for $\iota _{4;10}$. 

By an argument in the proof of Proposition~\ref{prop_blowup_even-odd} (see also Figure~\ref{table_blowup_genus4-5}), the involution on $N_5$ which is obtained from $\iota _{4;10}$ by the blowup at the fixed point $x_0$ which lies in a \refline \ of $\iota _{4;10}$ is \topconj to $\iota _{5;7}$. 
By Lemma~\ref{rel_y_5i}, the mapping class $\Phi (\Delta (l_1^{-1}))=\Psi (l_1^{-1})\in \M (N_5)$ coincides with the crosscap slide $\overline{Y_{5,1}}$. 

By Lemma~\ref{rel_y_ij}, the relation $\overline{Y_{5,1}}=\bar{4}\bar{3}\bar{2}\bar{1}y1234$ holds in $\M (N_5)$. 
Thus we have 
\begin{eqnarray*}
\iota _{5;7}&=&\Phi (\iota _{4;10})=\Phi (\bar{b}\Delta (l_1^{-1})\bar{y}23y2\bar{y}\bar{1}\bar{2}\bar{3})\\
&=&\bar{b}\overline{Y_{5,1}}\bar{y}23y2\bar{y}\bar{1}\bar{2}\bar{3}\\
&=&\bar{b}\cdot \bar{4}\bar{3}\bar{2}\bar{1}y1234\cdot \bar{y}23y2\bar{y}\bar{1}\bar{2}\bar{3}\\
\end{eqnarray*}
Therefore, we obtain the \DC \ $\bar{b}\bar{4}\bar{3}\bar{2}\bar{1}y1234\bar{y}23y2\bar{y}\bar{1}\bar{2}\bar{3}$ for $\iota _{5;7}$.

%\appendix

\par
{\bf Acknowledgement:} The authors would like to express their gratitude to Susumu Hirose, for his encouragement and helpful advices. 
The first author was supported by JSPS KAKENHI Grant Number JP19K23409 and JP21K13794.
The second author was supported by MEXT Grants-in-Aid for Scientific Research on Innovative Areas (JP17H06460 and JP17H06463).
This study was partially supported by JST CREST Grant Number JPMJCR17J4.


\begin{thebibliography}{99}
\bibitem{Birman}
J. S. Birman, \emph{Mapping class groups and their relationship to braid groups}, Comm. Pure Appl. Math. \textbf{22} (1969), 213--238.

\bibitem{Birman-Chillingworth} 
J. S. Birman, D. R. J. Chillingworth, \emph{On the homeotopy group of a non-orientable surface}, Proc. Camb. Philos. Soc. \textbf{71} (1972), 437--448.

\bibitem{Birman-Hilden}
J. S. Birman, H. M. Hilden, \emph{On the mapping class groups of closed surfaces as covering spaces},
Advances in the Theory of Riemann Surfaces (Proc. Conf., Stony Brook, N.Y., 1969), Ann. of Math. Studies 66, Princeton Univ. Press, Princeton, N.J., 81--115, 1971.

\bibitem{BCCS}
E. Bujalance, F. J. Cirre, M. D. E. Conder, B. Szepietowski, \emph{Finite group actions on bordered surfaces of small genus}, J. Pure Appl. Algebra \textbf{214} (2010), no. 12, 2165--2185.

\bibitem{BEMS}
E. Bujalance, J. J. Etayo, E. Mart\'{i}nez, B. Szepietowski, \emph{On the connectedness of the branch loci of non-orientable unbordered Klein surfaces of low genus}, Glasg. Math. J. \textbf{57} (2015), 211--230.

\bibitem{Dehn}
M. Dehn, \emph{Die Gruppe der Abbildungsklassen}, Acta Math. \textbf{69} (1938), 135--206.

\bibitem{Dugger}
D. Dugger, \emph{Involutions on surfaces}, J. Homotopy Relat. Struct. \textbf{14} (2019), 919--992.

\bibitem{Endo-Nagami}
H. Endo, S. Nagami, \emph{Signature of relations in mapping class groups and non-holomorphic Lefschetz fibrations}, Trans. Amer. Math. Soc. \textbf{357} (2005), 3179--3199.

\bibitem{Evans-Kolbe}
M. E. Evans, B. Kolbe, \emph{Isotopic tiling theory for hyperbolic surfaces}, Geom. Dedicata \textbf{212} (2021), 177--204.
\bibitem{Farb-Margalit}
B. Farb, D. Margalit, \emph{A primer on mapping class groups}, Princeton University Press, Princeton, NJ, 2012.

\bibitem{Gurtas}
Y. Gurtas, \emph{Positive Dehn twist expressions for some new involutions in the mapping class group II}, 
arXiv:math.GT/0404311.

\bibitem{Hirose}
S. Hirose, \emph{Presentations of periodic maps on oriented closed surfaces of genera up to 4}, 
Osaka J. of Mathematics \textbf{47} (2010), 385--421.

\bibitem{Hirose2}
S. Hirose, \emph{Generators for the mapping class group of a nonorientable surface}, Kodai Math. J. \textbf{41} (2018), 154--159.

\bibitem{Ishizaka}
M. Ishizaka, \emph{Presentation of hyperelliptic periodic monodromies and splitting families}, Rev. Mat. Complut. \textbf{20} (2007), 483--495.

\bibitem{Karrass-Magnus-Solitar}
A. Karrass, W. Magnus, D. Solitar, \emph{Elements of finite order in groups with a single defining relation}, Comm. Pure Appl. Math. \textbf{13} (1960), 57--66.

\bibitem{Korkmaz1}
M. Korkmaz, \emph{Noncomplex smooth 4-manifolds with Lefschetz fibrations}, Internat. Math. Res. Notices, (2001), 115--128.

\bibitem{Korkmaz2}
M. Korkmaz, \emph{Mapping class groups of nonorientable surfaces}, Geom. Dedicata. \textbf{89} (2002), 109--133.

\bibitem{Lickorish1} 
W. B. R. Lickorish, \emph{Homeomorphisms of non-orientable two-manifolds}, Proc. Camb. Philos. Soc. \textbf{59} (1963), 307--317.

\bibitem{Macbeath} 
A. M. Macbeath, \emph{The classification of non-euclidean plane crystallographic groups}, Canadian J. Math. \textbf{19} (1967), 1192--1205.

\bibitem{Matsumoto}
Y. Matsumoto, \emph{Lefschetz fibrations of genus two --- a topological approach}, 
in Topology and Teichmüller Spaces (Katinkulta, 1995), World Sci. Publ., River Edge, NJ., 123–148, 1996.

\bibitem{Murasugi}
K. Murasugi, \emph{The center of a group with a single defining relation}, Math. Ann. \textbf{155} (1964), 246--251. 

\bibitem{Nielsen}
J. Nielsen, \emph{Die Struktur periodischer Transformationen von Fl\"{a}chen}, Math. -fys. Medd. Danske Vid. Selsk. \textbf{15} (1937) (English transl. in ``Jakob Nielsen collected works, Vol. 2'', 65--102).

\bibitem{Nielsen2}
J. Nielsen, \emph{Abbildungsklassen endlicher Ordnung}, Acta Math. \textbf{75} (1943), 23--115.

\bibitem{Omori}
Genki Omori, \emph{A positive factorization for the balanced superelliptic rotation}, arXiv: 2301.07924.

\bibitem{Stukow}
M. Stukow, \emph{A finite presentation for the mapping class group of a nonorientable surface with Dehn twists and one crosscap slide as generators}, J. Pure Appl. Algebra \textbf{218} (2014), 2226--2239.

\bibitem{Stukow2}
M. Stukow, \emph{A finite presentation for the hyperelliptic mapping class group of a nonorientable surface}, Osaka J. Math. \textbf{52} (2015), 495--514.

\bibitem{Szepietowski1} 
B. Szepietowski, \emph{Crosscap slides and the level 2 mapping class group of a nonorientable surface}, Geom. Dedicata \textbf{160} (2012), 169--183.

\bibitem{Szepietowski2} 
B. Szepietowski, \emph{A finite generating set for the level 2 mapping class group of a nonorientable surface}, Kodai Math. J. \textbf{36} (2013), 1--14.

\bibitem{Wilkie} 
H. C. Wilkie, \emph{On non-Euclidean crystallographic groups}, Math. Zeitschr. \textbf{91} (1966), 87--102.

\end{thebibliography}
\end{document}